\documentclass[12pt, a4paper]{amsart}
\usepackage{amsfonts, amssymb, amsmath}
\usepackage{amsthm}
\usepackage{tabularx}
\usepackage[margin=2.5cm]{geometry}
\usepackage{enumerate}
\usepackage[all]{xy}

\usepackage{xcolor}

\newcolumntype{C}[1]{>{\centering\arraybackslash}p{#1}}
\newcommand{\N}{{\mathbb N}}

\newcommand{\Q}{{\mathbb Q}}
\newcommand{\R}{{\mathbb R}}
\newcommand{\Z}{{\mathbb Z}}
\newcommand{\Pp}{{\mathbb P}}
\newcommand{\A}{{\mathbb A}}
\newcommand{\Cc}{{\mathbb C}}
\newcommand{\Oo}{\mathcal{O}}
\newcommand{\Ss}{\mathcal{S}}
\newcommand{\X}{\mathcal{X}}
\newcommand{\T}{\mathcal{T}}
\newcommand{\cen}{\mbox{Center}}
\newcommand{\se}[2]{\left\lbrace #1 \mbox{ }\vline\mbox{ } #2 \right\rbrace}

\newcommand{\tl}[1]{\tilde{#1}}

\newcommand{\st}{^{\ast}}
\newcommand{\ts}{_{\ast}}
\newcommand{\aast}{(\mbox{\Large{$\ast$}})}
\newcommand{\vc}[1]{\vcenter{\hbox{#1}}}
\newcommand{\rd}[1]{\lfloor #1\rfloor}
\newcommand{\ru}[1]{\lceil #1\rceil}

\newcommand{\gq}[1]{_{\geq #1}}
\newcommand{\sep}[2]{\left\lbrace\begin{matrix} #1 \\ #2 \end{matrix}\right.}
\newcommand{\tc}[2]{\begin{tabular}{c}#1\\\hline #2\end{tabular}}
\newcommand{\ttc}[3]{\begin{tabular}{c}#1\\ #2\\#3\end{tabular}}

\newcommand{\tcc}[2]{\begin{tabular}{c}#1\\ #2\end{tabular}}

\newcommand{\ang}[1]{\langle #1\rangle}
\newcommand{\du}{^{\vee}}
\newcommand{\spc}{\mbox{Spec }}
\newcommand{\gd}{dep_{Gor}}
\newcommand{\ua}[1]{\underset{#1}{\Rightarrow}}
\newcommand{\ual}[1]{\underset{#1}{\Leftarrow}}
\newcommand{\uan}[1]{\stackrel{-}{\underset{#1}{\Rightarrow}}}

\newtheorem{thm}{Theorem}[section]
\newtheorem{pro}[thm]{Proposition}
\newtheorem{cor}[thm]{Corollary}
\newtheorem{lem}[thm]{Lemma}
\theoremstyle{definition}
\newtheorem{rk}[thm]{Remark}
\newtheorem{eg}[thm]{Example}
\newtheorem{cov}[thm]{Convention}
\newtheorem{defn}[thm]{Definition}
\newtheorem{nota}[thm]{Notation}

\newtheorem*{conj}{Conjecture} 

\begin{document}
\title{Minimal resolutions of threefolds}
\author{Hsin-Ku Chen}
\address{School of Mathematics, Korea Institute for Advanced Study, 85 Hoegiro, Dongdaemun-gu, Seoul 02455, Republic of Korea} 
\email{hkchen@kias.re.kr}
\begin{abstract}
	We describe the resolution of singularities of a threefold which has minimal Picard number. We describe the relation between
	this minimal resolution and an arbitrary resolution of singularities.
\end{abstract}
\maketitle

\section{Introduction}

Smooth varieties has certain nice properties, and both algebraic and analytic method can be applied to them. However, when studying problems in birational geometry, particularly those related to the minimal model program, it becomes necessary to investigate varieties with singularities. Fortunately, as Hironaka's famous theorem states, every varieties in characteristic zero has a
resolution of singularity. The existence of resolution of singularities provides a way
to study singular varieties. For instance, by comparing a variety to its resolution of singularities, one can
measure the complexity of a singularity. This is a fundamental technique in higher dimensional birational geometry.\par
Since we want to understand a singularity through its resolution, it is natural to inquire about the difference between two distinct resolutions of singularities. For an algebraic surface $S$, there exists a smooth surface known as the minimal resolution of $S$. This is a resolution of singularities $\bar{S}\rightarrow S$ such that $\rho(\bar{S}/S)$ is minimal. The minimal resolution $\bar{S}$ is unique, and any birational morphism $S'\rightarrow S$ from a smooth surface $S'$ to $S$ factors through $\bar{S}\rightarrow S$.\par
In this paper we want to find a high dimensional analog of the minimal resolution for surfaces.  It is not reasonable to assume the existence of a unique minimal resolution for higher-dimensional singularities. For instance, if $X\dashrightarrow X'$ is a smooth flop over $W$, then $X$ and $X'$ are two different resolutions of singularities for $W$. Since flops are symmetric (at least in dimension three), it appears that $X$ and $X'$ are both minimal. Thus, we need to consider the following issues:
\begin{enumerate}[(1)]
\item To define "minimal resolutions", which ideally should be resolutions of singularities with some minimal geometric invariants.
\item To compare two different minimal resolutions. We need some symmetry between them, so that even if minimal resolutions are not unique,
	it is not necessary to distinguish them.
\item To compare a minimal resolution with an arbitrary resolution of singularities.
\end{enumerate} 
Inspired by the two-dimensional case, it is natural to consider resolutions of singularities with minimal Picard number
(we will call them P-minimal resolutions, see Section \ref{spmr} for more precise definition). We know that a fixed singularity may
have more than one P-minimal resolutions, and two different P-minimal resolutions can differ by a smooth flop. It is also possible that two
different P-minimal resolutions are ``differ by a singular flop'': consider $X\dashrightarrow X'$ be a possibly singular flop over $W$. Let
$\tl{X}\rightarrow X$ and $\tl{X'}\rightarrow X'$ be P-minimal resolutions of $X$ and $X'$ respectively. Then because of the symmetry
between flops, one may expect that $\tl{X}$ and $\tl{X'}$ are two different P-minimal resolutions of $W$. We call the birational map
$\tl{X}\dashrightarrow \tl{X'}$ a P-desingularization of the flop $X\dashrightarrow X'$ (a precise definition can be found in Section
\ref{spmr}). If we consider P-desingularizations of flops as elementary birational maps, then in dimension three, P-minimal resolutions
have nice properties.

\begin{thm}\label{thm1}
	Assume that $X$ is a projective threefold over the complex number and $\tl{X}_1$, $\tl{X}_2$ are two different P-minimal
	resolutions of $X$. Then $\tl{X}_1$ and $\tl{X}_2$ are connected by P-desingularizations of terminal and $\Q$-factorial flops.\par
	Moreover, if $X$ has terminal and $\Q$-factorial singularities, then the birational map $\tl{X}_1\dashrightarrow \tl{X}_2$ has an
	$\Omega$-type factorization.
\end{thm}
Please see Section \ref{sflp} for the definition of $\Omega$-type factorizations.
\begin{thm}\label{thm2}
	Assume that $X$ is a projective threefold over the complex number and $W\rightarrow X$ is a birational morphism from a smooth threefold
	$W$ to $X$. Then for any P-minimal resolution $\tl{X}$ of $X$ one has a factorization
	\[W=\tl{X}_k\dashrightarrow ... \dashrightarrow \tl{X}_1\dashrightarrow \tl{X}_0=\tl{X}\]
	such that $\tl{X}_{i+1}\dashrightarrow \tl{X}_i$ is either a smooth blow-down, or a P-desingularization of a terminal $\Q$-factorial flop.
\end{thm}

Since three-dimensional terminal flops are topologically symmetric, some topological invariants like Betti numbers won't change after
P-desingularizations of terminal flops. Hence it is easy to see that P-minimal resolutions are the resolution of singularities with
minimal Betti numbers.
\begin{cor}\label{cor}
	Assume that $X$ is a projective threefold over the complex number and $W\rightarrow X$ is a birational morphism from a smooth threefold
	$W$ to $X$. Then for any P-minimal resolution $\tl{X}$ of $X$ one has that $b_i(\tl{X})\leq b_i(W)$ for all $i=0$, ..., $6$. 
\end{cor}

Although in dimension three P-minimal resolutions behaves well, for singularities of dimension greater than three, P-minimal resolutions may not be truly "minimal". An simple example is a smooth flip. If $X\dashrightarrow X'$ is a smooth flip over $W$, then both
$X$ and $X'$ are P-minimal resolutions of $W$, but $X'$ is better than $X$. Now the only known smooth flips are standard flips
\cite[Section 11.3]{hu}, and if $X\dashrightarrow X'$ is a standard flip, then it is easy to see that $b_i(X)\geq b_i(X')$ for all $i$
and the inequality is strict for some $i$. Thus the resolution of singularities with minimal Betti numbers may be the right minimal resolution
for higher dimensional singularities. Because of Corollary \ref{cor}, in dimension three P-minimal resolutions are exactly those smooth
resolutions which have minimal Betti numbers. Therefore, this new definition of minimal resolutions is compatible with our three-dimensional theorems.\par
We now return to the proof of our main theorems. Let $X$ be a threefold and $W\rightarrow X$ be a resolution of singularities. One can
run $K_W$-MMP over $X$ \[ W=X_0\dashrightarrow X_1\dashrightarrow...\dashrightarrow X_k=X.\]
Let $\tl{X}_i$ be a P-minimal resolution of $X_i$, then $\tl{X}_0=W$ and it is easy to see that $\tl{X}_k$ is also a P-minimal resolution
of $X$. Thus our main theorems can be easily proved if we know the relation between $\tl{X}_i$ and $\tl{X}_{i+1}$. Since $X_i$ has only terminal
and $\Q$-factorial singularities, studying P-minimal resolutions of $X_i$ becomes simpler.\par
In \cite{c} J. A. Chen introduces feasible resolutions for terminal threefolds, which is a resolution of singularities consisting of a sequence of divisorial contractions to points with minimal discrepancies (see Section \ref{stt} for more detail). Given a terminal threefold $X$ and
a feasible resolution $\bar{X}$ of $X$, one can define the generalized depth of $X$ to be the integer $\rho(\bar{X}/X)$.
The generalized depth is a very useful geometric invariant of a terminal threefold. In our application, the crucial factor is that
one can test whether a resolution of singularities $W\rightarrow X$ is a feasible resolution or not by comparing $\rho(W/X)$ and
the generalized depth of $X$. We need to understand that how do generalized depths change after steps of the minimal model program. After that,
we can prove that for terminal and $\Q$-factorial threefolds, P-minimal resolutions and feasible resolutions coincide.\par 
Now we only need to figure out the following two things: how do generalized depths change after a step of MMP, and
how do P-minimal resolutions change after a step of MMP. To answer those questions we have to factorize a step of MMP into more simpler
birational maps. In \cite{ch} J. A. Chen and Hacon proved that three-dimensional terminal flips and divisorial contractions to curves can
be factorized into a composition of (inverses of) divisorial contractions and flops. In this paper we construct a similar factorization for
divisorial contractions to points. After knowing the factorization, we are able to answer the two questions above and prove our main theorems.\par
In addition to the above, we introduce the notion of Gorenstein depth for terminal threefolds. The basic idea is as follows: given a sequence of steps of MMP of terminal threefolds
\[ X_0\dashrightarrow X_1\dashrightarrow ...\dashrightarrow X_k,\]
one can show that the generalized depth of $X_k$ is bounded above by the integer $k$ and the generalized depth of $X_0$. That is to say, the number of steps of MMP bounds the singularities on the minimal model. One may ask that, is there an opposite
bound? Namely if we know the singularities of the minimal model $X_k$, can we bound singularities of $X_0$? In this paper we define the
Gorenstein depth of terminal threefolds, which roughly speaking measures only Gorenstein singularities. One can show that the Gorenstein
depth is always non-decreasing when running three-dimensional terminal MMP. Our result on Gorenstein depth will have important applications in \cite{me}.\par
This paper is structured as follows: Section \ref{sper} is a preliminary section. In section \ref{sfac} we develop some useful tools to construct
relations between divisorial contractions to points. Those tools, as well as the explicit classification of divisorial contractions,
will be used in Section \ref{slk} to construct links of different divisorial contractions to points. In Section \ref{sdep} we prove the property
of the generalized depth. The construction of diagrams in Theorem \ref{thm1} will be given in Section \ref{sflp}. All our main theorems will be
proved in Section \ref{spmr}. In the last section we discuss possible higher dimensional generalization of the notion of minimal resolutions,
and possible application of our main theorems.\par 
I thanks Jungkai Alfred Chen for his helpful comments. The author is support by KIAS individual Grant MG088901.

\section{Preliminaries}\label{sper}
\subsection{Notations and conventions}\label{snc}
In this paper we only consider projective varieties over the complex number.\par
For a divisorial contraction, we mean a birational morphism $Y\rightarrow X$ which contracts an irreducible divisor $E$ to a locus
of codimension at least two, such that $K_Y$ is anti-ample over $X$. We will denote $v_E$ the valuation corresponds to $E$.\par
Let $G$ be a cyclic group of order $r$ generated by $\tau$. For any $\Z$-valued $n$-tuple $(a_1,...,a_n)$
one can define a $G$-action on $\A^n_{(x_1,...,x_n)}$ by $\tau(x_i)=\xi^{a_i}x_i$, where $\xi=e^{\frac{2\pi i}{r}}$.
We will denote the quotient space $\A^n/G$ by $\A^n_{(x_1,...,x_n)}/\frac{1}{r}(a_1,...,a_n)$.\par
We say that $w$ is a weight on $W/G=\A^n_{(x_1,...,x_n)}/G$ defined by $w(x_1,...,x_n)=\frac{1}{r}(b_1,...,b_n)$ if $w$ is a map
$\Oo_W\rightarrow \frac{1}{r}\Z_{\geq 0}$ such that
\[w\left(\sum_{(i_1,...,i_n)\in\Z^n_{\geq0}}c_{(i_1,...,i_n)}x_1^{i_1}...x_n^{i_n}\right)=
	\min\se{\frac{1}{r}(b_1i_1+...+b_ni_n)}{c_{(i_1,...,i_n)}\neq0}.\]\par
Assume that $\phi:X\dashrightarrow Y$ is a birational map. Let $U\subset X$ be the largest open set such that $\phi|_U$ is an identity map
and $Z\subset X$ be an irreducible subset such that $Z$ intersects $U$ non-trivially. We will denote $Z_Y$ the closure
of $\phi|_U(Z|_U)$.

\subsection{Weighted blow-ups}\label{swu}
Let $W\cong \A^n_{(x_1,...,x_n)}/\frac{1}{r}(a_1,...,a_n)$ be a cyclic-quotient singularity.
There is an elementary way to construct a birational morphism $W'\rightarrow W$, so called the weighted blow-up, defined as follows.\par
We write everything in the language of toric varieties. Let $N$ be the lattice $\ang{e_1,...,e_n,v}_{\Z}$,
where $e_1$, ..., $e_n$ is the standard basic of $\R^n$ and $v=\frac{1}{r}(a_1,...,a_n)$. Let $\sigma=\ang{e_1,...,e_n}_{\R\gq0}$.
We have $W\cong \spc\Cc [N\du\cap \sigma\du]$.\par
Let $w=\frac{1}{r}(b_1,...,b_n)$ be a vector such that $b_i=\lambda a_i+k_ir$ for $\lambda\in \N$ and $k_i\in\Z$, with $b_i\neq 0$.
We define a weighted blow-up of $W$ with weight $w$ to be the toric variety defined by the fan consists of those cones
\[\sigma_i=\ang{e_1,...,e_{i-1},w,e_{i+1},...,e_n}.\]
Let $U_i$ be the toric variety defined by the cone $\sigma_i$ and lattice $N$, namely \[U_i=\spc\Cc [N\du\cap \sigma_i\du].\]
\begin{lem}\label{wbup}
	One has that \[ U_i\cong \A^n/\ang{\tau,\tau'}\]	
	where $\tau$ is the action given by \[ x_i\mapsto \xi_{b_i}^{-r}x_i,\quad x_j\mapsto \xi_{b_i}^{b_j}x_j,\mbox{ }j\neq i\]
	and $\tau'$ is the action given by \[ x_i\mapsto \xi_{b_i}^{a_i}x_i,\quad x_j\mapsto \xi_{rb_i}^{a_jb_i-a_ib_j}x_j,\mbox{ }j\neq i.\]
	Here $\xi_k$ denotes a $k$-th roots of unity for any positive integer $k$.\par
	In particular, the exceptional divisor of $\bar{W}\rightarrow W$ is $\Pp(b_1,...,b_n)/G$ where
	$G$ is a cyclic group of order $m$ where $m$ is an integer divides $\lambda$.
\end{lem}
\begin{proof}
	Let $T_i$ be a linear transformation such that $T_ie_j=e_j$ if $j\neq i$ and $T_iw=e_i$. One can see that
	\[T_ie_i=\frac{r}{b_i}(e_i-\sum_{j\neq i}\frac{b_j}{r}e_j)\]
	and \[ T_iv=\sum_{j\neq i}\frac{a_j}{r}e_j+\frac{a_i}{r}\frac{r}{b_i}(e_i-\sum_{j\neq i}\frac{b_j}{r}e_j)
		=\frac{a_i}{b_i}e_i+\sum_{j\neq i} \frac{a_jb_i-a_ib_j}{rb_j}e_j.\]
	Under this linear transformation $\sigma_i$ becomes the standard cone $\ang{e_1,...,e_n}_{\R\gq0}$. 
	Note that
	\begin{align*}
		k_iT_ie_i+\lambda T_iv&=\frac{k_ir+\lambda a_i}{b_i}e_i+\sum_{j\neq i}\frac{\lambda(a_jb_i-a_ib_j)-k_ib_jr}{rb_i}e_j \\
		&=e_i+\sum_{j\neq i}\frac{\lambda a_jb_i-b_ib_j}{rb_i}e_j=e_i-\sum_{j\neq i}k_je_j.
	\end{align*}
	Hence $e_i\in T_i N$ and $T_iN=\ang{e_1,...,e_n,T_ie_i,T_iv}_{\Z}$. Now $T_ie_i$ corresponds to the action $\tau$ and $T_iv$
	corresponds to the action $\tau'$. This means that $U_i\cong \A^n/\ang{\tau,\tau'}$.\par
	The computation above shows that $\tau^{k_i}={\tau'}^{\lambda}$. If we glue $(x_i=0)\subset\A^n/\ang{\tau}$ together then we get a weight
	projective space $\Pp(b_1,...,b_n)$. The relation $\tau^{k_i}={\tau'}^{\lambda}$ implies that $(x_i=0)\subset U_i$ can be viewed as
	$\Pp(b_1,...,b_n)$ quotient by a cyclic group of order $m$ for some factor $m$ of $\lambda$.
\end{proof}
\begin{cor}\label{wbcc}
	Let $x_1$, ..., $x_n$ be the local coordinates of $W$ and $y_1$, ..., $y_n$ be the local coordinates of $U_i$.
	The change of coordinates of $U_i\rightarrow W$ are given by $x_j=y_jy_i^{\frac{b_j}{r}}$ and $x_i=y_i^{\frac{b_i}{r}}$.
\end{cor}
\begin{proof}
	The change of coordinate is defined by $T_i^t$, where $T_i$ is defined as in Lemma \ref{wbup}.
\end{proof}
\begin{cor}
	Assume that \[S=(f_1(x_1,...,x_n)=...=f_k(x_1,...,x_n)=0)\subset W\] is a complete intersection
	and $S'$ is the proper transform of $S$ on $W'$. Assume that the exceptional locus $E$ of $S'\rightarrow S$ is irreducible and reduced.
	Then \[a(E,S)=\frac{b_1+...+b_n}{r}-\sum_{i=1}^kw(f_k)-1.\]
\end{cor}
\begin{proof}
	Assume first that $k=0$. Denote $\phi:W'\rightarrow W$. Then on $U_i$ we have
	\[\phi\st dx_1\wedge...\wedge dx_n=\frac{b_i}{r}y_i^{\frac{b_i}{r}-1}
		\left(\prod_{j\neq i} y_i^{\frac{b_j}{r}}\right) dy_1\wedge...\wedge dy_n,\]
	hence $K_{W'}=\phi\st K_W+(\frac{b_1+...+b_n}{r}-1)E$.\par
	Now the statement follows from the adjunction formula.
\end{proof}
\begin{cor}
	Let $F=exc(W'\rightarrow W)$. Then \[F^n=\frac{(-1)^{n-1}r^{n-1}}{b_1...b_nm}.\] Here $m$ is the integer in Lemma \ref{wbup}.
\end{cor}
\begin{proof}
	From the change of coordinate formula in Corollary \ref{wbcc} one can see that $F|_F=\Oo_{\Pp(b_1,...,b_n)}(-r)$. It follows that
	\[F^n=(F|_F)^{n-1}=\frac{(-1)^{n-1}r^{n-1}}{b_1...b_nm}.\]
\end{proof}

\begin{defn}
	Let $\phi_i:U_i\rightarrow W$ be the morphism in Corollary \ref{wbcc}. For any $G$-semi-invariant function $u\in\Oo_W$
	we can define the strict transform of $u$ on $U_i$ by $(\phi_i^{-1})\ts(u)=\phi\st(u)/y_i^{w(u)}$.
\end{defn}

In this paper we are going to consider terminal threefolds which are embedding into cyclic quotient of $\A^4$ or $\A^5$
\[ X\hookrightarrow \A^4_{(x,y,z,u)}/\frac{1}{r}(a,b,c,d)\quad\mbox{or}\quad X\hookrightarrow \A^5_{(x,y,z,u,t)}/\frac{1}{r}(a,b,c,d,e).\]
We say that $Y\rightarrow X$ is a weighted blow-up with weight $w$ if $Y$ is the proper transform of $X$ inside the weighted blow-up
of $\A^4_{(x,y,z,u)}/\frac{1}{r}(a,b,c,d)$ or $\A^5_{(x,y,z,u,t)}/\frac{1}{r}(a,b,c,d,e)$ with weight $w$.
\begin{nota}\label{cov}
	Assume that $X$ is of the above form and let $Y\rightarrow X$ be a weighted blow-up.
	The notation $U_x$, $U_y$, $U_z$, $U_u$ and $U_t$ will stand for $U_1$, ..., $U_5$ in Lemma \ref{wbup}.
\end{nota}
\begin{nota}\label{wcov}
	Assume that $w$ is a weight on $\A^n_{(x_1,...,x_n)}$ determined by $w(x_1,...,x_n)=(a_1,...,a_n)$ and
	\[f(x_1,...,x_n)=\sum_{(i_1,...,i_n)\in\Z_{\geq 0}^n}\lambda_{i_1,...,i_n}x_1^{i_1}...x_n^{i_n}\] is a
	regular function on $\A^n$. We denote \[f_w=\sum_{a_1i_1+...+a_ni_n=w(f)}\lambda_{i_1,...,i_n}x_1^{i_1}...x_n^{i_n}.\]
\end{nota}

\subsection{Terminal threefolds}
\subsubsection{Local classification}
The local classification of terminal threefolds were done by Reid \cite{re2} for Gorenstein cases and Mori \cite{mo3} for non-Gorenstein cases.
\begin{defn}
	A compound Du Val point $P\in X$ is a hypersurface singularity locally analytically defined by $f(x,y,z)+tg(x,y,z,t)=0$, where
	$f(x,y,z)$ defines a Du Val singularity.
\end{defn}
\begin{thm}[\cite{re2}, Theorem 1.1]
	Let $P\in X$ be a point of threefold. Then $P\in X$ is an isolated compound Du Val point if and only if $P\in X$ is terminal 
	and $K_X$ is Cartier near $P$.
\end{thm}

\begin{thm}[\cite{mo3}, cf. \cite{re} Theorem 6.1]
	Let $P\in X$ be a germ of three-dimensional terminal singularity such that $K_X$ has Cartier index $r>1$. Then
	\[ X\cong (f(x,y,z,u)=0)\subset \A^4_{(x,y,z,u)}/\frac{1}{r}(a_1,...,a_4)\]
	such that $f$, $r$ and $a_i$ are given by Table \ref{ta3}.
\begin{table}
\begin{tabular}{|c|c|c|c|c|}\hline
Type & $f(x,y,z,u)$ & $r$ & $a_i$ & condition \\\hline
$cA/r$ & $xy+g(z^r,u)$ & any & $(\alpha,-\alpha,1,r)$ & \tcc{$g\subset m_P^2$}{$\alpha$ and $r$ are coprime}\\\hline
$cAx/4$ & \tcc{$xy+z^2+g(u)$}{$x^2+z^2+g(y,u)$} & $4$ & $(1,1,3,2)$ & $g\in m_P^3$ \\\hline
$cAx/2$ & $xy+g(z,u)$ & $2$ & $(0,1,1,1)$ & $g\in m_P^4$ \\\hline
$cD/3$ & \ttc{$x^2+y^3+z^3+u^3$}{$x^2+y^3+z^2u+yg(z,u)+h(z,u)$}{$x^2+y^3+z^3+yg(z,u)+h(z,u)$} & $3$ & $(0,2,1,1)$ &
	\tcc{$g\in m_P^4$}{$h\in m_P^6$} \\\hline
$cD/2$ & \ttc{$x^2+y^3+yzu+g(z,u)$}{$x^2+yzu+y^n+g(z,u)$}{$x^2+yz^2+y^n+g(z,u)$} & $2$ & $(1,0,1,1)$ &
	$g\in m_P^4$, \ttc{}{$n\geq4$}{$n\geq3$} \\\hline
$cE/2$ & $x^2+y^3+yg(z,u)+h(z,u)$ & $2$ & $(1,0,1,1)$ & \tcc{$g$, $h\in m_P^4$}{$h_4\neq 0$}\\\hline
\end{tabular}
\caption{Classification of terminal threefolds} \label{ta3}
\end{table}
\end{thm}
Assume that $P\in X$ is a three-dimensional terminal singularity. Then there exists a section $H\in|-K_X|$ which has Du Val singularities
(so-called a general elephant). Please see  \cite[(6.4)]{re} for details.
\subsubsection{Classification of divisorial contractions to points}
Divisorial contractions to points between terminal threefolds are well-classified by Kawamata, Hayakawa, Kawakita and Yamamoto.
\begin{thm}
	Assume that $Y\rightarrow X$ is a divisorial contraction to a point between terminal threefolds, then there exists an embedding
	$X\hookrightarrow W$ with $W=\A^4_{(x,y,z,u)}$ or $\A^5_{(x,y,z,u,t)}$ and a weight $w(x,y,z,u)=\frac{1}{r}(a_1,...,a_4)$
	or $w(x,y,z,u,t)=\frac{1}{r}(a_1,...,a_5)$, such that $Y\rightarrow X$ is a weighted blow-up with respect to $w$.\par
	The defining equation of $X\subset W$ and the weight is given in Table \ref{taA} to Table \ref{taE4}.
\end{thm}
For reader's convenience we put those tables in Section \ref{slk}. In those tables we use the following notation: for an non-negative
integer $m$ the notation $g\gq m$ means a function $g\in\Oo_W$ such that $w(g)=m$. The notation $p_m$ means a function $p\in\Oo_W$
which is homogeneous of weight $m$ with respect to the weight $w$.\par
The reference of each cases in  Table \ref{taA}, ..., Table \ref{taE4} is the follows:
\begin{itemize}
\item Case A1 is \cite[Theorem 1.2 (i)]{k2}. Case A2 is \cite[Theorem 2.6]{y}.
\item Case Ax1-Ax4 are \cite{h1} Theorem 7.4, Theorem 7.9, Theorem 8.4 and Theorem 8.8 respectively.
\item Case D1-D5 are Theorem 2.1-2.5 in \cite{h4} respectively.
\item Case D6 and Case D7 are \cite[Theorem 1.2 (ii)]{k2}. Case D8-D11 is \cite{y} Theorem 2.1-2.4. Case D12 is \cite[Theorem 2.7]{y}.
\item Case D13 is \cite[Theorem 9.9, Theorem 9.14, Theorem 9.20]{h1}. Case D14 is \cite[Theorem 9.25]{h1}.
\item Case D16 is \cite[Proposition 4.4]{h2}. Case D17 is \cite[Proposition 4.7, 4.12]{h2} Case D18 is \cite[Proposition 4.9]{h2}.
	Case D18 is \cite[Proposition 5.4]{h2}. Case D19 is 
	\cite[Proposition 5.8, Proposition 5.13, Proposition 5.22, Proposition 5.28, Proposition 5.35]{h2}.
	Case D20 is \cite[Proposition 5.18, Proposition 5.25]{h2}.
	Case D21 is \cite[Proposition 5.16, Proposition 5.32]{h2}. Case D22 is \cite[Proposition 5.9, Proposition 5.36]{h2}.
\item Case D23 and D24 is \cite[Theorem 1.2(ii)]{k2} and \cite[Theorem 1.1 (iii)]{h3}. Case D25-D28 is \cite{h3}
	Theorem 1.1 (i), (i'), (ii'), (iii), (ii) respectively. Case D29 is \cite[Theorem 2]{k3}.
\item Case E1-E18 is \cite[Theorem 1]{h5}.
\item Case E19-E21 is \cite{y} Theorem 2.5, Theorem 2.9, Theorem 2.10 respectively.
\item Case E22 is \cite[Theorem 10.11, Theorem 10.17, Theorem 10.22, Theorem 10.28, Theorem 10.33, Theorem 10.41]{h1}.
	Case E23 is \cite[Theorem 10.33, Theorem 10.47]{h1}. Case E24 is \cite[Theorem 10.54, Theorem 10.61]{h1}.
	Case E25 is \cite[Theorem 10.67]{h1}. Case E26 is \cite[Theorem 1.2]{h3}.
\end{itemize}
\subsubsection{The depth}\label{stt}

\begin{defn}
	Let $Y\rightarrow X$ be a divisorial contraction contracts a divisor $E$ to a point $P$. We say that $Y\rightarrow X$ is a 
	\emph{$w$-morphism} if $a(X,E)=\frac{1}{r_P}$, where $r_P$ is the Cartier index of $K_X$ near $P$.
\end{defn}
\begin{defn}
	The depth of a terminal singularity $P\in X$, $dep(P\in X)$, is the minimal length of the sequence
	\[ X_m\rightarrow X_{m-1}\rightarrow \cdots\rightarrow X_1\rightarrow X_0=X,\]
	such that $X_m$ is Gorenstein and $X_i\rightarrow X_{i-1}$ is a $w$-morphism for all $1\leq i\leq m$.\par
	The generalize depth of a terminal singularity $P\in X$, $gdep(P\in X)$, is the minimal length of the sequence
	\[ X_n\rightarrow X_{n-1}\rightarrow \cdots\rightarrow X_1\rightarrow X_0=X,\]
	such that $X_n$ is smooth and $X_i\rightarrow X_{i-1}$ is a $w$-morphism for all $1\leq i\leq n$. The variety $X_n$
	is called a \emph{feasible resolution} of $P\in X$.\par
	The Gorenstein depth of a terminal singularity $P\in X$, $\gd(P\in X)$, is defined by $gdep(P\in X)-dep(P\in X)$.\par
	For a terminal threefold we can define \[dep(X)=\sum_P dep(P\in X),\]
	\[gdep(X)=\sum_P gdep(P\in X)\]
	and \[ \gd(X)=\sum_P \gd(P\in X).\]
\end{defn}
\begin{rk}
	In the above definition, the existence of a sequence \[ X_m\rightarrow X_{m-1}\rightarrow \cdots\rightarrow X_1\rightarrow X_0=X,\]
	such that $X_m$ is Gorenstein follows from \cite[Theorem 1.2]{h2}. The existence of a sequence
	\[ X_n\rightarrow X_{n-1}\rightarrow \cdots\rightarrow X_1\rightarrow X_0=X,\] such that $X_n$ is smooth follows from \cite[Theorem 2]{c}.
\end{rk}
\begin{defn}
	Assume that $Y\rightarrow X$ is a $w$-morphism such that $gdep(Y)=gdep(X)-1$. Then we say that $Y\rightarrow X$ is a
	\emph{strict $w$-morphism}.
\end{defn}

\begin{lem}\label{val}
	Assume that $Y\rightarrow X$ is a divisorial contraction which is obtained by weighted blowing-up a weight
	$w(x_1,...,x_n)=\frac{1}{r}(a_1,...,a_n)$ with respect to an embedding $X\hookrightarrow \A^n_{(x_1,...,x_n)}/G$,
	where $G$ is a cyclic group of index $r$. Assume that $E$ is an exceptional divisor over $X$ and
	$v_E(x_1,...,x_n)=\frac{1}{r}(b_1,...,b_n)$, then $\cen_YE\cap U_i$ non-trivially if and only if $\frac{b_i}{a_i}\leq\frac{b_j}{a_j}$
	for all $1\leq j\leq n$. Here $U_1$, ..., $U_n$ denotes the canonical affine chart of the weighted blow-up on $Y$.
\end{lem}
\begin{proof}
	Let $y_1$, ..., $y_n$ be the local coordinates of $U_i$, then we have the following change of coordinates formula:
	\[ x_i=y_i^{\frac{a_i}{r}},\quad x_j=y_i^{\frac{a_j}{r}}y_j\mbox{ if }i\neq j.\]
	One can see that \[ v_E(y_i)=\frac{b_i}{a_i},\quad v_E(y_j)=\frac{b_i}{r}-\frac{a_jb_i}{ra_i}\mbox{ if }i\neq j.\]
	We know that $\cen_EY$ intersects $U_i$ non-trivially if and only if $\frac{b_i}{r}-\frac{a_jb_i}{ra_i}\geq 0$ for all $j\neq i$,
	or equivalently, $\frac{b_j}{a_j}\geq \frac{b_i}{a_i}$ for all $j\neq i$.
\end{proof}
\begin{cor}\label{wef}
	Assume that $Y\rightarrow X$ and $Y_1\rightarrow X$ are two $w$-morphisms and let $E$ and $F$ be the exceptional divisors of
	$Y\rightarrow X$ and $Y_1\rightarrow X$ respectively. Then there exists $u\in\Oo_X$ such that $v_E(u)<v_F(u)$.
\end{cor}
\begin{proof}
	Let $X\rightarrow \A^n_{(x_1,...,x_n)}/G$ be the embedding so that $Y_1\rightarrow X$ can be obtained by the weighted blow-up with
	with respect to the embedding. We may assume that $(x_n=0)$ defines a Du Val section. Then \[v_E(x_n)=a(E,X)=a(F,X)=v_F(x_n).\]
	It follows that $a_n=b_n=1$ where $(a_1,...,a_n)$ and $(b_1,...,b_n)$ are integers in Lemma \ref{val}.
	Now since $a(F,X)=a(E,X)$, one has that $\cen_{Y_1}E$ is a non-Gorenstein point (if $Y_1$ is generically Gorenstein along $\cen_{Y_1}E$
	then an easy computation shows that $a(E,X)>a(F,X)$). It follows that $\cen_{Y_1}E\cap U_n$ is empty
	since $a_n=1$ implies that $U_n$ is Gorenstein. Thus by Lemma \ref{val} we know that there exists $j$ so that $\frac{b_j}{a_j}<1$.
	Hence $v_E(u)>v_F(u)$ if $u=x_j$. 
\end{proof}

\subsection{Chen-Hacon factorizations}

We have the following factorization of steps of three-dimensional terminal MMP by J. A. Chen and Hacon.
\begin{thm}[\cite{ch} Theorem 3.3]\label{chf}
	Assume that either $X\dashrightarrow X'$ be a flip over $V$, or $X\rightarrow V$ is a divisorial contraction to a curve and $X$ is
	not Gorenstein over $U$. Then there exists a diagram
	\[\vc{\xymatrix{Y_1\ar[d] \ar@{-->}[r] & \cdots \ar@{-->}[r] & Y_k\ar[d]\\
		X\ar[rd] &  & X'\ar[ld] \\ & V &	 }}\]
	such that $Y_1\rightarrow X$ is a $w$-morphism, $Y_k\rightarrow X'$ is a divisorial contraction, $Y_1\dashrightarrow Y_2$ is a
	flip or a flop and $Y_i\dashrightarrow Y_{i+1}$ is a flip for $i>1$. If $X\rightarrow V$ is divisorial then $Y_k\rightarrow X'$
	is a divisorial contraction to a curve and $X'\rightarrow V$ is a divisorial contraction to a point.
\end{thm}
\begin{rk}\label{chfr}
	Notation as in the above theorem. From the construction of the diagram we know that
	\begin{enumerate}[(1)]
	\item Let $C_{Y_1}$ be a flipping/flopping curve of $Y_1\dashrightarrow Y_2$. Then $C_X$ is a flipping curve of $X\dashrightarrow X'$.
	\item Assume that the exceptional locus of $X\rightarrow V$ contains a non-Gorenstein point $P$ which is not a $cA/r$ or a $cAx/r$ point,
		then $Y_1\rightarrow X$ can be chosen to be any $w$-morphism over $P$. This statement follows from the proof of \cite[Theorem 3.1]{ch}.
	\end{enumerate}
\end{rk}
We have the following properties of the depth \cite[Proposition 2.15, 3.8, 3.9]{ch}:
\begin{lem}\label{chfp}
	Let $X$ be a terminal threefold.
	\begin{enumerate}
	\item If $Y\rightarrow X$ is a divisorial contraction to a point, then $dep(Y)\geq dep(X)-1$.
	\item If $Y\rightarrow X$ is a divisorial contraction to a curve, then $dep(Y)\geq dep(X)$.
	\item If $X\dashrightarrow X'$ is a flip, then $dep(X)>dep(X')$.
	\end{enumerate}
\end{lem}
\subsection{The negativity lemma}

We have the negativity lemma for flips.
\begin{lem}\label{ntl}
	Assume that $X\dashrightarrow X'$ is a $K_X+D$-flip, then for all exceptional divisor $E$ one has that $a(E,X,D)\leq a(E,X',D_{X'})$.
	The inequality is strict if $\cen_XE$ is contained in the flipping locus.
\end{lem}
\begin{proof}
	It is a special case of \cite[Lemma 3.38]{km2}.
\end{proof}

What we really need is the following corollary of the negatively lemma.
\begin{cor}\label{ntl2}
	Assume that $X\dashrightarrow X'$ is a $K_X+D$-flip and $C\subset X$ is an irreducible curve which is not a flipping curve.
	Then $(K_X+D).C\geq(K_{X'}+D_{X'}).C_{X'}$. The inequality is strict if $C$ intersects the flipping locus non-trivially.
\end{cor}
\begin{proof}
	Let $\xymatrix{X & W\ar[l]_{\phi} \ar[r]^{\phi'} & X'}$ be a common resolution such that $C$ is not contained in the indeterminacy
	locus of $\phi$. Then Lemma \ref{ntl} implies that $F=\phi\st(K_X+D)-{\phi'}\st(K_{X'}+D_{X'})$ is an effective divisor and is
	supported on exactly those exceptional divisors whose centers on $X$ are contained in the flipping locus. Hence
	\[ (K_X+D).C-(K_{X'}+D_{X'}).C_{X'}=(\phi\st(K_X+D)-{\phi'}\st(K_{X'}+D_{X'})).C_W=F.C_W\geq 0.\]
	The last inequality is strict if and only if $C_W$ intersects $F$ non-trivially, or equivalently, $C$ intersects the flipping locus
	non-trivially.
\end{proof}

\section{Factorize divisorial contractions to points}\label{sfac}
Let $Y\rightarrow X$ be a divisorial contraction contracts a divisor $E$ to a point. We construct the diagram
\[ \vc{\xymatrix{ Z_1\ar@{-->}[r] \ar[d] & ... \ar@{-->}[r] & Z_k \ar[d] \\ Y \ar[rd] & & Y_1\ar[ld] \\ & X & }}\]
as follows: Let $Z_1\rightarrow Y$ be a $w$-morphism and let $H\in|-K_X|$ be a Du Val section. By \cite[Lemma 2.7 (ii)]{ch} one has that
$a(E,X,H)=0$. We run $(K_{Z_1}+H_{Z_1}+\epsilon E_{Z_1})$-MMP over $X$ for some $\epsilon>0$ such that $(Z_1,H_{Z_1}+E_{Z_1})$ is klt.
Notice that a general curve inside $E_{Z_1}$ intersects the pair negatively, and a general curve in $F$ intersects the pair positively
where $F=exc(Z_1\rightarrow Y)$. Thus after finitely many $K_{Z_1}+H_{Z_1}+E_{Z_1}$-flips $Z_1\dashrightarrow ...\dashrightarrow Z_k$,
the MMP ends with a divisorial contraction $Z_k\rightarrow Y_1$ which contracts $E_{Z_k}$, and $Y_1\rightarrow X$ is a divisorial
contraction contracts $F_{Y_1}$.

\begin{lem}\label{kan}
	Keep the above notation. Assume that $K_{Z_1}$ is anti-nef over $X$ and $E_{Z_1}$ is not covered by $K_{Z_1}$-trivial curves.
	Then $Z_i\dashrightarrow Z_{i+1}$ is a $K_{Z_i}$-flip or flop for all $i$ and $Z_k\rightarrow Y_1$ is a $K_{Z_k}$-divisorial contraction.
	In particular, $Y_1$ and $Z_2$, ..., $Z_k$ are all terminal.
\end{lem}
\begin{proof}
	Assume first that $k=1$. If $Z_1\rightarrow Y_1$ is a $K_{Z_1}$-negative contraction then we have done. Otherwise
	$Z_1\rightarrow Y_1$ is a $K_{Z_1}$-trivial contraction. In this case $E_{Z_1}$ is covered by $K_{X_1}$-trivial curves,
	which contradicts to our assumption.\par
	Now assume that $k>1$. We know that $Z_1\dashrightarrow Z_2$ is a $K_{Z_1}$-flip or flop. Notice also that a general curve on
	$E_{Z_1}$ is $K_{Z_1}$-negative, hence a general curve on $E_{Z_2}$ is $K_{Z_2}$-negative by Corollary \ref{ntl2}.
	Now the relative effective cone $NE(Z_2/X)$ is a two-dimensional cone. One of a boundary of $NE(Z_2/X)$ corresponds to the
	flipped/flopped curve of $Z_1\dashrightarrow Z_2$, and is $K_{Z_2}$-non-negative. Since there is a $K_{Z_2}$-negative curve,
	we know that the other boundary of $NE(Z_2/X)$ is $K_{Z_2}$-negative. So if $k=2$ $Z_2\rightarrow Y_1$ is a $K_{Z_2}$-divisorial
	contraction, and for $k>2$ $Z_2\dashrightarrow Z_3$ is a $K_{Z_2}$-flip. One can prove the statement by
	repeating this argument $k-2$ times.
\end{proof}
We are going to find the sufficient conditions for the assumptions of Lemma \ref{kan}. Our final result is Lemma \ref{knef}
and Lemma \ref{kneff}.\par
Let $X\hookrightarrow \A^n_{(x_1,...,x_n)}/G=W/G$ be the embedding such that $Y\rightarrow X$ and $Z_1\rightarrow Y$ are both weighted blow-ups
with respect to this embedding, for some cyclic group $G$. Let $\phi:W'\rightarrow W$ be the first weighted blow-up and
assume that $P=\cen_{Y}F$ is the origin of $U_i\subset W'$. Let $y_1$, ..., $y_n$ be a local coordinate system of $U_i$ 
as in Corollary \ref{wbcc}. We know that $E|_{U_i}=(y_i=0)$. Let $f_4$, ..., $f_n$ be the defining equation of $X\subset W/G$.
Then $f'_4$, ..., $f'_n$ defines $Y|_{U_i}$ where $f'_i=(\phi|_{U_i}^{-1})\ts(f_i)$.
Since $Y$ has terminal singularities, the weighted embedding dimension of $Y|_{U_i}$ near $P$ is less than $4$. Hence we may write
$f'_j=\xi_jy_j+f'_j(y_1,...y_4)$ for some $\xi_j$ which do not vanish on $P$, for $5\leq j\leq n$. One can always assume that
$H_{Y}=(y_3=0)$ and so $i\neq 3$.\par 
Let ${f'_j}^{\circ}=f'_j|_{y_i=y_3=0}$. Then ${f'_4}^{\circ}$, ..., ${f'_n}^{\circ}$ defines $H\cap E$ near $P$. If
${f'_j}^{\circ}$ is irreducible as a $G'$-semi-invariant function, then we let $\eta'_j={f'_j}^{\circ}$. Otherwise let $\eta'_j$ be a
$G'$-semi-invariant irreducible factor of ${f'_j}^{\circ}$.
\begin{lem}\label{irc}
	Assume that $Y\rightarrow X$ can be viewed as a four-dimensional weighted blow-up, then
	$\eta'_4=...=\eta'_n=0$ defines an irreducible component of $H_{Y}\cap E$.
\end{lem}
\begin{proof}
	Since $Y\rightarrow X$ can be viewed as a four-dimensional weighted blow-up, we know that $i\leq 4$ and
	${f'}^{\circ}_j=y_j+f'_j|_{y_3=y_i=0}$ for all $j>4$. Hence $\eta'_j={f'}^{\circ}_j$. One can see that the projection
	\[ (\eta'_4=...=\eta'_n=0)|_{H_Y\cap E}\subset \Pp(a_1,...,a_n)\rightarrow \Pp(a_1,...,a_4)\supset (\eta'_4=0)|_{H_Y\cap E}\]
	is an isomorphism. Since $\eta'_4$ is an irreducible function, it defines an irreducible curve.
\end{proof}

Notice that $\eta'_j$ is a polynomial in $y_1$, ..., $y_n$. There exists $\eta_j\in\Oo_W$ such that $\eta'_j=(\phi|_{U_i}^{-1})\ts(\eta_j)$.
We assume that $Y\rightarrow X$ is a weighted blow-up with the weight $\frac{1}{r}(a_1,...,a_n)$ and $Z_1\rightarrow Y$ is a weighted blow-up
with the weight $\frac{1}{r'}(a'_1,...,a'_n)$.

\begin{lem}\label{kng}
	Let $\Gamma=(\eta'_4=...=\eta'_n=0)$ and assume that $\Gamma$ is an irreducible and reduced curve. Then
	\[K_{Z_1}.\Gamma_{Z_1}=-\frac{a_3^2v_E(\eta_4)...v_E(\eta_n)r^{n-3}}{ma_1...a_n}+
		\frac{a'_iv_F(\eta'_4)...v_F(\eta'_n){r'}^{n-4}}{a'_1...a'_n}.\]
	Here $m$ is the integer in Lemma \ref{wbup} corresponds to the weighted blow-up $Y\rightarrow X$.
\end{lem}
\begin{proof}
	Since $\Gamma\subset E$ and $K_{Z_1}+H_{Z_1}$ is numerically trivial over $X$, we only need to show that
	\begin{equation}\label{e1}
	H_{Z_1}.\Gamma_{Z_1}=\frac{a_3^2v_E(\eta_4)...v_E(\eta_n)r^{n-3}}{ma_1...a_n}-
		\frac{a'_iv_F(\eta'_4)...v_F(\eta'_n){r'}^{n-4}}{a'_1...a'_n}.
	\end{equation}
	We know that $H_{Z_1}.\Gamma_{Z_1}=H.\Gamma-v_F(H_Y)F.\Gamma_{Z_1}$. We need to say that the first term of (\ref{e1}) equals to
	$H.\Gamma$ and the second term of (\ref{e1}) equals to $v_F(H_Y)F.\Gamma_{Z_1}$.\par
	We have an embedding $Y\subset W'\subset\Pp_W(a_1,...,a_n)$. Let $D_j$ be the divisor on $W'$ corresponds to $\eta'_j$, then
	$\Gamma=D_4.\cdots. D_n.E.H$ is a weighted complete intersection, so $\Gamma_{Z_1}=D_{4,Z_1}.\cdots. D_{n,Z_1}.E_{Z_1}.H_{Z_1}$.
	To compute $H.\Gamma$ we view $\Gamma$ as a curve inside $\Pp(a_1,...,a_n)$ which is defined by $H=D_4=...=D_n=0$. It follows that
	\[ H.\Gamma=\frac{a_3^2v_E(\eta_4)...v_E(\eta_n)r^{n-3}}{ma_1...a_n}.\]
	To compute $F.\Gamma_{Z_1}$ one write
	\begin{align*}
	F.\Gamma_{Z_1}&=F.(\psi\st D_4-v_F(\eta'_4)F).\cdots.(\psi\st D_n-v_F(\eta'_n)F)
			.(\psi\st E-v_F(E)F).(\phi\st H_{Y}-v_F(H_{Y})F)\\
			&=(-1)^{n-1}v_F(\eta'_4)...v_F(\eta'_n)v_F(E)v_F(H_{Y})F^n.
	\end{align*}
	Since $Z_1\rightarrow Y$ is a $w$-morphism, the integer $\lambda$ in Section \ref{swu} is one. Hence
	we know that $F^n=\frac{(-1)^{n-1}{r'}^{n-1}}{a'_1...a'_n}$. Now $v_F(E)=\frac{a'_i}{r'}$ and $v_F(H_{Y})=a(Y,F)=\frac{1}{r'}$, so
	\[ v_F(H_Y)F.\Gamma_{Z_1}=\frac{a'_iv_F(\eta'_4)...v_F(\eta'_n){r'}^{n-4}}{a'_1...a'_n}.\]
\end{proof}
\begin{lem}\label{knef}
	Notation and assumption as in Lemma \ref{kng}. Assume that
	\begin{enumerate}[(i)]
	\item For all $4\leq j\leq n$, there exists an integer $\delta_j$ so that $x_{\delta_j}^{k_j}\in\eta_j$ for some positive integer $k_j$.
		Moreover, those integers $\delta_4$, ..., $\delta_n$ are all distinct.
	\item If $j\neq i$, $3$, $\delta_4$, ..., $\delta_n$, then $a_3a'_j\geq a_j$.
	\end{enumerate}
	Then $K_{Z_1}.\Gamma_{Z_1}\leq 0$.  
\end{lem}
\begin{proof}
	Fix $j\geq 4$. From the construction and our assumption we know that that $i$, $\delta_4$, ..., $\delta_n$ are all distinct.
	One can see that $rv_E(\eta_j)=k_j a_{\delta_j}$ and $r'v_F(\eta'_j)\leq k_ja'_{\delta_j}$. Thus we have a relation
	\[ \frac{rv_E(\eta_j)}{a_{\delta_j}}\geq \frac{r'v_F(\eta'_j)}{a'_{\delta_j}}.\]\par
	One can always assume that if $j>4$, $j\neq i$, then $\delta_j=j$. By interchanging the order of $y_1$, ..., $y_4$ we may assume that
	$\delta_4=4$. Now if $i<4$ then we may assume that $i=1$. If $i>4$ then we may assume that $\delta_i=1$. We can write
	\[ \frac{a_3^2v_E(\eta_4)...v_E(\eta_n)r^{n-3}}{ma_1...a_n}=
		\frac{1}{ma_i}\frac{a_3}{a_2}\frac{rv_E(\eta_4)}{a_{\delta_4}}...\frac{rv_E(\eta_n)}{a_{\delta_n}}\]
	and \[ \frac{a'_iv_F(\eta_4)...v_F(\eta'_n){r'}^{n-4}}{a'_1...a'_n}=
		\frac{1}{r'}\frac{1}{a'_2}\frac{r'v_F(\eta'_4)}{a'_{\delta_4}}...\frac{r'v_F(\eta'_n)}{a'_{\delta_n}}.\]
	Since $ma_i=r'$, $\frac{a_3}{a_2}\geq\frac{1}{a'_2}$ and
	$\frac{rv_E(\eta_j)}{a_{\delta_j}}\geq \frac{r'v_F(\eta'_j)}{a'_{\delta_j}}$, we know that
	\[ \frac{a_3^2v_E(\eta_4)...v_E(\eta_n)r^{n-3}}{ma_1...a_n}\geq\frac{a'_iv_F(\eta'_4)...v_F(\eta'_n){r'}^{n-4}}{a'_1...a'_n},\]
	So $K_{Z_1}.\Gamma_{Z_1}\leq 0$.
\end{proof}
\begin{rk}\label{knef2}$ $
	\begin{enumerate}[(1)]
	\item From the construction we know that if $j\geq 5$, $j\neq i$, then one can choose $\delta_j=j$.
	\item If $Y\rightarrow X$ can be viewed as a four-dimensional weighted blow-up, then the condition (i) of Lemma \ref{knef}
		always holds. Indeed, in this case one has $i\leq 4$, so $\eta'_4$ is a two-variable irreducible function, hence there exists
		$\delta_4\leq 4$ such that $y_{\delta_4}^{k_4}\in\eta'_4$ for some positive integer $k_4$. One also has $\delta_j=j$ for all $j>4$.
		Thus the condition (i) of Lemma \ref{knef} holds.
	\item If for $j\neq i$, $3$, $\delta_4$, ..., $\delta_n$ one has that $a_j\leq a_i$, then the condition (ii) of Lemma \ref{knef} holds.
		Indeed, by Lemma \ref{wbup} we know that $U_i\cong\A^n/\ang{\tau,\tau'}$ where $\tau$ corresponds to the vector
		$v=\frac{1}{a_i}(a_1,...,a_{i-1},-r,a_{i+1},a_n)$. Let $\bar{v}$ be the vector corresponds to the cyclic action near $P\in U_i$,
		then $v\equiv m\bar{v}$ (mod $\Z^n$) and $r'=ma_i$. Since $Z_1\rightarrow Y$ is a $w$-morphism, and since $H_Y$ is defined by $y_3=0$,
		we know that $a'_3=1$ and $\bar{v}\equiv a_3\frac{1}{r'}(a'_1,...,a'_n)$ (mod $\Z^n$). One can see that 
		$a_3a'_j\equiv a_j$ (mod $r'$). This implies that $a_3a'_j\geq a_j$ since \[a_j\leq a_i\leq ma_i=r'.\]\par
	\end{enumerate}
\end{rk}

\begin{lem}\label{kneff}
	Assume that $K_{Z_1}$ is anti-nef over $X$ and there exists $u\in\Oo_X$ such that $v_E(u)<\frac{a(E,X)}{a(F,X)}v_F(u)$,
	then $Z_i\dashrightarrow Z_{i+1}$ is a $K_{Z_i}$-flip or flop for all $1\leq i\leq k-1$, and $Z_k\rightarrow Y_1$ is a terminal
	divisorial contraction.\par
	In particular, if there exists $j\neq i$ such that $a_3a'_j>a_j$, then the conclusion of this lemma holds.
\end{lem}
\begin{proof}
	We only need to show that $E_{Z_1}$ is not covered by $K_{Z_1}$-trivial curves, then the conclusion follows from Lemma \ref{kan}.\par
	Assume that $E_{Z_1}$ is covered by $K_{Z_1}$-trivial curves. Since $K_{Z_1}$ is anti-nef, those $K_{Z_1}$-trivial curves
	are contained in the boundary of the relative effective cone $NE(Z_1/X)$. Hence $k=1$ and $Z_1\rightarrow Y_1$ is a $K_{Z_1}$-trivial
	divisorial contraction. Notice that if $C_Y\subset E$ is a curve which do not contain $P$, then $K_{Z_1}.C_{Z_1}=K_Y.C_Y<0$,
	hence the curve $C_{Z_1}$ is not contracted by $Z_1\rightarrow Y_1$. Thus $Z_1\rightarrow Y_1$ is a divisorial contraction to the curve
	$C_{Y_1}$. Notice that in this case $a(E,Y_1)=0$.\par
	By computing the discrepancy one can see that the pull-back of $F_{Y_1}$ on $Z_1$ is $F_{Z_1}+\frac{a(E,X)}{a(F,X)}E_{Z_1}$.
	It follows that for all $u\in\Oo_X$ one has that \[v_E(u)\geq\frac{a(E,X)}{a(F,X)}v_F(u).\]
	Hence if there exists $u$ such that $v_E(u)<\frac{a(E,X)}{a(F,X)}v_F(u)$, then $E_{Z_1}$ is not covered by $K_{Z_1}$-trivial curves,
	so $Z_k\rightarrow Y_1$ is a terminal divisorial contraction.\par
	Now by Lemma \ref{dcp} we know that \[\frac{a(E,X)}{a(F,X)}=\frac{r'a_3}{r+a_3a'_i}.\]
	Consider $u=x_j$. For $j\neq i$ we know that $v_E(x_j)=\frac{a_j}{r}$
	and  \[ v_F(x_j)=v_F(y_jy_i^{\frac{a_j}{r}})=\frac{a'_j}{r'}+\frac{a_ja'_i}{rr'}=\frac{ra'_j+a_ja'_i}{rr'}.\]
	The inequality $v_E(u)\geq\frac{a(E,X)}{a(F,X)}v_F(u)$ becomes
	\[ \frac{a_j}{r}\geq \frac{r'a_3}{r+a_3a'_i}\frac{ra'_j+a_ja'_i}{rr'}=\frac{1}{r}\frac{a_3(ra'_j+a_ja'_i)}{r+a_3a'_i},\]
	or equivalently, \[ a_j(r+a_3a'_i)\geq a_3(ra'_j+a_ja'_i),\] or simpler, \[a_j\geq a_3a'_j.\]
	Hence the condition $a_3a'_j>a_j$ implies that $v_E(u)<\frac{a(E,X)}{a(F,X)}v_F(u)$.
\end{proof}
\begin{lem}\label{dcp}
	One has that \[a(E,X)=\frac{a_3}{r},\quad a(F,X)=\frac{r+a_3a'_i}{rr'}.\]
\end{lem}
\begin{proof}
	Since $a(E,X,H)=0$, we know that $a(E,X)=v_E(H)=\frac{a_3}{r}$. Then
	\[ a(F,X)=\frac{1}{r'}+\frac{a_3}{r}\frac{a'_i}{r'}=\frac{r+a_3a'_i}{rr'}.\]
\end{proof}
\begin{rk}
	Note that the assumption in Lemma \ref{knef} depends only on the first weighted blow-up $Y\rightarrow X$. In other word, we can
	check whether the assumption holds or not by simply consider the embedding which defines the weighted blow-up $Y\rightarrow X$,
	instead of considering the (possibly) larger embedding which defines both $Y\rightarrow X$ and $Z_1\rightarrow Y$. Likewise,
	to apply Lemma \ref{kneff} we can only consider the embedding defines $Y\rightarrow X$, if condition $a_3a'_j>a_j$
	already holds under this embedding.
\end{rk}
\begin{nota}$ $
	\begin{enumerate}[(1)]
	\item We say that the condition ($\Xi$) holds if the conditions (i) and (ii) in Lemma \ref{knef} holds for all possible choice of $\Gamma$.
		We say that the condition ($\Xi'$) holds if the conditions (2) and (3) in Remark \ref{knef2} holds for all possible choice of $\Gamma$.
		As explained in Remark \ref{knef2}, we know that the condition ($\Xi'$) implies the condition ($\Xi$).
	\item We say that the condition ($\Xi_-$) (resp. ($\Xi'_-$)) holds if the condition ($\Xi$) (reps. ($\Xi'$)) holds and the inequality
		in Lemma \ref{knef} is strict for all possible choice of $\Gamma$. Using the notation in Lemma \ref{knef} it is equivalent to say that 
		either there exists $j\neq i$, $3$, $\delta_4$, ..., $\delta_n$ such that $a_3a'_j>a_j$, or there exists
		$j\geq 4$ such that \[\frac{rv_E(\eta_j)}{a_{\delta_j}}>\frac{r'v_F(\eta'_j)}{a'_{\delta_j}}.\]
	\item We say that the condition ($\Theta_u$) holds for some function $u$ if $v_E(u)<\frac{a(E,X)}{a(F,X)}v_F(u)$. 
		We say that the condition ($\Theta_j$) holds for some index $j$ if $a_3a'_j>a_j$. In either case Lemma \ref{kneff} can be applied.
	\end{enumerate}
\end{nota}
\begin{nota}
	We say that a divisorial contraction $Y\rightarrow X$ is linked to another divisorial contraction $Y_1\rightarrow X$ if the
	diagram \[ \vc{\xymatrix{ Z_1\ar@{-->}[r] \ar[d] & ... \ar@{-->}[r] & Z_k \ar[d] \\ Y \ar[rd] & & Y_1\ar[ld] \\ & X & }}\]
	exists, where $Z_1\rightarrow Y$ is a strict $w$-morphism over a non-Gorenstein point, $Z_k\rightarrow Y_1$ is a divisorial contraction
	and $Z_i\dashrightarrow Z_{i+1}$ is a flip or a flop for all $1\leq i\leq k-1$. We use the notation $Y\ua{X}Y_1$ if
	$Y\rightarrow X$ is linked to $Y_1\rightarrow X$.\par
	Furthermore, if all $Z_i\dashrightarrow Z_{i+1}$ are all flips, or $k=1$, then we say that $Y$ is negatively linked to $Y_1$,
	and use the notation $Y\uan{X}Y_1$. 
\end{nota}
\begin{rk}
	At this point it is not clear that why $Z_1\rightarrow Y$ should be a divisorial contraction to a non-Gorenstein point. 
	In fact, from the classification of divisorial contractions between terminal threefolds (cf. tables in Section \ref{slk})
	one can see that if there are two different divisorial contractions $Y\rightarrow X$ and $Y_1\rightarrow X$, then
	$Y$ or $Y_1$ always contain a non-Gorenstein point. It is natural to construct the diagram starting with a most singular point,
	which is always a non-Gorenstein point.
\end{rk}

\begin{rk}$ $
	\begin{enumerate}[(1)]
	\item If ($\Xi$) or ($\Xi'$) holds and ($\Theta_u$) or ($\Theta_j$) holds for some function $u$ or index $j$, then by Lemma \ref{knef}
		and Lemma \ref{kneff} one has that $Y\ua{X}Y_1$.
	\item Assume that ($\Xi_-$) or ($\Xi'_-$) holds and ($\Theta_u$) or ($\Theta_j$) holds for some function $u$ or index $j$, then
		one has that $Y\uan{X}Y_1$.
	\end{enumerate}
\end{rk}

\begin{lem}\label{yly1}
	Assume that \[ X\cong(x_1(x_1+p(x_2,...,x_4))+g(x_2,...,x_4)=0)\subset\A^4/G,\]
	such that 
	\begin{enumerate}
	\item $v_E(g)=\frac{a_1}{r}+v_E(p)=2\frac{a_1}{r}-1$.
	\item $i=1$, $a_2+a_4=a_1$ and $a_3=1$.
	\end{enumerate}
	Then $Y\uan{X}Y_1$.
\end{lem}
\begin{proof}
	We know that $a_1>a_j$ for $j=2$, ..., $4$, so ($\Xi'$) holds. 
	Consider the embedding \[X\cong(x_1x_5+g(x_2,...,x_4)=x_5-x_1-p(x_2,...,x_4)=0)
		\subset\A^5_{(x_1,...,x_5)}/G,\]
	then $Y\rightarrow X$ can be obtained by weighted blowing-up the weight $\frac{1}{r}(a_1,...,a_5)$ with respect this embedding,
	where $a_5=rv_E(p)$. The origin of $U_1$ is a cyclic quotient point of type $\frac{1}{a_1}(-r,a_2,...,a_5)$. 
	The only $w$-morphism is the weighted blow-up corresponds to the weight $w(y_2,...y_4)=\frac{1}{a_1}(a_2,...,a_4)$.
	One can see that $a'_5=rv_E(g)>rv_E(p)=a_5$, hence ($\Theta_5$) holds. Moreover, one can see that $\eta'_5=y_5-p(y_2,0,y_4)$,
	so $r'v_F(\eta'_5)=rv_E(\eta_5)=rv_E(p(x_2,0,x_4))$, hence \[ \frac{rv_E(\eta_5)}{a_5}>\frac{r'v_F(\eta'_5)}{a'_5}.\]
	Thus $Y\uan{X}Y_1$.
\end{proof}

\begin{lem}\label{disef}
	Assume that $Y\rightarrow X$ and $Y_1\rightarrow X$ are two divisorial contractions such that $Y\ua{X}Y_1$. Let $E$ and $F$ be the
	exceptional divisors of $Y\rightarrow X$ and $Y_1\rightarrow X$ respectively. Assume that there exists $u\in\Oo_X$ such that
	$v_F(u)=\frac{1}{r}$ and $v_F(u')>0$ where $r$ is the Cartier index of $\cen_XE$ and $u'$ is the strict transform of $u$ on $Y$.
	Then $a(F,X)<a(E,X)$ if $a(E,X)>1$.
\end{lem}
\begin{proof}
	Notice that we have \[v_F(u)=v_F(\tl{u})+v_E(u)v_F(E).\] Since $v_E(u)\geq\frac{1}{r}$ and $v_F(\tl{u})>0$, we know that $v_F(E)<1$. Thus
	\[ a(F,X)=\frac{1}{r'}+a(E,X)v_F(E)\leq \frac{1}{r'}+a(E,X)\frac{r'-1}{r'}=a(E,X)+\frac{1-a(E,X)}{r'}\] where $r'$ is the Cartier index
	of $\cen_YF$. Hence $a(F,X)<a(E,X)$ if $a(E,X)>1$.
\end{proof}

\section{Constructing links}\label{slk}
The aim of this section is to prove the following proposition:
\begin{pro}\label{dpf}
	Let $X$ be a terminal threefold and $Y\rightarrow X$, $Y'\rightarrow X$ be two different divisorial contractions to points over $X$.
	Then there exists $Y_1$, ..., $Y_k$, $Y'_1$, ..., $Y'_{k'}$ such that
	\[ Y=Y_1\ua{X}Y_2\ua{X}...\ua{X}Y_k=Y'_{k'}\ual{X}...\ual{X}Y'_1=Y.\]
\end{pro}
\begin{proof}
	We need a case-by-case discussion, according to the type of the singularity on $X$. Please see Proposition \ref{ca},
	Proposition \ref{cax}, Proposition \ref{cd1}-\ref{cd6} and Proposition \ref{ce1}-\ref{ce4}.
\end{proof}
We keep the notation in Section \ref{sfac}. 
\subsection{Divisorial contractions to $cA/r$ points}
In this subsection we assume that $X$ has $cA/r$ singularities. Divisorial contractions over $X$ are listed in Table \ref{taA}.\par
\begin{table}
\begin{tabular}{|c|c|c|c|c|}\hline
No. & defining equations  & \tc{$(r;a_i)$}{weight} & \tc{type}{\scriptsize{$a(X,E)$}} & condition  \\\hline
A1 & $xy+z^{rk}+g\gq{ka}(z,u)$ & \tc{$(r;\beta,-\beta,1,r)$}{$\frac{1}{r}(b,c,a,r)$} & \tc{$cA/r$}{$a/r$} &
	\tcc{$b\equiv a\beta$ (mod $r$),}{$b+c=rka$}\\\hline
A2 & \tcc{$x^2-y^2+z^3+$}{$xu^2+g\gq6(x,y,z,u)$} & \tc{$(1;-)$}{$(4,3,2,1)$} & \tc{$cA_2$}{3} & $xz\not\in g(x,y,z,u)$ \\\hline	
\end{tabular}\caption{Divisorial contractions to $cA/r$ points} \label{taA}\end{table}
\begin{pro}\label{ca} $ $
	\begin{enumerate}[(1)]
	\item If $Y\rightarrow X$ is of type A1 with $a>1$, then $Y\uan{X}Y_1$ for some $Y_1\rightarrow X$ which is of type A1 with
		the discrepancy less than $a$.
	\item If $Y\rightarrow X$ is of type A1 with $a=1$ and $b>r$, then $Y\ua{X}Y_1$ where $Y_1$ is a type A1 weighted blow-up with
		the weight $\frac{1}{r}(b-r,c+r,1,r)$. One also has $Y_1\ua{X}Y$ if we begin with $Y_1\rightarrow X$ and interchange the role of
		$x$ and $y$. Moreover, $Y\not\uan{X}Y_1$ if and only if $\eta_4=y$.
	\item If $Y\rightarrow X$ is of type A2, then $Y\uan{X}Y_1$ where $Y_1\rightarrow X$ is a divisorial contraction with type A1.
	\end{enumerate}
\end{pro}
\begin{proof}
	Assume first that $Y\rightarrow X$ is of type A1 and we are going to prove (1) and (2). If both $b$ and $c$ are less than $r$,
	then $a=k=1$. In this case there is exactly one divisorial contraction of type A1, so there is nothing to prove. Thus we may assume that
	one of $b$ or $c$, say $b>r$.\par
	The origin of the chart $U_x\subset Y$ is a cyclic quotient point. On this chart one can choose $(y_1,...,y_4)=(x,u,z,y)$ with $i=1$ and
	$\delta_4=4$. One can see that ($\Xi'$) holds. Now the two action in Lemma \ref{wbup} is given by
	\[ \tau=\frac{1}{b}(-r,c,a,r),\quad \tau'=\frac{1}{b}(\beta,\frac{-\beta(b+c)}{r},\frac{b-a\beta}{r},b-\beta).\]
	Since $U_x$ is terminal, there exists a vector $\tau''=\frac{1}{b}(b-\delta,\epsilon,1,\delta)$ such that
	$\tau\equiv a\tau''$ (mod $\Z^4$) and $\tau'\equiv\lambda'\tau''$ (mod $\Z^4$) for some integer $\lambda'$.
	There is exactly one $w$-morphism over the origin of $U_x$ extracts the exceptional divisor $F$ so that $v_F$ which corresponds to
	the vector $\tau''$. One can also see that \[\frac{\epsilon}{b}=v_F(y)=v_F(g')\geq \frac{rk}{b}\]
	where $g'$ is the strict transform of $g$ on $U_x$, since if $z^{rp}u^q\in g'$, then $ap+q\geq ak$ and
	\[v_F(z^{rp}u^q)=\frac{1}{b}(rp+\delta q)=\frac{1}{ab}(rap+\delta aq)\geq\frac{rka}{ab}=\frac{rk}{b}\] for $\delta a\geq r$ because that
	$\delta a\equiv r$(mod $b$) and $b>r$.
	Thus \[ a_4=c<rka\leq a\epsilon=a_3a'_4,\] hence ($\Theta_4$) holds and there exists $Y_1\rightarrow X$ such that $Y\ua{X}Y_1$.\par
	We need to check whether ($\Xi'_-$) holds or not. We has that ${f'_4}^{\circ}=\eta'_4=y_4+{g'}^{\circ}$. One always has that
	\[\frac{r'v_F(\eta'_4)}{a'_4}=\frac{b\frac{\epsilon}{b}}{\epsilon}=1.\] Now ${g'}^{\circ}=0$ if and only if
	\[\frac{rv_E(\eta_4)}{a_4}=\frac{r\frac{c}{r}}{c}=1,\] and $\delta a=r$ if and only if \[ a_3a'_2=a\delta=r=a_2.\]
	Thus ($\Xi'_-$) holds if and only if ${g'}^{\circ}\neq 0$ or $a$ do not divide $r$.\par
	One can compute the discrepancy of $Y_1\rightarrow X$ using Lemma \ref{dcp}. We know that $a'_i=b-\delta$ and $a_3=a$, so
	\[ a(F,X)=\frac{r+a_3a'_i}{rr'}=\frac{r-\delta a+ba}{rb}\leq \frac{a}{r}=a(E,X).\]
	If $a=1$, then $r=\delta$, so $a(F,X)=a(E,X)=\frac{1}{r}$. One can verify that
	$Y_1\rightarrow X$ is the weighted blow-up with the weight $\frac{1}{r}(b-r,c+r,1,r)$. In this case $Y\uan{X}Y_1$ if and only if
	${g'}^{\circ}\neq 0$. Hence $Y\not\uan{X}Y_1$ if and only if $\eta_4=y$. This proves (2).\par
	Now assume that $a>1$. We already know that $\delta a\geq r$. If $\delta a>r$ then $a(F,X)<a(E,X)$ and $Y\uan{X}Y_1$, so (1) holds.
	Hence one only need to show that $\delta a\neq r$. If $\delta a=r$, then
	\[ b=a\beta+\lambda' r=a(\beta+\lambda'\delta)\] where $\lambda'=\frac{b-a\beta}{r}$. 
	One can see that $b-\beta=(a-1)\beta+\lambda'a\delta$. On the other hand, since $\tau'\equiv\lambda'\tau''$ (mod $\Z^n$), we know that
	$b-\beta\equiv \lambda'\delta$ (mod $b$). Hence $b$ divides \[ b-\beta-\lambda'\delta=(a-1)(\beta+\lambda'\delta).\]
	It is impossible since $(a-1)(\beta+\lambda'\delta)$ is an positive integer and is less than $b$.\par
	Finally assume that $Y\rightarrow X$ of type A2. In this case one needs to look at the chart $U_x\subset Y$ and we choose
	$(y_1,...,y_4)=(x,z,y,u)$ with $i=1$ and $\delta_4=4$. One can see that ($\Xi'$) holds.
	The origin of the chart $U_x$ is a $cAx/4$ point of the following form
	\[ (x^2-y^2+z^3+u^2+g'(x,y,z,u)=0)\subset\A^4_{(x,y,z,u)}/\frac{1}{4}(1,1,2,3).\]
	From \cite[Theorem 7.9]{h1} we know that there are exactly two $w$-morphisms over this point which are
	defined by weighted blowing-up the weights $w_{\pm}(x\pm y,x\mp y,z,u)=\frac{1}{4}(5,1,2,3)$. 
	For both these two $w$-morphism one has that $a_2=2$, $a_3=3$ and $a'_2=2$, so ($\Theta_2$) holds and ($\Xi_-$) holds since $a_3a'_2>a_2$.
	Thus there exists $Y_1\rightarrow X$ so that $Y\uan{X}Y_1$. One can compute that the discrepancy of $Y_1\rightarrow X$ is one,
	so $Y_1\rightarrow X$ is of type A1. This proves (3).
\end{proof}

\subsection{Divisorial contractions to $cAx/r$ points}
In this subsection we assume that $X$ has $cAx/r$ singularities with $r=2$ or $4$. Divisorial contractions over $X$ are listed in Table
\ref{taAx}.\par
\begin{table}
\begin{tabular}{|c|c|c|c|c|}\hline
No. & defining equations  & \tc{$(r;a_i)$}{weight} & \tc{type}{\scriptsize{$a(X,E)$}} & condition  \\\hline
Ax1 & $x^2+y^2+g\gq{\frac{2k+1}{2}}(z,u)$ & \tc{$(4;1,3,1,2)$}{$\frac{1}{4}(b,c,1,2)$} & \tc{$cAx/4$}{$1/4$} &
	\small{\tcc{$(b,c)=(2k+1,2k+3)$}{or $(2k+3,2k+1)$}} \\\hline
Ax2 & \tcc{$x^2+y^2+(\lambda x+\mu y)p_{\frac{2k+1}{4}}(z,u)$}{$+g\gq{\frac{2k+3}{2}}(z,u)$} &
	\tc{$(4;1,3,1,2)$}{$\frac{1}{4}(b,c,1,2)$} & \tc{$cAx/4$}{$1/4$} &
	\small{\ttc{$(b,c,\lambda,\mu)=$}{$(2k+5,2k+3,1,0)$}{or $(2k+3,2k+5,0,1)$}} \\\hline
Ax3 & $x^2+y^2+g\gq k(z,u)$ & \tc{$(2;0,1,1,1)$}{$\frac{1}{2}(b,c,1,1)$} & \tc{$cAx/2$}{$1/2$} &
	\small{\tcc{$(b,c)=(k,k+1)$}{or $(k+1,k)$}} \\\hline
Ax4 & \tcc{$x^2+y^2+(\lambda x+\mu y)p_{\frac{k}{2}}(z,u)$}{$+g\gq{k+1}(z,u)$} &
	\tc{$(2;0,1,1,1)$}{$\frac{1}{2}(b,c,1,1)$} & \tc{$cAx/2$}{$1/2$} &
	\small{\ttc{$(b,c,\lambda,\mu)=$}{$(k+2,k+1,1,0)$}{or $(k+1,k+2,0,1)$}} \\\hline
\end{tabular}\caption{Divisorial contractions to $cAx/r$ points} \label{taAx}\end{table}

\begin{pro}\label{cax}
	\begin{enumerate}[(1)]
	\item Assume that $Y\rightarrow X$ is of type Ax1 or Ax3, then $Y\rightarrow X$ is the only divisorial contraction over $X$.
	\item Assume that $Y\rightarrow X$ is of type Ax2 or Ax4, then there are exactly two divisorial contractions over $X$. Let
		$Y_1\rightarrow X$ be another divisorial contraction, then $Y_1\rightarrow X$ has the same type of $Y\rightarrow X$ and
		one has that $Y\uan{X}Y_1\uan{X}Y$.
	\end{enumerate}
\end{pro}
\begin{proof}
	The number of divisorial contractions follows from \cite[Section 7,8]{h1}. So we can assume that $Y\rightarrow X$ is of type Ax2 or
	Ax4, and Lemma \ref{yly1} implies that $Y\uan{X}Y_1$.
\end{proof}
\subsection{Divisorial contractions to $cD$ points}
In this subsection we assume that $X$ has $cD$ singularities. At first we consider $w$-morphisms over $X$, which are listed in Table \ref{taD}.
Notice that for types D1, D2 or D5 in Table \ref{taD} there is at most one divisorial contraction over $X$ which is of the given type.
It is because that the equations of type D1, D2 and D5 coming from the normal form of $cD$-type singularities, which are unique,
and the blowing-up weights are determined by the defining equations.

\begin{table}
\begin{tabular}{|c|c|c|c|c|}\hline
No. & defining equations  & weight & \tc{type}{\scriptsize{$a(X,E)$}} & condition  \\\hline
D1 & $x^2+y^2u+\lambda yz^k+g\gq{l}(z,u)$ & $(b,b-1,1,2)$ & \tc{$cD$}{$1$} & $b=\min\{k-1,\rd{\frac{l}{2}}\}$\\\hline
D2 & $x^2+y^2u+\lambda yz^k+g\gq{2l}(z,u)$ & $(b,b,1,1)$ & \tc{$cD$}{$1$} & $b=\min\{k,l\}$ \\\hline
D3 & $\sep{x^2+ut+\lambda yz^k+g\gq{2b+2}(z,u)}{y^2+p_{2b}(x,z,u)+t}$ & \small{$(b+1,b,1,1,2b+1)$} & \tc{$cD$}{$1$} & $k\geq b+2$\\\hline
D4 & \tcc{$x^2+y^2u+yh\gq k(z,u)$}{$+g\gq{2b+1}(x,z,u)$} & $(b+1,b,1,1)$ & \tc{$cD$}{$1$} & $k\geq b+1$\\\hline
D5 & $\sep{x^2+yt+g\gq{2b}(z,u)}{yu+p_b(z,u)+t}$ & \small{$(b,b-1,1,1,b+1)$} & \tc{$cD$}{$1$} & $z^b\in p(z,u)$ \\\hline
\end{tabular}\caption{Divisorial contractions to $cD$ points with discrepancy one} \label{taD}\end{table}

\begin{lem}\label{dmor}
	Assume that there exists two different divisorial contractions with discrepancy one over $X$. Then one of the following holds:
	\begin{enumerate}[(1)]
	\item One of the divisorial contractions is of type D1.
	\item The two morphisms are of type D2 and D5 respectively.
	\item Both of the divisorial contractions are of type D3.
	\item Both of the divisorial contractions are of type D4.
	\end{enumerate}
\end{lem}
\begin{proof}
	Assume that $Y\rightarrow X$ and $Y_1\rightarrow X$ are the two given divisorial contractions. It is enough to prove the following statements:
	\begin{enumerate}[(i)]
	\item If $Y\rightarrow X$ is of type D4, then $Y_1\rightarrow X$ is of type D1 or D4.
	\item If $Y\rightarrow X$ is of type D2, then $Y_1\rightarrow X$ is of type D1 or D5.
	\item If $Y\rightarrow X$ is of type D5, then $Y_1\rightarrow X$ is not of type D3.
	\end{enumerate}\par
	Let $E$ and $F$ be the exceptional divisor of $Y\rightarrow X$ and $Y_1\rightarrow X$ respectively. Then $a(F,Y)<1$ since otherwise
	$a(F,X)>1$. Thus $P=\cen_FY$ is a non-Gorenstein point.\par
	First assume that $Y\rightarrow X$ is of type D4. In this case $P$ may be the origin of $U_x$ or the origin of $U_y$ and they are both
	cyclic quotient points. Exceptional divisors over $P$ with discrepancy less than one are described in \cite[Proposition 3.1]{me}.
	The origin of $U_x$ is a $\frac{1}{b+1}(b,1,1)$ point. If $P$ is this point, then since $z=0$ defines a Du Val section,
	we have that $v_F(z)=a(F,X)=1$. One can verify that $v_F(u)=v_F(z)=1$ and $v_F(x)=v_F(y)=b$. Now $v_F(x)=b$ only when
	$xp_b(z,u)\in g(x,z,u)$ for some homogeneous polynomial $p(z,u)$ of degree $b$. One can check that $v_F(x+p(z,u))=b+1$.
	In this case $Y_1\rightarrow X$ is also of type D4 after a change of coordinate $x\mapsto x-p(z,u)$. If $P$ is the origin of $U_y$, then
	it is a $\frac{1}{b}(1,-1,1)$ point. One can verify that $v_F(u)>1$. This implies that $Y_1\rightarrow X$ is of type D1.\par
	Now assume $Y\rightarrow X$ is of type D2, then $P$ is the origin of $U_y\subset Y$, which is a $cA/b$ point. Exceptional divisors of
	discrepancy less than one over $P$ are described in \cite[Proposition 3.4]{me}. One can verify that if $\lambda\neq 0$ and $k=b$, then
	$v_F(u)=1$. In this case $Y_1\rightarrow X$ is of type D5. Otherwise $v_F(u)=2$ and so $Y_1\rightarrow X$ is of type D1.\par
	Finally assume that $Y\rightarrow X$ is of type D5. One can see that $Y_1\rightarrow X$ can not have type D3 since $z^b\in p(z,u)$. 
	This finishes the proof.
\end{proof}
\begin{pro}\label{cd1}
	Assume that there exists two different divisorial contractions with discrepancy one over $X$, say $Y\rightarrow X$ and $Y_1\rightarrow X$.
	\begin{enumerate}[(1)]
	\item If $Y\rightarrow X$ and $Y_1\rightarrow X$ are both of type D4, then $Y\uan{X}Y_1\uan{X}Y$.
	\item If $Y\rightarrow X$ is of type D3 and $Y_1\rightarrow X$ is of type D1, then $Y\uan{X}Y_1$. If $Y_1\rightarrow X$ is of type D3, then
		there exists another divisorial contraction $Y_2\rightarrow X$ which is of type D1, so that
		$Y\uan{X}Y_2\stackrel{-}{\underset{X}{\Leftarrow}}Y_1$.
	\item If $Y\rightarrow X$ is of type D2 and $Y_1\rightarrow X$ is of type D5, then $Y\uan{X}Y_1$.	
	\item If $Y\rightarrow X$ is of type D1 and $Y_1\rightarrow X$ is not of type D3, then $Y\uan{X}Y_1$.
	\end{enumerate}
\end{pro}
\begin{proof}
	Assume first that $Y\rightarrow X$ and $Y_1\rightarrow X$ are both of type D4. Notice that in this case $xp_b(z,u)\in g(x,z,u)$.
	Thus $Y\uan{X}Y_1$ by Lemma \ref{yly1}.\par
	Now assume that $Y\rightarrow X$ is of type D3. Consider the chart $U_t\subset Y$ which is defined by
	\[ (x^2+u+\lambda yz^kt^{b+k-2b-2}+g'(z,u,t)=y^2-p(x,z,u)+t=0)\subset\]\[\A^5_{(x,y,z,u,t)}/\frac{1}{2b+1}(b+1,b,1,1,-1).\]
	Notice that using the notation in Section \ref{sfac}, we know that
	\[{f'_4}^{\circ}=\eta'_4=x^2+u+g'(0,u,0),\quad {f'_5}^{\circ}=y^2+p(x,0,u).\]
	$\eta'_5$ can be $y\pm \mu u^b$ if $p(x,0,u)=-\mu^2u^{2b}$ for some $\mu\in\Cc$ and otherwise $\eta'_5={f'_5}^{\circ}$. One can see that
	$\eta'_4=\eta'_5=0$ defines an irreducible and reduced curve. There is only one $w$-morphism over the origin of $U_t$ which is defined
	by weighted blowing up the weight $w(x,y,z,u,t)=\frac{1}{2b+1}(b+1,b,1,2b+2,2b)$. Now in this case we choose $(y_1,...,y_5)=(x,y,z,u,t)$
	with $i=5$, $\delta_4=4$, $\delta_5=2$. One can see that ($\Xi$) holds and ($\Theta_4$) holds. Also one has that
	$\frac{rv_E(\eta_4)}{a_4}=2b+2$ while $\frac{r'v_F(\eta'_4)}{a'_4}=1$. Thus ($\Xi_-$) holds and so there exists
	$Y_2\rightarrow X$ such that $Y\uan{X}Y_2$. One can compute that $Y_2\rightarrow X$ which is of type D1. If $Y_1\rightarrow X$ is of
	type D1 then $Y_2=Y_1$ since there are at most one divisorial contraction with type D1. This proves the statement (2).\par
	Now assume that $Y\rightarrow X$ is of type D2 and $Y_1\rightarrow X$ is of type D5. In this case we consider the embedding corresponds to
	of type $Y_1\rightarrow X$. Under this embedding $Y\rightarrow X$ is given by the weighted blow-up with the weight $(b,b,1,1,b)$
	and the chart $U_y\subset Y$ is given by \[ U_y=(x^2-t+g'(y,z,u)=yu+z^b+t=0)\subset\A^5_{(x,y,z,u,t)}/\frac{1}{b}(0,-1,1,1,0).\]
	We take $(y_1,...,y_5)=(y,u,z,x,t)$ with $\delta_4=4$ and $\delta_5=5$. Then ($\Xi$) holds. The origin of $U_y$ is a $cA/b$ point
	and the weight $w(y_1,...,y_5)=\frac{1}{b}(b-1,1,1,b,2b)$ defines a $w$-morphism over $U_y$. One can see that ($\Theta_5$) holds.
	Moreover, since $a'_5=2b>b=a_5$, we know that ($\Xi_-$) holds. Thus $Y\uan{X}Y_1$.\par
	Finally assume that $Y\rightarrow X$ is of type D1 and $Y_1\rightarrow X$ is not of type D3. Let $b$ and $b_1$ be the integers in Table
	\ref{taD} corresponds to $Y\rightarrow X$ and $Y_1\rightarrow X$ respectively. First we claim that $b\leq b_1$. Indeed,
	if $Y_1\rightarrow X$ is of type D5, then $z^{b_1}\in h(z,u)$ which implies that $b_1\geq b+1$. If $Y_1\rightarrow X$ is of type D2
	or D4 then the inequality $b\leq b_1$ follows from Corollary \ref{wef}.	Now the origin of the chart $U_u\subset Y$ is defined by
	\[ (x^2+y^2+\lambda yz^ku^{k-b-1}+g'(z,u)=0)\subset\A^4_{(x,y,z,u)}/\frac{1}{2}(b,b-1,1,1),\] which is a $cAx/2$ point.
	We can take $(y_1,...,y_4)=(u,y,z,x)$ with $i=1$ and $\delta_4=4$. Now $w$-morphisms over this point is fully described in
	\cite[Section 8]{h1}. Since $b\leq b_1$, we know that the $2k-b-1>b$, and the
	multiplicity of $g'(z,u)$ is greater than or equal to $2b$. Hence if $F$ is the exceptional divisor of a $w$-morphism over $Y$,
	then \[v_F(y)\geq b>b-1=v_E(y).\] Thus ($\Xi_-$) and ($\Theta_2$) holds and one has $Y\uan{X}Y_1$.
\end{proof}

Now we study divisorial contractions of discrepancy greater than one. All such divisorial contractions are listed in Table \ref{taD2}.

\begin{table}
\begin{tabular}{|c|c|c|c|c|}\hline
No. & defining equations  & weight & \tc{type}{\scriptsize{$a(X,E)$}} & condition  \\\hline
D6 & $x^2+y^2u+z^k+g\gq{2b+1}(x,y,z,u)$ & $(b+1,b,a,1)$ & \tc{$cD$}{$a$} & $ak=2b+1$\\\hline
D7 & $\sep{x^2+yt+g\gq{2b+2}(y,z,u)}{yu+z^k+p_{b+1}(z,u)+t}$ & $(b+1,b,a,1,b+2)$ & \tc{$cD$}{$a$} & $ak=b+1$\\\hline
D8 & $\sep{x^2+ut+\lambda z^{\frac{b+1}{4}}+g\gq{b+1}(y,z,u)}{y^2+\mu z^{\frac{b-1}{4}}+p_{b-1}(x,z,u)+t}$ &
	$(\frac{b+1}{2},\frac{b-1}{2},4,1,b)$ & \tc{$cD$}{$4$} &
	\ttc{$\frac{b+1}{4}\in\N$, $\lambda=1$,}{$\mu=0$, or $\frac{b-1}{4}\in\N$,}{$\mu=1$, $\lambda=0$}.\\\hline
D9 & $\sep{x^2+ut+z^{\frac{b+1}{2}}+g\gq{b+1}(y,z,u)}{y^2+p_{b-1}(x,z,u)+t}$ & $(\frac{b+1}{2},\frac{b-1}{2},2,1,b)$
	& \tc{$cD$}{$2$} &\\\hline
D10 & $x^2+y^2u+z^b+g\gq{2b}(y,z,u)$ & $(b,b,2,1)$ & \tc{$cD$}{2} & \\\hline
D11 & \tcc{$x^2+y^2u+yp_3(z,u)+$}{$u^3+g\gq6(z,u)$} & $(3,3,1,2)$ & \tc{$cD_4$}{2} & $z^3\in p(z,u)$\\\hline
D12 & \tcc{$x^2+y^2u+z^3+$}{$yu^2+g\gq6(y,z,u)$} & $(3,4,2,1)$ & \tc{$cD_4$}{3} & \\\hline
\end{tabular}\caption{Divisorial contractions to $cD$ points with discrepancies greater than one} \label{taD2}\end{table}

\begin{pro}\label{cd2}
	Assume that $Y\rightarrow X$ is a divisorial contraction with the discrepancy $a>1$. Then there exists a divisorial contraction
	$Y_1\rightarrow X$ such that $Y\ua{X}Y_1$, and $a(F,X)<a$ where $F=exc(Y_1\rightarrow X)$.
\end{pro}
\begin{proof}
	First notice that ($\Theta_u$) holds in case D6-D10 or D12, and ($\Theta_z$) holds in case D11. It is because that
	$v_E(u)$ or $v_E(z)=1$ in those cases and $\frac{a(E,X)}{a(F,X)}=a>1$.\par
	Now we list all cases in Table \ref{taD2}, write down the chart on $Y$ we are looking at, and write down the variables $y_1$, ..., $y_n$.
	One can easily see that ($\Xi$) holds in all cases.
	\begin{enumerate}[(1)]
	\item Assume that $Y\rightarrow X$ is of type D6. Consider the chart $U_x\subset Y$ and take $(y_1,...,y_4)=(x,y,z,u)$ with $\delta_4=4$.
	\item Assume that $Y\rightarrow X$ is of type D7. Consider the chart $U_t\subset Y$ and take $(y_1,...,y_5)=(y,u,z,x,t)$,
		$\delta_4=4$ and $\delta_5=1$ or $2$.
	\item Assume that $Y\rightarrow X$ is of type D8 or D9. We consider the chart $U_t\subset Y$ and take $(y_1,...,y_5)=(x,y,z,u,t)$ with
		$\delta_4=4$ and $\delta_5=2$.
	\item Assume that $Y\rightarrow X$ is of type D10. We consider the chart $U_y\subset Y$ and take $(y_1,...,y_4)=(y,u,z,x)$ with $\delta_4=4$.
	\item Assume that $Y\rightarrow X$ is of type D11. We consider the chart $U_y\subset Y$ and $(y_1,...,y_4)=(y,z,u+\lambda y,x)$ for
		some $\lambda\in\Cc$ with $\delta_4=4$.
	\item Assume that $Y\rightarrow X$ is of type D12. We consider the chart $U_y\subset Y$ and $(y_1,...,y_4)=(y,z,x+\lambda y,u)$ for
		some $\lambda\in\Cc$ with $\delta_4=4$.
	\end{enumerate}		
	Now we know that there exists $Y_1\rightarrow X$ so that $Y\ua{X}Y_1$. Then $Y_1\rightarrow X$ is of one of types in Table \ref{taD} or
	Table \ref{taD2}. One can see that $v_F(z)=1$ if $Y_1\rightarrow X$ is of type D1-D5, D7-D9, D11 and $v_F(u)=1$ if $Y_1\rightarrow X$
	is of type D6, D10 or D12. Since $\cen_YF$ is the origin of the chart $U_x$, $U_y$ or $U_t$, one can apply Lemma \ref{disef} to say that
	$a(F,X)<a$. This finishes the proof.
\end{proof}

\subsection{Divisorial contractions to $cD/r$ points with $r>1$}
In this subsection we assume that $X$ has $cD/r$ singularities with $r=2$ or $3$. We first study $w$-morphisms over $X$.

\begin{table}
\begin{tabular}{|c|c|c|c|c|}\hline
No. & defining equations  & \tc{$(r;a_i)$}{weight} & \tc{type}{\scriptsize{$a(X,E)$}} & condition  \\\hline
D13 & $x^2+y^3+g\gq k(y,z,u)$ & \tc{$(3;0,2,1,1)$}{$\frac{1}{3}(3,2,4,1)$} & \tc{$cD/3$}{$1/3$}
	& \small{\ttc{$k=2$ and $zu^2$}{or $z^3\in g$, or}{$k=3$ and $z^2u\in g$}} \\\hline
D14 & $x^2+y^3+z^3+g\gq4(y,z,u)$ & \tc{$(3;0,2,1,1)$}{$\frac{1}{3}(6,5,4,1)$} & \tc{$cD/3$}{$1/3$} & \\\hline
D15 & $x^2+yzu+g\gq2(y,z,u)$ & \tc{$(2;1,1,1,0)$}{$\frac{1}{2}(3,1,1,2)$} & \tc{$cD/2$}{$1/2$} & \\\hline
D16 & $x^2+yzu+g\gq3(y,z,u)$ & \tc{$(2;1,1,1,0)$}{$\frac{1}{2}(3,b,c,d)$} & \tc{$cD/2$}{$1/2$} &
	$(b,c,d)=$\tcc{$(3,1,2)$}{$(1,1,4)$}\\\hline
D17 & $\sep{x^2+yt+g\gq3(z,u)}{zu+y^3+t}$ & \tc{$(2;1,1,1,0,1)$}{$\frac{1}{2}(3,1,1,2,5)$} & \tc{$cD/2$}{$1/2$} & \\\hline
D18 & $x^2+y^2u+\lambda yz^k+g\gq l(z,u)$ & \tc{$(2;1,1,1,0)$}{$\frac{1}{2}(b,b-2,1,4)$} & \tc{$cD/2$}{$1/2$} &
	\small{$b=\min\{k-2,\ru{\frac{l}{2}}-1\}$}\\\hline
D19 & $x^2+y^2u+\lambda yz^k+g\gq l(z,u)$ & \tc{$(2;1,1,1,0)$}{$\frac{1}{2}(b,b,1,2)$} & \tc{$cD/2$}{$1/2$} &
	$b=\min\{k,l\}$ \\\hline
D20 & $\sep{x^2+ut+\lambda yz^k+g\gq {b+2}(z,u)}{y^2+p_b(x,z,u)+t}$ &
	\tc{$(2;1,1,1,0,0)$}{\small{$\frac{1}{2}(b+2,b,1,2,2b+2)$}} & \tc{$cD/2$}{$1/2$} & $k\geq b+4$ \\\hline
D21 & \tcc{$x^2+y^2u+yh\gq k(z,u)$}{$+g\gq{b+1}(x,z,u)$}  & \tc{$(2;1,1,1,0)$}{$\frac{1}{2}(b+2,b,1,2)$} &
	\tc{$cD/2$}{$1/2$} & $k\geq b+2$ \\\hline
D22 & $\sep{x^2+yt+g\gq{2b}(z,u)}{yu+z^b+t}$ & \tc{$(2;1,1,1,0,1)$}{\small{$\frac{1}{2}(b,b-2,1,2,b+2)$}} &
	\tc{$cD/2$}{$1/2$} & \\\hline
\end{tabular}\caption{Divisorial contractions to $cD/r$ points with discrepancy one} \label{taD3}\end{table}

\begin{pro}\label{cd3}
	Assume that $X$ has $cD/3$ singularities.
	\begin{enumerate}[(1)]
	\item If $Y\rightarrow X$ is of type D14, or $Y\rightarrow X$ is of type D13 and both $zu^2$ and $z^2u\not\in g(y,z,u)$, then
		there is only one $w$-morphism over $X$.
	\item If $Y\rightarrow X$ is of type D13 and $zu^2$ or $z^2u\in g(y,z,u)$, then there are two or three $w$-morphisms over $X$.
		Say $Y_1\rightarrow X$, ..., $Y_k\rightarrow X$ are other $w$-morphisms with $k=1$ or $2$, then
		$Y\uan{X}Y_i\uan{X}Y$ for all $1\leq i\leq k$.
	\end{enumerate}
\end{pro}
\begin{proof}
	The statement about the number of $w$-morphisms follows from \cite[Section 9]{h1}. Now we may assume that $Y\rightarrow X$ is of
	type D13 and $zu^2$ or $z^2u\in g(y,z,u)$. The chart $U_z\subset Y$ is defined by
	\[ (x^2+y^3+g'(y,z,u)=0)\subset\A^4_{(x,y,z,u)}/\frac{1}{4}(3,2,1,1)\] with $u^2$ or $zu\in g'(y,z,u)$.
	We can take $(y_1,...,y_4)=(z,x,u+\lambda z,y)$ for some $\lambda\in\Cc$ with $\delta_4=4$. Now the $w$-morphism over $U_z$ can be obtained
	by weighted blowing-up the weight $w(y_1,...y_4)=\frac{1}{4}(3,5,1,2)$. One can see that ($\Theta_2$) and ($\Xi'_-$) holds.
	Hence we can get a divisorial contraction $Y_1\rightarrow X$ such that $Y\uan{X}Y_1$. One can compute that $Y_1\rightarrow X$ is also a
	$w$-morphism.\par
	If there are three $w$-morphisms over $X$, then the defining equation of $X$ is of the form $x^2+y^3+zu(z+u)$ as in \cite[Section 9.A]{h1},
	so $g'(y,z,u)=u(z+u)$. One can make a change of coordinate $u\mapsto u-z$ and again consider the weighted blow-up with the same weight
	$\frac{1}{4}(3,2,1,5)$. In this way we can get a divisorial contraction $Y_2\rightarrow X$ which is different to $Y_1$ and we also
	have that $Y\uan{X}Y_2$. This finishes the proof.
\end{proof}
\begin{pro}\label{cd4}
	Assume that $X$ has $cD/2$ singularities and $Y\rightarrow X$ is of type D15, D16 or D17.
	\begin{enumerate}[(1)]
	\item If $Y\rightarrow X$ is of type D15, then there is only one $w$-morphism over $X$.
	\item If $Y\rightarrow X$ is of type D17, then there exists exactly two $w$-morphisms over $X$. The other one $Y_1\rightarrow X$ is
		of type D16 and one has that $Y\ua{X}Y_1$.
	\item If $Y\rightarrow X$ is of type D16 and there is no $w$-morphisms over $X$ with type D17, then there are exactly three $w$-morphisms
		over $X$. They are all of type D16 and are negatively linked to each other.
	\end{enumerate}
\end{pro}
\begin{proof}
	The statement about the number of $w$-morphisms follows from \cite[Section 4]{h2}. Assume that $Y\rightarrow X$ is of type D17.
	Consider the chart $U_t\subset Y$ with $(y_1,...,y_5)=(y,u,y+z,x,t)$ with $\delta_4=4$ and $\delta_5=1$. One can see that ($\Xi$) holds.
	Now the origin of $U_t$ is a cyclic quotient point. Let $F$ be the exceptional divisor of the $w$-morphism over $U_t$, then
	one has that $v_F(y_1,...y_5)=\frac{1}{5}(6,2,1,3,3)$. One can see that ($\Theta_1$) holds.\par
	Assume that $Y\rightarrow X$ is of type D16 and there is no $w$-morphisms of type D17 over X. By \cite[Section 4]{h2} we know that
	neither $y^4$ nor $z^4\in g(y,z,u)$. Assume first that $(b,c,d)=(1,1,4)$. Consider the chart $U_u\subset Y$ which has a $cAx/4$ singular
	point at the origin. We choose $(y_1,...,y_4)=(y,u,y+z,x)$ with $\delta_4=4$. One can see that ($\Xi'$) holds. Let $w$ be the
	weight on $U_u$ so that $w(y_1,...,y_4)=\frac{1}{4}(5,2,1,3)$. Then the weighted blow-up with weight $w$ gives a $w$-morphism.
	It follows that ($\Theta_1$) holds and also ($\Xi'_-$) holds since $a'_1=5>3=a_1$. Hence there exists a $w$-morphism
	$Y_1\rightarrow X$ so that $Y\uan{X}Y_1$. If we interchange the role of $y$ and $z$, we can get another $w$-morphism $Y_2\rightarrow X$
	with $Y\uan{X}Y_2$.\par
	Now assume that $(b,c,d)=(3,1,2)$. Consider the chart $U_y\subset Y$ which is defined by
	\[ (x^2+zu+g'(y,z,u)=0)\subset\A^4_{(x,y,z,u)}/\frac{1}{3}(0,1,1,2).\] One can take $(y_1,...,y_4)=(y,u,y+z,x)$ and then ($\Xi'$) holds.
	Let $w$ be the weight $w(y_1,...,y_4)=\frac{1}{3}(1,5,1,3)$, then the weighted blow-up with the weight $w$
	gives a $w$-morphism over $U_y$ and ($\Theta_2$) and ($\Xi'_-$) holds. If we take $w$ to be another weight
	$w(x,y,z,u)=\frac{1}{3}(3,1,4,2)$, then we get another $w$-morphism over $U_y$ and ($\Theta_z$) holds.
	Thus we can get two different $w$-morphisms over $X$ and $Y$ is negatively linked to both of them.
\end{proof}
\begin{pro}\label{cd5}
	Assume that $X$ has $cD/2$ singularities and $Y\rightarrow X$ is of type D18-D23 and assume that there are two $w$-morphisms
	$Y\rightarrow X$ and $Y_1\rightarrow X$.
	\begin{enumerate}[(1)]
	\item Assume that $Y\rightarrow X$ is of type D18 and $Y_1\rightarrow X$ is not of type D20, then $Y\uan{X}Y_1$.
	\item Assume that both $Y\rightarrow X$ and $Y_1\rightarrow X$ are not of type D18, then one of the following holds: 
		\begin{enumerate}[({2}-1)]
		\item $Y\rightarrow X$ is of type D19 and $Y_1\rightarrow X$ is of type D22. One has that $Y\uan{X}Y_1$.
		\item Both $Y\rightarrow X$ is of type D21 and $Y\uan{X}Y_1\uan{X}Y$.
		\item Both $Y\rightarrow X$ is of type D20 and there exists another $w$-morphism $Y_2\rightarrow X$ which is of type D18,
			so that $Y\uan{X}Y_2\stackrel{-}{\underset{X}{\Leftarrow}}Y_1$.
		\end{enumerate}
	\end{enumerate}
\end{pro}
\begin{proof}
	The computation is similar to the proof of Proposition \ref{cd1} after replacing the types D1-D5 by D18-D22 so we will omit the proof.
	Notice that an analog result of Lemma \ref{dmor} can be proved by a similar computation, or can be directly followed by
	\cite[Section 5]{h2}.
\end{proof}

Now we consider non-$w$-morphisms over $X$. Notice that there is no divisorial contraction with discrepancy greater than $\frac{1}{3}$ over
$cD/3$ points. Divisorial contractions of discrepancy greater than $\frac{1}{2}$ over $cD/2$ points are listed in Table \ref{taD4}. 
\begin{table}
\begin{tabular}{|c|c|c|c|c|}\hline
No. & defining equations  & \tc{$(r;a_i)$}{weight} & \tc{type}{\scriptsize{$a(X,E)$}} & condition  \\\hline
D23 & \tcc{$x^2+y^2u+z^m+$}{$g\gq{b+1}(x,y,z,u)$} & \tc{$(2;1,1,1,0)$}{$\frac{1}{2}(b+2,b,a,2)$}
	& \tc{$cD/2$}{$a/2$} & \tcc{$ma=2b+2$,}{$a$ and $b$ are odd} \\\hline
D24 & $\sep{x^2+yt+g\gq{b+2}(z,u)}{yu+z^m+p_{\frac{b}{2}+1}(z,u)+t}$ &
	\tc{$(2;1,1,1,0,1)$}{$\frac{1}{2}(b+2,b,a,2,b+4)$}
	& \tc{$cD/2$}{$a/2$} & \tcc{$ma=b+2$}{$a\equiv b$(mod $2$)} \\\hline
D25 & \tcc{$x^2+y^2u+z^{4b}+$}{$g\gq{4b}(y,z,u)$} & \tc{$(2;1,1,1,0)$}{$(2b,2b,1,1)$} & \tc{$cD/2$}{$1$} & \\\hline
D26 & $x^2+yzu+y^4+z^b+u^c$ & \tc{$(2;1,1,1,0)$}{$(2,1,2,1)$} & \tc{$cD/2$}{$1$} & \tcc{$b,c\geq 4$}{$b$ is even} \\\hline
D27 & $\sep{x^2+ut+y^4+z^4}{yz+u^2+t}$ & \tc{$(2;1,1,1,0,0)$}{$(2,1,1,1,3)$} & \tc{$cD/2$}{$1$} & \\\hline
D28 & $\sep{x^2+ut+g\gq{2b+2}(y,z,u)}{y^2+p_{2b}(x,z,u)+t}$ &
	\tc{$(2;1,1,1,0,0)$}{$(b+1,b,1,1,2b+1)$} & \tc{$cD/2$}{$1$} &
	\ttc{Either $b$ is odd, or}{$b$ is even and}
	{$xz^{b-1}$ or $z^{2b}\in p$}\\\hline
D29 & $\sep{x^2+ut+g\gq{2b+2}(y,z,u)}{y^2+p_{2b}(x,z,u)+t}$ & \tc{$(2;1,1,1,0,0)$}{$(b+1,b,2,1,2b+1)$} &
	\tc{$cD/2$}{$2$} & $xz^{\frac{b-1}{2}}$ or $z^b\in p$\\\hline
\end{tabular}
\caption{Divisorial contractions to $cD/r$ points with large disprepancies} \label{taD4}
\end{table}

\begin{pro}\label{cd6}
	Assume that $r=2$ and $Y\rightarrow X$ is a divisorial contraction with the discrepancy $\frac{a}{2}>1$. Then there exists a
	divisorial contraction $Y_1\rightarrow X$ such that $Y\ua{X}Y_1$, and $a(F,X)<\frac{a}{2}$ where $F=exc(Y_1\rightarrow X)$.
\end{pro}
\begin{proof}
	First assume that $Y\rightarrow X$ is of type D23. Consider the chart $U_x\subset Y$ which is defined by
	\[ (x+y^2u+z^m+g'(x,y,z,u)=0)\subset\A^4_{(x,y,z,u)}/\frac{1}{b+2}(-2,b,a,2).\]
	We take $(y_1,...,y_4)=(x,y,z,u)$ with $\delta_4=2$ or $4$. One can see that ($\Xi'$) holds.
	Since $a_4=2$ and $a\geq 3$, we know that ($\Theta_4$) holds. Thus there exists $Y_1\rightarrow X$ such that
	$Y\ua{X}Y_1$. The origin of $U_x$ is a cyclic quotient point. let $F$ be the exceptional divisor of the $w$-morphism over this point. Then
	$v_F(x)\leq \frac{m}{b+2}$. It follows that \[ a(F,X)=\frac{1}{b+2}+\frac{a}{2}v_F(x)\leq \frac{2+ma}{2b+4}=1<\frac{a}{2}.\]\par
	Assume that $Y\rightarrow X$ is of type D24. Consider the chart $U_t\subset Y$ which is defined by
	\[ (x^2+y+g'(z,u,t)=yu+z^m+p(z,u)+t=0)\subset\A^4_{(x,y,z,u,t)}/\frac{1}{b+4}(c_1,...,c_5)\] where $(c_1,...,c_5)=(b+2,b,a,2,-2)$
	if $a$, $b$ are odd, and $(c_1,...,c_5)=(b+3,b+2,\frac{a}{2},1,-1)$ if $a$, $b$ are even. We take $(y_1,...,y_5)=(y,u,z,x,t)$ with
	$\delta_4=4$ and $\delta_5=1$ or $2$. Then ($\Xi$) holds. Now the origin of $U_t$ is a cyclic quotient point.
	Let $F$ be the exceptional divisor over this point, then $v_F(y_1,...,y_5)=\frac{1}{b+4}(a'_1,...,a'_5)$ with $a'_2+a'_4=b+4$.
	It follows that $a(a'_2+a'_4)>b+4=a_2+a_4$, hence ($\Theta_j$) holds for $j=2$ or $4$. Thus there exists
	$Y_1\rightarrow X$ so that $Y\ua{X}Y_1$. One has that \[ a(F,X)=\frac{1}{b+4}+\frac{a}{2}v_F(x)\leq\frac{2+ma}{2b+8}\leq\frac{1}{2}.\]
	Hence $Y_1\rightarrow X$ is a $w$-morphism.\par
	Assume that $Y\rightarrow X$ is of type D25. The chart $U_y\subset Y$ is given by
	\[ (x^2+yu+z^{4b}+g'(y,z,u)=0)\subset\A^4_{(x,y,z,u)}/\frac{1}{4b}(0,2b-1,1,2b+1).\]
	We take $(y_1,...,y_4)=(y,u,z,x)$ with $\delta_4=4$. One can see that ($\Xi'$) holds. The origin of $U_y$
	is a $cA/4b$ point and there is only one $w$-morphism over this point. Let $F$ be the exceptional divisor of the $w$-morphism,
	then $v_F(y_1,...,y_4)=\frac{1}{4b}(2b-1,2b+1,1,4b)$. Hence ($\Theta_2$) holds. One can also
	compute that $a(F,X)=\frac{1}{2}$, hence there exists a $w$-morphism $Y_1\rightarrow X$ such that $Y\ua{X}Y_1$.\par
	Assume that $Y\rightarrow X$ is of type D26. The chart $U_z\subset Y$ is a $cA/4$ point given by
	\[ (x^2+yu+y^4+z^{2b-4}+u^cz^{c-4}=0)\subset\A^4_{(x,y,z,u)}/\frac{1}{4}(0,1,1,3).\]
	We take $(y_1,...,y_4)=(y,u,y+z,x)$ with $\delta_4=4$. One can see that ($\Xi'$) holds. Now let $w$ be the weight
	such that $w(y_1,...,y_4)=\frac{1}{4}(1,3,1,4)$ if $b=4$ and $w(y_1,...,y_4)=\frac{1}{4}(1,7,1,4)$ if $b\geq 6$. Hence ($\Theta_2$) holds
	and there exists $Y_1\rightarrow X$ such that $Y\ua{X}Y_1$. One can compute that $a(F,X)=\frac{1}{2}$.\par
	Assume that $Y\rightarrow X$ is of type D27. Consider the chart $U_t\subset Y$ which is defined by
	\[ (x^2+u+y^4+z^4=yz+u^2+t=0)\subset\A^5_{(x,y,z,u,t)}/\frac{1}{6}(5,1,1,4,2).\] We take $(y_1,...,y_5)=(u,y,y+z,x,t)$ with $\delta_4=4$
	and $\delta_5=1$. In this case ($\Xi$) holds. Now the origin of $U_t$ is a cyclic quotient point. Let
	$F$ be the exceptional divisor over this point which corresponds to a $w$-morphism, then $v_F(y_1,...,y_5)=\frac{1}{6}(5,1,1,4,2)$.
	One can see that ($\Theta_1$) holds. Thus there exists $Y_1\rightarrow X$ which extracts $F$, so that
	$Y\ua{X}Y_1$. One can compute that $a(F,X)=\frac{1}{2}$.\par
	Finally assume that $Y\rightarrow X$ is of type D28 or D29. The chart $U_t\subset Y$ is defined by
	\[ (x^2+u+g'(y,z,u,t)=y^2+p(x,z,u)+t=0)\subset\A^4_{(x,y,z,u,t)}/\frac{1}{4b+2}(1,-1,a-2b-1,2,-2),\]
	where $a=2$ if case D28 and $a=4$ in case D29. We take $(y_1,...,y_5)=(x,y,z,u,t)$ with $\delta_4=4$ and $\delta_5=2$. Then ($\Xi$) holds.
	The origin of $U_t$ is a cyclic quotient point. Let $F$ be the exceptional divisor corresponds
	to the $w$-morphism over this point, then $v_F(y_1,...,y_5)=\frac{1}{4b+2}(a'_1,...,a'_5)$ with $a'_1+a'_2=4b+2$,
	$a'_2(2b+1-a)\equiv 1$ (mod $4b+2$) and $a'_3=1$. From the defining equation one can see that $a'_4>1$, hence ($\Theta_4$) holds.
	Thus there exists a divisorial contraction $Y_1\rightarrow X$ which extracts $F$, so that $Y\ua{X}Y_1$.
	We only need to show that $a(F,X)<\frac{a}{2}$.\par
	Assume that $Y\rightarrow X$ is of type D28. In this case $a'_2$ is the integer such that $a'_2(2b-1)\equiv 1$ (mod $4b+2$). If
	$b$ is odd, then $a'_2=b$ since \[ b(2b-1)=2b^2-b=(4b+2)\frac{b-1}{2}+1.\] One can see that $a'_5\leq 2b$.
	If $b$ is even, then $a'_2=3b+1$ since \[ (3b+1)(2b-1)=6b^2-b-1=(4b+2)(\frac{3}{2}b-1)+1.\] Hence $a'_1=b+1$. Now since
	$xz^{b-1}$ or $z^{2b}\in p(x,z,u)$, we also have that $a'_5\leq 2b$. In either cases we have
	\[a(F,X)=\frac{1}{4b+2}+\frac{a'_5}{4b+2}\leq\frac{2b+1}{4b+2}=\frac{1}{2}<1=\frac{a}{2}.\]\par
	Finally assume that $Y\rightarrow X$ is of type D29. We want to show that $a'_5<4b+2$. Then
	\[ a(F,X)=\frac{1}{4b+2}+2\frac{a'_5}{4b+2}\leq\frac{1}{4b+2}+\frac{8b+2}{4b+2}<2=\frac{a}{2}\] and we can finish the proof.
	If $z^b\in p(x,z,u)$, then $a'_5\leq b$. If $a'_2<2b+1$, then $a'_5\leq 4b$. Assume that $z^b\not\in p(x,z,u)$ and $a'_2\geq 2b+1$, 
	then $a'_1\leq 2b+1$ and $xz^{\frac{b-1}{2}}\in p(x,z,u)$. Hence \[ a'_5\leq 2b+1+\frac{b-1}{2}<4b+2.\]
\end{proof}

\subsection{Divisorial contractions to $cE$ points}
In this subsection we assume that $X$ has $cE$ singularities. First we study $w$-morphisms over $cE$ points. All $w$-morphisms over
$cE$ type points are listed in Table \ref{taE} and Table \ref{taE2}.
\begin{table}
\begin{tabular}{|c|c|c|c|c|}\hline
No. & defining equations  & weight & \tc{type}{\scriptsize{$a(X,E)$}} & condition  \\\hline
E1 & $x^2+y^3+g\gq4(y,z,u)$ & $(2,2,1,1)$ & \tc{$cE_6$}{$1$} & $\frac{\partial^2}{\partial y^2}g(y,z,u)=0$ \\ \hline
E2 & $x^2+xp_2(z,u)+y^3+g\gq5(y,z,u)$ & $(3,2,1,1)$ & \tc{$cE_{6,7}$}{$1$} & \\ \hline
E3 & $x^2+y^3+g\gq6(y,z,u)$ &  $(3,2,2,1)$ & \tc{$cE$}{$1$} & \\ \hline
E4 & $x^2+y^3+y^2p_2(z,u)+g\gq8(y,z,u)$  & $(4,3,2,1)$ & \tc{$cE$}{$1$} &  \\ \hline
E5 & $x^2+xp_4(y,z,u)+y^3+g\gq9(y,z,u)$  & $(5,3,2,1)$ & \tc{$cE$}{$1$} & \\ \hline
E6 & $x^2+y^3+y^2p_3(z,u)+g\gq{10}(y,z,u)$  &$(5,4,2,1)$ & \tc{$cE_{7,8}$}{$1$} & \\ \hline
E7 & $x^2+y^3+g\gq{12}(y,z,u)$  & $(6,4,3,1)$ & \tc{$cE$}{$1$} &  \\ \hline
E8 & $x^2+y^3+y^2p_4(z,u)+g\gq{14}(y,z,u)$  & $(7,5,3,1)$  & \tc{$cE_{7,8}$}{$1$} & \\ \hline
E9 & $x^2+xp_7(y,z,u)+y^3+g\gq{15}(y,z,u)$  & $(8,5,3,1)$ & \tc{$cE_{7,8}$}{$1$} & \\ \hline
E10 & $x^2+y^3+g\gq{18}(y,z,u)$  & $(9,6,4,1)$ & \tc{$cE_{7,8}$}{$1$} &  \\ \hline
E11 & $x^2+y^3+y^2p_6(z,u)+g\gq{20}(y,z,u)$  & $(10,7,4,1)$  & \tc{$cE_8$}{$1$} & \\ \hline
E12 & $x^2+y^3+g\gq{24}(y,z,u)$  &$(12,8,5,1)$  & \tc{$cE_8$}{$1$} & \\ \hline
E13 & $x^2+y^3+g\gq{30}(y,z,u)$  & $(15,10,6,1)$ & \tc{$cE_8$}{$1$} &  \\ \hline
\end{tabular}
\caption{Divisorial contractions to $cE$ points with discrepancy one}\label{taE}
\end{table}
\begin{table}
\begin{tabular}{|c|c|c|c|c|}\hline
No. & defining equations  & weight & \tc{type}{\scriptsize{$a(X,E)$}} & condition  \\\hline
E14 & $\sep{x^2+y^3+tz+g\gq6(y,z,u)}{p_4(x,y,z,u)+t}$  & $(3,2,1,1,5)$ & \tc{$cE_{6,7}$}{$1$} & $p(x,y,z,u)$ is irreducible \\ \hline
E15 & $x^2+xp_2(z,u)+y^3+g\gq6(x,y,z,u)$  & $(4,2,1,1)$ & \tc{$cE_6$}{$1$} &  \\ \hline
E16 & $\sep{x^2+y^3+tp_2(z,u)+g\gq6(y,z,u)}{q_3(y,z,u)+t}$  & $(3,2,1,1,4)$ & \tc{$cE_7$}{$1$} & $q(y,z,u)$ is irreducible \\ \hline
E17 & $x^2+y^3+yz^3+g\gq6(y,z,u)$  & $(3,3,1,1)$ & \tc{$cE_7$}{$1$} & $y^2u^2\in g$\\ \hline
E18 & $\sep{x^2+yt+g\gq{10}(y,z,u)}{y^2+p_6(y,z,u)+t}$ & $(5,3,2,1,7)$ & \tc{$cE_{7,8}$}{$1$} & $y^2+p(y,z,u)$ is irreducible \\ \hline
\end{tabular}
\caption{Divisorial contractions to $cE$ points with discrepancy one, continued}\label{taE2}
\end{table}

We assume that there exist two different $w$-morphisms over $X$, say $Y\rightarrow X$ and $Y_1\rightarrow X$. Let $F=exc(Y_1\rightarrow X)$.
Let $P=\cen_YF$. One always has that $a(F,Y)<1$, so $P$ is a non-Gorenstein point.

\begin{lem}\label{ep}
	Assume that both $Y\rightarrow X$ and $Y_1\rightarrow X$ are of type E1-E13. Then $Y\rightarrow X$ is not of type E1 or E6.
\end{lem}
\begin{proof}
	Assume that $Y\rightarrow X$ is of type E1, then the only non-Gorenstein point on $Y$ is the origin
	of \[ U_y=({x'}^2+{y'}^2+g'(z',u')=0)\subset\A^4_{(x',y',z',u')}/\frac{1}{2}(0,1,1,1).\] This is a $cAx/2$ point. The exceptional
	divisor $G$ of discrepancy less than one over this point is given by the weighted blow-up with the weight
	$w(x',y',z',u')=\frac{1}{2}(2,3,1,1)$. One can compute that $a(G,X)=2$, hence there are only one $w$-morphism over $X$. Thus
	$Y\rightarrow X$ is not of type E1.\par
	Assume that $Y\rightarrow X$ is of type E6. If $X$ has $cE_8$ singularities then there is only one non-Gorenstein point on $Y$,
	namely the origin of $U_y$. If $X$ has $cE_7$ singularities, then the origin of $U_z$ is also a non-Gorenstein point. Assume first
	that $P$ is the origin of $U_z$, then $P$ is a cyclic quotient point of index two and there is only one exceptional divisor
	over $P$ with discrepancy less than one. Hence $F$ should correspond to this exceptional divisor. One can compute that
	$v_F(x,y,z,u)=(3,3,1,1)$, so $Y_1\rightarrow X$ should be of type E17. Nevertheless, in this case one can see that
	$v_F(\sigma)\leq v_E(\sigma)$ for all $\sigma\in\Oo_X$. This contradict to Corollary \ref{wef}. Hence $P$ can not be the origin of $U_z$.\par
	We want to show that $P$ is also not the origin of $U_y$. The chart $U_y$ is defined by
	\[ ({x'}^2+y'(y'+p(z',u'))+g'(y',z',u')=0)\subset\A^4_{(x',y',z',u')}/\frac{1}{4}(1,3,2,1).\]
	The origin of $U_y$ is a $cAx/4$ point. Since $(u=0)$ defines a Du Val section, we know that $v_F(u)=v_F(y')+v_F(u')=1$.
	Hence both $v_F(y')$ and $v_F(u')<1$. This means that $v_F(y)\leq 3$. Assume that $Y\rightarrow X$ and $Y_1\rightarrow X$ corresponds
	to the same embedding $X\hookrightarrow \A^4$. Then since $v_F(y)\leq 3$, we know that $Y_1\rightarrow X$ is of type E1-E5.
	However, in those cases one always has that $v_F(\sigma)\leq v_E(\sigma)$ for all $\sigma\in\Oo_X$. This contradict to Corollary \ref{wef}.
	Thus $Y_1\rightarrow X$ corresponds to a different embedding.\par
	Let $Z\rightarrow Y$ be a $w$-morphism over the origin of $U_y$. From the classification we know that $Z\rightarrow Y$ is a
	weighted blow-up with the weight $w(x',y',z',u')=\frac{1}{4}(5,k,2,1)$ for $k=3$ or $7$. One can compute that non-Gorenstein points on $Z$
	over $U_y$ are cyclic quotient points. Let $\bar{Z}\rightarrow Z$ be a economic resolution over those cyclic quotient points, then
	$F$ appears on $\bar{Z}$ since $a(F,Y)<1$. Moreover, $\bar{Z}\rightarrow X$ can be viewed as a sequence of weighted blow-ups with
	respect to the embedding $X\hookrightarrow \A^4_{(x,y,z,u)}$. We write $(x_1,...,x_4)=(x,y,z,u)$ and let
	$X\hookrightarrow\A^4_{(x'_1,...,x'_4)}$ be the embedding corresponds to the weighted blow-up $Y_1\rightarrow X$.
	One can always assume that $x'_4=x_4=u$ since $v_E(u)=1$. We write $x'_j=x_j+q_j$. Since $Y\rightarrow X$ and $Y_1\rightarrow X$
	correspond to different embedding, there exists $j<4$ such that $q_j\neq 0$, and $v_F(x'_j)>v_F(x_j)=v_F(q)$. Since $\bar{Z}\rightarrow X$
	can be viewed as a sequence of weighted blow-ups with respect to the embedding $X\hookrightarrow \A^4_{(x,y,z,u)}$, we know that the
	defining equation of $\bar{Z}$ is of the form $x_j+q_j+\bar{h}$ such that $v_F(x'_j)=v_F(\bar{h})$. Hence there is exactly one $j$ such that
	$q_j\neq 0$, and the defining equation of $X$ is of the form $\xi(x_j+q_j)+h$. One can see that either $x_j=z$, or $x_j=y$ and $q_j=p$.\par
	Now if $x_j=z$, then $x'_1=x_1=x$ and $x'_2=x_2=y$. One can see that $v_F(x'_2)=v_F(y)\leq 3$. So $Y_1\rightarrow X$ is of type E1-E5.
	In those cases $v_F(x'_j)\leq 2$, so $v_F(x_j)=v_F(q_j)=1$ and $v_F(x'_j)=2$. Hence $Y_1\rightarrow X$ is of type E3-E5 and
	$v_F(y)=v_F(x'_2)\geq 2$. Also since $v_F(q_j)=1$, $q_j=\lambda u$ for some $\lambda\in\Cc$, hence $v_F(z)=v_F(x'_j-g_j)=1$.
	But then \[ \frac{v_F(z)}{v_E(z)}=\frac{1}{2}\leq\frac{v_F(y)}{v_E(y)}.\]
	By Lemma \ref{val}, $\cen_YF$ can not be the origin of $U_y$. This leads to a contradiction.\par
	Finally we assume that $x_j=y$ and $q_j=p$. Notice that $p=\lambda_1zu+\lambda_2u^3$, hence $v_F(p)\geq 2$. If $v_F(z)=v_F(x'_3)=1$,
	then $Y_1\rightarrow X$ is of type E1 or E2, and so $v_F(x'_j)=2$. However we know that $v_F(p)\geq 2$. This contradict to the
	assumption that $v_F(x'_j)>v_F(q_j)=v_F(p)$. Hence $v_F(z)\geq 2$ and so $v_F(p)\geq 3$. Since \[v_F(y)=v_F(x_j)=v_F(q_j)=v_F(p)\geq 3\]
	and $v_F(y)\leq 3$ by the previous discussion, we know that $v_F(y)=3$. Recall that we write
	\[ U_y=({x'}^2+y'(y'+p(z',u'))+g'(y',z',u')=0)\subset\A^4_{(x',y',z',u')}/\frac{1}{4}(1,3,2,1).\]
	Since $v_F(y)=3$, $v_F(E)=v_F(y')=\frac{3}{4}$. This means that $a(F,Y)=\frac{1}{4}$, so $F$ corresponds to a $w$-morphisms over $U_y$.
	Nevertheless, as we mentioned before, $w$-morphisms over $U_y$ can be obtained by a weighted blow-up with respect to the above embedding,
	hence $Y_1\rightarrow X$ and $Y\rightarrow X$ corresponds to the same four-dimensional embedding, this leads to a contradiction.
\end{proof}

\begin{lem}\label{epx}
	Assume that both $Y\rightarrow X$ and $Y_1\rightarrow X$ are of type E1-E13. If $P$ is the origin of $U_x\subset Y$, then
	$Y\rightarrow X$ is of type E2, E5 or E9, and $Y_1\rightarrow X$ has the same type. One has that $Y\uan{X}Y_1\uan{X}Y$.
\end{lem}
\begin{proof}
	The assumption implies that the origin of $U_x$ is contained in $Y$, so $Y\rightarrow X$ is of type E2, E5, E9 and $P$ is a cyclic
	quotient point. If $Y\rightarrow X$ is a weight blow-up with the weight $(b,c,d,1)$, then $b=c+d$ and 
	\[ U_x=(x'+p(z',u')+g'(x',y',z',u')=0)\subset \A^4_{(x',y',z',u')}/\frac{1}{b}(-1,c,d,1).\]
	Since $a(F,Y)<1$, $F$ is the valuation described by \cite[Proposition 3.1]{me}. Hence $v_F(y',z',u')=\frac{1}{b}(c',d',a')$ with
	$c'+d'=b$ and $\frac{a'}{b}=a(F,Y)$. Since $a(F,X)=a(F,Y)+v_F(x')=1$, we know that $v_F(x')=1-\frac{a'}{b}$. One can compute that
	\[ v_F(x,y,z,u)=(b-a',c-\frac{(a'c-c')}{b},d-\frac{a'd-d'}{b},1).\] Since $c'<b$ and $a'c\equiv c'$(mod $b$), we know that
	$\frac{a'c-c'}{b}\geq0$, so $v_F(y)\leq c=v_E(y)$. Likewise, we know that $v_F(z)\leq d=v_E(z)$. One also has that $v_F(x)<v_E(x)$ and
	$v_F(u)=v_E(u)$.\par
	On the other hand, Corollary \ref{wef} says that there exists $\sigma\in\Oo_X$ such that $v_F(\sigma)>v_E(\sigma)$. This can only happen
	when \[ v_F(x')=v_F(p(z',u'))<v_F(x'+p(z',u'))=v_F(g'(x',y',z',u'))\] and in this case one can choose $\sigma=x+p(z,u)\in\Oo_X$.
	Now $Y_1\rightarrow X$ can be obtained by a weighted blow-up with respect to the embedding
	\[ X\hookrightarrow (\sigma^2-\sigma p(z,u)+y^3+g(y,z,u)=0)\subset\A^4_{(\sigma,y,z,u)} \] and with the weight
	$w_1(\sigma,y,z,u)=(b_1,c_1,d_1,1)$ where $c_1=c-\frac{(a'c-c')}{b}$ and $d_1=d-\frac{a'd-d'}{b}$. Since $c_1\leq c$, $d_1\leq d$
	and $b_1=v_F(v)>v_E(v)$, by Lemma \ref{val} we know that $\cen_{Y_1}E$ is the origin of $U_{1,\sigma}$.\par
	Now if we interchange $Y$ and $Y_1$, then the above argument yields that $c\leq c_1$ and $d\leq d_1$. Hence $c=c_1$ and $d=d_1$ and
	so $Y\rightarrow X$ and $Y_1\rightarrow X$ are of the same type. One has that $a'=1$ and $c'=c$, $d'=d$. Thus $F$ is the exceptional
	divisor of the $w$-morphism over $P$. Now we know that $Y\uan{X}Y_1$ by Lemma \ref{yly1} and also $Y_1\uan{X}Y$ by the symmetry.
\end{proof}

\begin{lem}\label{epy}
	Assume that both $Y\rightarrow X$ and $Y_1\rightarrow X$ are of type E1-E13. Then $P$ is not the origin of $U_y\subset Y$.
\end{lem}
\begin{proof}
	By lemma \ref{ep} we know that $Y\rightarrow X$ is not of type E1 or E6, hence $Y\rightarrow X$ is of type E4, E8 or E11 and the origin
	of $U_y$ is a cyclic quotient point. We assume that $Y\rightarrow X$ is a weighted blow-up with the weight
	$(b,c,d,1)$. Following the same computation as in the proof of Lemma \ref{epx}, we may write $Y_1\rightarrow X$ as a weighted blow-up
	with respect to the embedding \[ X\hookrightarrow (x^2+(\sigma-p(z,u))^2\sigma+g(\sigma,z,u)=0)\subset\A^4_{(x,\sigma,z,u)}\]
	and with the weight $(b_1,c_1,d_1,1)$, and $\cen_{Y_1}E$ is the origin of $U_{1,\sigma}\subset Y_1$. Nevertheless, in this case one always
	has that $b_1<b$ since $b>c$. The symmetry between $Y$ and $Y_1$ yields that $b>b_1>b$, which is impossible.
\end{proof}
\begin{lem}\label{epij}
	Assume that both $Y\rightarrow X$ and $Y_1\rightarrow X$ are of type E1-E13. Let $X\hookrightarrow \A^4_{(x_1,...,x_4)}$
	be the embedding corresponds to $Y\rightarrow X$ in Table \ref{taE}. Then $P$ is the origin of $U_i$ for some $i\leq 4$.
\end{lem}
\begin{proof}
	Assume that $P$ is not the origin of $U_i$ for all $i\leq 4$. Then $Y$ has a non-Gorenstein point on $U_i\cap U_j$ for some $i\neq j$.
	In this case $Y\rightarrow X$ is of type E7 or E10-E13. For simplicity we assume that $i=1$ and $j=2$. If $Y\rightarrow X$ is a
	weighted blow-up with the weight $(a_1,...,a_4)$, then we have the following observation:
	\begin{enumerate}[(1)]
	\item $a_1=dk_1$ and $a_2=dk_2$ for some integers $k_1$, $k_2$ and $d$. We may assume that $k_2=2$ and $k_1$ is odd.
	\item $x_1^{k_2}$ and $x_2^{k_1}$ appear in $f$ where $f$ is the defining equation of $X$. Moreover $v_E(x_1^{k_2})=v_E(x_2^{k_1})=v_E(f)$.
	\item $P$ is a cyclic quotient point of index $d$. On $U_1$ the local coordinate system is given by $(x'_1,x'_3,x'_4)$, where $x'_l$
		is the strict transform of $x_l$ on $U_1$.
	\end{enumerate}
	 Since $F$ is a valuation of discrepancy less than one over $P$, we know that $v_F(x'_1,x'_3,x'_4)=\frac{1}{d}(a'_1,a'_3,a'_4)$
	 with $a'_l<d$ for $l=1$, $3$, and $4$. One can compute that
	 \[v_F(x_1,...,x_4)=(k_1a'_1,k_2a'_1,\frac{1}{d}(a'_3+a_3a'_1),\frac{1}{d}(a'_4+a_4a'_1)),\]
	 and $Y_1\rightarrow X$ can be obtained by the weighted blow-up with respect to the same embedding $X\hookrightarrow \A^4_{(x_1,...,x_4)}$
	 and with the weight $v_F$. Nevertheless, one can easily see that $v_F(x_l)\leq v_E(x_l)$ for all $1\leq l\leq 4$. This contradict to
	 Corollary \ref{wef}. 
\end{proof}

\begin{pro}\label{ce1}
	Assume that both $Y\rightarrow X$ and $Y_1\rightarrow X$ are both of type E1-E13. Then $Y\rightarrow X$ is of type E2, E5 or E9, and
	$Y_1\rightarrow X$ has the same type. One has that $Y\uan{X}Y_1\uan{X}Y$.
\end{pro}
\begin{proof}
	Let \[X\hookrightarrow (f(x,y,z,u)=0)\subset \A^4_{(x,y,z,u)}\] be the embedding corresponds to $Y\rightarrow X$, and
	\[ X\hookrightarrow (f_1(x_1,y_1,z_1,u_1)=0)\subset \A^4_{(x_1,y_1,z_1,u_1)}\]
	be the embedding corresponds to $Y_1\rightarrow X$. If $\cen_YF$ is the origin of $U_x\subset Y$, or $\cen_{Y_1}E$ is the origin of
	$U_{x_1}\subset Y_1$, then the statement follows from Lemma \ref{epx}. Assume that we are not in this case. Then Lemma \ref{epy} and
	Lemma \ref{epij} imply that $\cen_YF=U_z\subset Y$, and $\cen_{Y_1}E=U_{z_1}\subset Y_1$.\par
	We may assume that $v_E(f)\leq v_F(f_1)$. Since $v_E(u)=v_F(u_1)=1$, one can always assume that $u=u_1$ and $(u=0)$ defines a Du Val
	section. Lemma \ref{val} implies that $v_E(z)>v_E(z_1)$ and $v_F(z)<v_F(z_1)$, hence $z\neq z_1$. We may write $z_1=z+h$. If
	$v_E(h)\geq v_E(z)$, then we may replace $z$ by $z+h$, which will lead to a contradiction. Hence $v_E(h)<v_E(z)$. Thus $h=\lambda u^k$
	for some $k<v_E(z)$. Since $(u=0)$ defines a Du Val section, we know that $z^4$, $yz^3$ or $z^5\in f$. It follows that $u^{4k}$, $yu^{3k}$
	or $u^{5k}$ appear in either $f$ or $f_1$. This means that $v_E(f)\leq 4k$, $v_E(y)+3k$ or $5k$ for some $k<v_E(z)$.
	One can easily check that for all the cases in Table \ref{taE} this inequality never holds. Thus we get a contradiction.
\end{proof}

\begin{pro}\label{ce2}
	Assume that $Y\rightarrow X$ is of type E14-E18. Then there exists $Y_1\rightarrow X$ which is of type E3 or E6, such that $Y\ua{X}Y_1$.
	Moreover, if $Y\rightarrow X$ is of type E15 or E17, then $Y\uan{X}Y_1$.
\end{pro}
\begin{proof}
	Assume that $Y\rightarrow X$ is of type E14. Consider the chart
	\[U_t=({x'}^2+{y'}^3+z'+g'(y',z',u',t')=p(x',y',z',u')+t'=0)\subset\A^5_{(x',y',z',u',t')}/\frac{1}{5}(3,2,1,1,4).\]
	We choose $(y_1,...,y_5)=(x',y',u',z',t')$ with $\delta_4=1$ and $\delta_5=2$ or $4$, or $\delta_4=2$ and $\delta_5=1$ or $4$.
	Then ($\Xi$) holds. Now let $F$ be the exceptional divisor corresponds to the $w$-morphism over the origin
	of $U_t$, then \[v_F(y_1,...,y_5)=\frac{1}{5}(3,2,1,6,4).\] One can see that ($\Theta_4$) holds.
	Thus there exists a divisorial contraction $Y_1\rightarrow X$ so that $Y\ua{X}Y_1$ which extracts $F$. One can compute that
	$v_F(x,y,z,u)=(3,2,2,1)$, so $Y_1\rightarrow X$ is of type E3.\par
	Assume that $Y\rightarrow X$ is of type E15. Consider the chart
	\[ U_x=({x'}^2+p(z',u')+{y'}^3+g'(x',y',z',u')=0)\subset\A^4_{(x',y',z',u')}/\frac{1}{4}(3,2,1,1).\] We choose $(y_1,...,y_4)=(x',z',u',y')$
	with $\delta_4=4$. Then ($\Xi'$) holds. Now the origin of $U_x$ is a $cAx/4$ point. After a suitable
	change of coordinates we may assume that ${u'}^2\not\in p(z',u')$. Then the $w$-morphism over this point can be given by
	a weighted blow-up with the weight $v_F(y_1,...,y_4)=\frac{1}{4}(3,5,1,2)$. One can see that ($\Theta_2$) and ($\Xi'_-$) hold.
	Hence there exists a divisorial contraction $Y_1\rightarrow X$ such that $Y\uan{X}Y_1$. One also has $v_F(x,y,z,u)=(3,2,2,1)$,
	so $Y_1\rightarrow X$ is of type E3.\par
	Assume that $Y\rightarrow X$ is of type E16. Consider the chart
	\[U_t=({x'}^2+{y'}^3+p(z',u')+g'(y',z',u',t')=q(y',z',u')+t'=0)\subset\A^5_{(x',y',z',u',t')}/\frac{1}{4}(3,2,1,1,3).\]
	We choose $(y_1,...,y_5)=(y',z',u',x',t')$ with $\delta_4=4$ and $\delta_5=1$ or $2$. Then ($\Xi$) holds.
	Now the origin of $U_t$ is a $cAx/4$ point. After a suitable change of coordinates we may assume that ${u'}^2\not\in p(z',u')$.
	Then the weight $v_F(y_1,...,y_5)=\frac{1}{4}(2,5,1,3,3)$ defines a $w$-morphism over $U_t$. One can see that ($\Theta_2$) holds.
	Hence there exists a divisorial contraction $Y_1\rightarrow X$ such that $Y\ua{X}Y_1$. One can
	compute that $v_F(x,y,z,u)=(3,2,2,1)$, so again $Y_1\rightarrow X$ is of type E3.\par
	Assume that $Y\rightarrow X$ is of type E17. Consider the chart
	\[ U_y=({x'}^2+{y'}^3+{z'}^3+g'(y',z',u')=0)\subset\A^4_{(x',y',z',u')}/\frac{1}{3}(0,2,1,1).\]
	We can choose $(y_1,...,y_4)=(y',z',u',x')$ with $\delta_4=4$. Then ($\Xi'$) holds.
	The origin of $U_y$ is a $cD/3$ point. Notice that ${y'}^2{u'}^2\in g'(y',z',u')$, so the $w$-morphism over $U_y$ is given by
	weighted blowing-up the weight $v_F(y_1,...,y_4)=\frac{1}{3}(2,4,1,3)$. One can see that ($\Theta_2$) and ($\Xi'_-$) hold.
	One has that $v_F(x,y,z,u)=(3,2,1,1)$, so there exists $Y_1\rightarrow X$ which is of type E3, such that $Y\uan{X}Y_1$.\par
	Finally assume that $Y\rightarrow X$ is of type E18. Consider the chart
	\[U_t=({x'}^2+y'+g'(y',z',u',t')={y'}^2+p(y',z',u')+t=0)\subset\A^5_{(x',y',z',u',t')}/\frac{1}{7}(5,3,2,1,6).\]
	We choose $(y_1,...,y_5)=(y',z',u',x',t')$ with $\delta_4=4$ and $\delta_5=1$. Then ($\Xi$) holds.
	The origin of $U_t$ is a cyclic quotient point. Let $F$ be the exceptional divisor corresponds to the $w$-morphism over this point,
	then \[v_F(y_1,...,y_5)=\frac{1}{7}(10,2,1,5,6)\] (notice that the irreducibility of $y^2+p(y,z,u)$ implies that $p(0,z,u)\neq 0$, so
	$v_F(t)=\frac{6}{7}$). One can see that ($\Theta_1$) holds. Thus there exists a divisorial contraction
	$Y_1\rightarrow X$ which extracts $F$, so that $Y\ua{X}Y_1$. One can compute that $v_F(x,y,z,u)=(5,4,2,1)$, so $Y_1\rightarrow X$ is of
	type E6. 
\end{proof}

Now we study divisorial contractions over $cE$ points with discrepancy greater than one. Those divisorial contractions are given in
Table \ref{taE3}.

\begin{table}
\begin{tabular}{|c|c|c|c|c|}\hline
No. & defining equations  & {weight} & \tc{type}{\scriptsize{$a(X,E)$}} & condition  \\\hline
E19 & \tcc{$x^2+(y+p_2(z,u))^3+$}{$yu^3+g\gq 6(z,u)$} & $(3,3,2,1)$ & \tc{$cE_6$}{$2$} & $z\in p(z,u)$ \\\hline
E20 & $\sep{x^2+yt+g\gq{10}(y,z,u)}{y^2+p_6(z,u)+t}$ & $(5,3,2,2,7)$ & \tc{$cE_7$}{$2$} & $gcd(p_6,g_{10})=1$\\\hline
E21 & $x^2+y^3+u^7+g\gq{14}(z,u)$ & $(7,5,3,2)$ & \tc{$cE_{7,8}$}{$2$} & \tcc{$yz^3$, $z^5$ or}{$z^4u\in g(z,u)$} \\\hline
\end{tabular}
\caption{Divisorial contractions to $cE$ points with large disprepancies} \label{taE3}
\end{table}

\begin{pro}\label{ce3}
	Assume that $Y\rightarrow X$ is a divisorial contraction with discrepancy $a>1$. Then there exists a $w$-morphism $Y_1\rightarrow X$,
	such that $Y\ua{X}Y_1$.
\end{pro}
\begin{proof}
	Assume first that $Y\rightarrow X$ is of type E19. The chart $U_y\subset Y$ is defined by
	\[ (x^2+(y+p(z,u))^3+u^3+g'(y,z,u)=0)\subset\A^4_{(x,y,z,u)}/\frac{1}{3}(0,1,1,2).\] One can choose $(y_1,...,y_4)=(x,y+p,z,u)$
	with $\delta_4=1$. Then ($\Xi'$) holds. The origin of $U_y$ is a $cD/3$ point. The $w$-morphism over this point
	is given by weighted blowing-up the weight \[ w(y_1,...,y_4)=(3,4,1,2)\mbox{ or }(6,4,1,5).\] One can see that ($\Theta_4$) holds.
	Thus there exists $Y_1\rightarrow X$ such that $Y\ua{X}Y_1$. A direct computation shows that
	$Y_1\rightarrow X$ is a $w$-morphism.\par
	Assume that $Y\rightarrow X$ is of type E20. The chart $U_t\subset Y$ is defined by
	\[ (x^2+y+g'(y,z,u,t)=y^2+p(z,u)+t=0)\subset\A^5_{(x,y,z,u,t)}/\frac{1}{7}(5,3,2,2,6).\]
	We take $(y_1,...,y_5)=(y,z,u,x,t)$ with $\delta_4=4$ and $\delta_5=1$. Then ($\Xi$) holds. The $w$-morphism
	over the origin of $U_t$ is given by weighted blowing-up the weight $w(y_1,...,y_5)=\frac{1}{7}(5,1,1,6,3)$ One can see that 
	($\Theta_1$) holds. Hence there exists a divisorial contraction $Y_1\rightarrow X$ such that $Y\ua{X}Y_1$.
	One can compute that $Y_1\rightarrow X$ is a $w$-morphism.\par
	Finally assume that $Y\rightarrow X$ is of type E21. The chart $U_y\subset Y$ is defined by
	\[ (x^2+y+u^7+g'(y,z,u)=0)\subset\A^4_{(x,y,z,u)}/\frac{1}{5}(2,4,3,2).\] One can choose $(y_1,...,y_4)=(y,z,u,x)$. Then
	($\Xi'$) holds. The $w$-morphism over $U_y$ is given by weighted blowing-up the weight
	$w(y_1,...,y_4)=\frac{1}{5}(2,4,1,1)$. One can see that ($\Theta_2$) holds. Thus there exists a divisorial contraction 
	$Y_1\rightarrow X$ such that $Y\ua{X}Y_1$. One can compute that $Y_1\rightarrow X$ is a $w$-morphism.
\end{proof}

\subsection*{Divisorial contractions to $cE/2$ points}

Finally we need to study divisorial contractions over $cE/2$ point. All such divisorial contractions are listed in Table \ref{taE4}
\begin{table}
\begin{tabular}{|c|c|c|c|}\hline
No. & defining equations  & \tc{$(r;a_i)$}{weight} & \tc{type}{\scriptsize{$a(X,E)$}} \\\hline
E22 & $x^2+y^3+g\gq 3(y,z,u)$ & \tc{$(2;1,0,1,1)$}{$\frac{1}{2}(3,2,3,1)$} & \tc{$cE/2$}{$1/2$} \\\hline
E23 & $x^2+y^3+g\gq 5(y,z,u)$ & \tc{$(2;1,0,1,1)$}{$\frac{1}{2}(5,4,3,1)$} & \tc{$cE/2$}{$1/2$} \\\hline
E24 & $x^2+xp_{\frac{5}{2}}(y,z,u)+y^3+g\gq 6(y,z,u)$ & \tc{$(2;1,0,1,1)$}{$\frac{1}{2}(7,4,3,1)$} & \tc{$cE/2$}{$1/2$} \\\hline
E25 & $x^2+y^3+g\gq 9(y,z,u)$ & \tc{$(2;1,0,1,1)$}{$\frac{1}{2}(9,6,5,1)$} & \tc{$cE/2$}{$1/2$} \\\hline
E26 & $x^2+y^3+z^4+u^8+g\gq8(y,z,u)$ & \tc{$(2;1,0,1,1)$}{$(4,3,2,1)$} & \tc{$cE/2$}{$1$} \\\hline
\end{tabular}
\caption{Divisorial contractions to $cE/2$ points}\label{taE4}
\end{table}

\begin{pro}\label{ce4}
	Let $Y\rightarrow X$ be a divisorial contraction.
	\begin{enumerate}[(1)]
	\item Assume that there are two $w$-morphisms over $X$, then
		\begin{enumerate}[({1}-1)]
		\item If $Y\rightarrow X$ is of type E22 and assume that there exists another $w$-morphism $Y_1\rightarrow X$. Then
			$Y_1\rightarrow X$ is of type E22 or E23, and $Y\ua{X}Y_1$.
		\item If $Y\rightarrow X$ is of type E23, then there are exactly two $w$-morphisms. The other one $Y_1\rightarrow X$ is of type E22.
			Interchange $Y$ and $Y_1$ we can back to Case (1-1).
		\item If $Y\rightarrow X$ is of type E24, then there are exactly two $w$-morphisms. They are both of type E24 and
			are negatively linked to each other.
		\item $Y\rightarrow X$ is not of type E25.
		\end{enumerate}
	\item Assume that $Y\rightarrow X$ is of type E26, then there is a $w$-morphism $Y_1\rightarrow X$ which is of type E22 such that
		$Y\uan{X}Y_1$.
	\end{enumerate}
\end{pro}
\begin{proof}
	The statement about the number of $w$-morphisms follows from \cite[Section 10]{h1}. First assume that there exists two $w$-morphisms
	over $X$ and $Y\rightarrow X$ is of type E22. Let $F=exc(Y_1\rightarrow X)$. The only non-Gorenstein point on $Y$ is the origin of $U_z$,
	which is a $cD/3$ point defined by \[ ({x'}^2+{y'}^3+g'(y',z',u')=0)\subset\A^4_{(x',y',z',u')}/\frac{1}{3}(0,2,1,1).\]
	One can see that \[\frac{1}{2}=a(F,X)=a(F,Y)+\frac{1}{2}v_F(z'),\] hence $a(F,Y)=v_F(z')=\frac{1}{3}$. Thus $F$ corresponds to a 
	$w$-morphism over $U_z$. From Table \ref{taD3} we know that \[v_F(x',y',z',u'+\lambda z')=\frac{1}{3}(b,c,1,4)\]
	for some $\lambda\in\Cc$, where $(b,c)=(3,2)$ or $(6,5)$. Now one can choose $(y_1,...,y_4)=(y',u'+\lambda z',u'+\xi z',x')$ with
	$\delta_4=4$, where $\xi\in\Cc$ is a number so that $u+\xi z$ defines a Du Val section on $X$, and $\xi\neq \lambda$. Thus the 
	($\Xi'$) and ($\Theta_2$) hold and $Y\ua{X}Y_1$.\par
	Now assume that $Y\rightarrow X$ is of type E24. The chart $U_x\subset Y$ is defined by
	\[ (x'+p(y',z',u')+{y'}^3+g'(x',y',z',u')=0)\subset\A^4_{(x',y',z',u')}/\frac{1}{7}(5,4,3,1).\]
	The origin is a cyclic quotient point and there is only one $w$-morphism over this point. Let $F$ be the exceptional divisor corresponds
	to this $w$-morphism. By Lemma \ref{yly1} we know that there exists a divisorial contraction $Y_1\rightarrow X$ which extracts $F$,
	such that $Y\uan{X}Y_1$. One can compute that $a(F,X)=\frac{1}{2}$. Hence $Y_1\rightarrow X$ is also a $w$-morphism.\par 
	Finally assume that $Y\rightarrow X$ is of type E26. Consider the chart $U_y\subset Y$ which is defined by
	\[ ({x'}^2+y'+{z'}^4+{u'}^8+g'(y',z',u')=0)\subset\A^4_{(x',y',z',u')}/\frac{1}{6}(1,2,5,1).\]
	One can take $(y_1,...,y_4)=(y',z',u',x')$ with $\delta_4=4$. One can see that ($\Xi'$) holds. Now let $F$ 
	be the exceptional divisor corresponds to the $w$-morphism over $U_t$, then $v_F(y_1,...,y_4)=\frac{1}{6}(2,5,1,1)$. Hence
	($\Theta_2$) and ($\Xi'_-$) hold. Thus there exists a divisorial contraction $Y_1\rightarrow X$ which extracts $F$, so that
	$Y\uan{X}Y_1$. One can compute that $v_F(x,y,z,u)=\frac{1}{2}(3,2,3,1)$, so $Y_1\rightarrow X$ is of type E22. 
\end{proof}

\section{Estimating depths}\label{sdep}

We want to understand the change of singularities after running the minimal model program. The final result is the following proposition: 

\begin{pro}\label{ied}$ $
	\begin{enumerate}[(1)]
	\item Assume that $Y\rightarrow X$ is a divisorial contraction between terminal and $\Q$-factorial threefolds.
		\begin{enumerate}[({1}-1)]
		\item If $Y\rightarrow X$ is a divisorial contraction to a point, then
			\[gdep(X)\leq gdep(Y)+1\mbox{ and }dep(X)\leq dep(Y)+1.\]
			If $Y\rightarrow X$ is a divisorial contraction to a curve, then
			\[gdep(X)\leq gdep(Y)\mbox{ and }dep(X)\leq dep(Y).\]
		\item $\gd(X)\geq \gd(Y)$ and the inequality is strict if the non-isomorphic locus on $X$ contains a Gorenstein singular point.
		\end{enumerate}			
	 \item Assume that $X\dashrightarrow X'$ is a flip between terminal and $\Q$-factorial threefolds. 
	 	\begin{enumerate}[({2}-1)]
	 	\item \[ gdep(X)>gdep(X')\mbox{ and }dep(X)>dep(X').\]
	 	\item $\gd(X)\leq\gd(X')$ and the inequality is strict if the non-isomorphic locus on $X'$ contains a Gorenstein singular point.
	 	\end{enumerate}
	\end{enumerate}
\end{pro}
\begin{cor}\label{feam}
	Assume that \[ X_0\dashrightarrow X_1\dashrightarrow...\dashrightarrow X_k\] is a process of the minimal model program.
	Then \begin{enumerate}[(1)]
	\item $\rho(X_0/X_k)\geq gdep(X_k)-gdep(X_0)$ and the equality holds if and only if $X_i\dashrightarrow X_{i+1}$ is a strict $w$-morphism
		for all $i$.
	\item $\gd(X_k)\geq \gd(X_0)$.
	\end{enumerate}
	In particular, if $X$ is a terminal $\Q$-factorial threefold and $W\rightarrow X$ is a resolution of singularities. Then
	$\rho(W/X)\geq gdep(X)$ and the equality holds if and only if $W$ is a feasible resolution of $X$.
\end{cor}
\begin{proof}
	The statement (2) easily follows from inequalities in Proposition \ref{ied}. Assume that the sequence
	contains $m$ flips, then $\rho(X_0/X_k)=k-m$. On the other hand, we know that
	\[ gdep(X_{i+1})\leq\left\lbrace \begin{array}{ll} gdep(X_i)-1 & \mbox{if $X_i\dashrightarrow X_{i+1}$ is a flip} \\
		gdep(X_i)+1 & \mbox{otherwise}\end{array}\right..\]
	It follows that $gdep(X_k)\leq gdep(X_0)+k-2m$, hence $gdep(X_k)-gdep(X_0)\leq \rho(X_0/X_k)$.
	Now $gdep(X_k)-gdep(X_0)=\rho(X_0/X_k)$ if and only if $m=0$ and $gdep(X_{i+1})=gdep(X_i)+1$ for all $i$. Which is equivalent to
	that $X_i\dashrightarrow X_{i+1}$ is a strict $w$-morphism for all $i$.\par
	Now assume that $X$ is a terminal $\Q$-factorial threefold and $W\rightarrow X$ is a resolution of singularities.
	We can run $K_W$-MMP over $X$ and the minimal model is $X$ itself. Since $gdep(W)=0$, one has that
	$\rho(W/X)\geq gdep(X)$ and the equality holds if and only if $W$ is a feasible resolution of $X$.
\end{proof}

The inequalities for the depth part is exactly Lemma \ref{chfp}. We only need to prove the inequalities for the generalized depth and
the Gorenstein depth. 

\begin{cov}
	Let $\Ss$ be a set consists of birational maps between $\Q$-factorial terminal threefolds.
	We say that $\aast_{\Ss}$ holds if for all $Z\dashrightarrow V$ inside $\Ss$ one has that
	\begin{enumerate}[(1)]
	\item If $Z\rightarrow V$ is a divisorial contraction to a point, then
		\[ \gd(V)\geq \gd(Z)\geq \gd(V)-(dep(Z)-dep(V)+1).\]	
	\item If $Z\rightarrow V$ is a divisorial contraction to a smooth curve, then
		\[ \gd(V)\geq \gd(Z)\geq \gd(V)-(dep(Z)-dep(V)).\]
	\item If $Z\dashrightarrow V$ is a flip, then
		\[ \gd(V)\geq \gd(Z)\geq \gd(V)-(dep(Z)-dep(V)-1).\]
	\item If $Z\dashrightarrow V$ is a flop, then $\gd(V)=\gd(Z)$.
	\end{enumerate}
	Moreover, if there exists a Gorenstein singular point $P\in V$ such that $P$ is not contained in the isomorphic locus of
	$Z\dashrightarrow V$, then $\gd(V)>\gd(Z)$ unless $V\dashrightarrow Z$ is a flop.\par
	We say that $\aast^{(1)}_{\Ss}$ holds if statement (1) is true, but statement (2), (3) is unknown.\par
	If $V\dashrightarrow Z$ is a flip or a divisorial contraction. We denote the condition $\aast_{V\dashrightarrow Z}=\aast_{\Ss}$
	where $\Ss$ is the set contains only one element $V\dashrightarrow Z$.
\end{cov}
It is easy to see that if $Y\rightarrow X$ is a divisorial contraction, then $\aast_{Y\rightarrow X}$ holds if and only if the inequalities in
Proposition \ref{ied} (1) hold. Likewise, if $X\dashrightarrow X'$ is a flip, then $\aast_{X\dashrightarrow X'}$ holds if and only if
the inequalities in Proposition \ref{ied} (2) hold.

\begin{rk}
	If $Z\dashrightarrow V$ is a flop, then the singularities on $Z$ and $V$ are the same by \cite[Theorem 2.4]{ko2}. Hence the statement (4)
	is always true.
\end{rk}

\begin{cov}
	Given $n\in\Z_{\geq0}$. We denote
	\[ \Ss_n=\se{\phi:Z\dashrightarrow V}{\tcc{$\phi$ is a flip, a flop or a divisorial contraction between} 
		{between $\Q$-factorizal terminal threefolds, $gdep(Z)\leq n$}}.\]  
\end{cov}

\begin{lem}\label{pgd}
	Assume that $Y\rightarrow X$ and $Y_1\rightarrow X$ are two divisorial contractions between terminal threefolds, such that
	$Y\ua{X}Y_1$. If $\aast_{\Ss_{gdep(Y)-1}}$ holds, then $gdep(Y_1)\leq gdep(Y)$. Moreover, $gdep(Y_1)=gdep(Y)$ if and only if $Y_1\ua{X}Y$
\end{lem}
\begin{proof}
	We have a diagram
	\[ \vc{\xymatrix{ Z_1\ar@{-->}[r]\ar[d] & ... \ar@{-->}[r] & Z_k \ar[d] \\ Y \ar[rd] & & Y_1\ar[ld] \\ & X & }}\]
	such that $Z_1\rightarrow Y$ is a strict $w$-morphism and $Z_i\dashrightarrow Z_{i+1}$ is a flip or a flop for all
	$1\leq i\leq k-1$. Since $gdep(Z_1)=gdep(Y)-1$ and $\aast_{\Ss_{gdep(Y)-1}}$ holds, we know that $gdep(Z_2)\leq gdep(Y)-1$. Repeat this
	argument $k-2$ times one can say that $gdep(Z_k)\leq gdep(Y)-1$. Again since $\aast_{\Ss_{gdep(Y)-1}}$ holds we know that
	$gdep(Y_1)\leq gdep(Z_k)+1=gdep(Y)$.\par
	Now $gdep(Y_1)=gdep(Y)$ if and only if all the inequalities above are equalities. It is equivalent to $Z_k\rightarrow Y_1$ is a strict
	$w$-morphism and $Z_i\dashrightarrow Z_{i+1}$ is a flop for all $i=1$, ..., $k-1$, or $k=1$. In other word, we also have $Y_1\ua{X}Y$.
\end{proof}
\begin{cor}\label{yua}
	Assume that $Y\rightarrow X$ is a a strict $w$-morphism over $P\in X$ and $Y_1\rightarrow X$ is another divisorial contraction over $P$.
	If $\aast_{\Ss_{gdep(Y)}}$ holds, then there exists divisorial contractions $Y_1\rightarrow X$, ..., $Y_k\rightarrow X$ such that
	\[Y_1\ua{X}...\ua{X}Y_k\ua{X}Y.\]
\end{cor}
\begin{proof}
	By Proposition \ref{dpf} we know that there exists $Y_1$, ..., $Y_l$, $Y'_1$, ..., $Y'_{l'}$ such that 
	\[	Y_1\ua{X}...\ua{X}Y_l=Y'_{l'}\ual{X}...\ual{X}Y'_1\ual{X}Y.\]
	One can apply Lemma \ref{pgd} to the sequence $Y\ua{X}Y'_1\ua{X}...\ua{X}Y'_{l'}$ and conclude that $gdep(Y'_i)\leq gdep(Y)$ for all $i$.
	Since $Y'_i\rightarrow X\in\Ss_{gdep(Y)}$ for all $i=1$, ..., $l'$, one has $gdep(Y'_i)\geq gdep(X)-1=gdep(Y)$.
	Thus $gdep(Y'_i)=gdep(Y)$ and Lemma \ref{pgd} says that one has \[Y'_{l'}\ua{X}...\ua{X}Y'_1\ua{X}Y.\]
	Now we can take $k=l+l'-1$ and let $Y_i=Y'_{l'-i+l}$ for $l<i\leq k$.
\end{proof}
\begin{cor}\label{cawmo}
	Assume that $\aast_{\Ss_{n-2}}$ holds. Assume that $P\in X$ is a $cA/r$ point or a $cAx/r$ point such that $gdep(X)=n$.
	Then every $w$-morphism over $P$ is a strict $w$-morphism. In particular, one can always assume that the morphism $Y_1\rightarrow X$ in
	Theorem \ref{chf} is a strict $w$-morphism.
\end{cor}
\begin{proof}
	Proposition \ref{ca} and Proposition \ref{cax} says that if $Y\rightarrow X$ and $Y_1\rightarrow X$ are two different $w$-morphisms
	over $P$, then there exists $Y_2\rightarrow X$, ..., $Y_k\rightarrow X$ such that
	\[Y_1\underset{X}{\Leftrightarrow}Y_2\underset{X}{\Leftrightarrow}...\underset{X}{\Leftrightarrow}Y_k=Y.\]
	We can assume that $Y\rightarrow X$ is a strict $w$-morphism, so $gdep(Y)=n-1$.	Lemma \ref{pgd} implies that $Y_1$ is also a strict
	$w$-morphism. Hence every $w$-morphism over $P$ is a strict $w$-morphism.\par
	Now assume that $X$ is in the diagram in Theorem \ref{chf} and $P$ is a non-Gorenstein point in the exceptional set of $X\rightarrow W$.
	If $P$ is a $cA/r$ or a $cAx/r$ point, then we already know that every $w$-morphism over $P$ is a strict $w$-morphism.  
	Otherwise by Remark \ref{chfr} (2) we know that any $w$-morphism $Y_1\rightarrow X$ over $P$
	induces a diagram in Theorem \ref{chf}. Hence we can choose $Y_1\rightarrow X$ to be a strict $w$-morphism. 
\end{proof}

\begin{cov}
	Let $DV$ be the set of following symbols
	\[ DV=\{A_i,D_j,E_k\}_{i\in\N,j\in\N_{\geq 4},k=6,7,8}.\]
	One can define an ordering on $DV$ by
	\[ A_i<A_{i'}<D_j<D_{j'}<E_k<E_{k'}\mbox{ for all }i<i',j<j',k<k'.\]\par
	Given $\square\in DV$, define
	\[\T_{\square}=\se{\tcc{$X$ is a terminal}{$\Q$-factorial threefold}}
		{\tcc{$GE(P\in X)\leq\square$ for all}{non-Gorenstein point $P\in X$}},\]
	\[ \T_{\square,n}=\se{X\in\T_{\square}}{gdep(X)\leq n} \] and \[\T_n=\bigcup_{\square\in DV}\T_{\square,n}.\]
	Here $GE(P\in X)$ denotes the type of the general elephant near $P$. That is, the type of a general Du Val section $H\in|-K_X|$ near
	an analytic neighborhood of $P\in X$.
\end{cov}
\begin{cov}
	Let $\T$ be a set of terminal threefolds. We say that the condition $(\Pi)_{\T}$ holds if for all
	$X\in \T$ and for all $Y\rightarrow X$ which is a strict $w$-morphism over a non-Gorenstein point of $X$, one has that
	$dep(Y)=dep(X)-1$.
\end{cov}
\begin{rk}\label{rge}$ $
	\begin{enumerate}[(1)]
	\item Assume that $Y\rightarrow X$ is a $w$-morphism over a non-Gorenstein point $P$, then the general elephant of $Y$ over $X$ is better
		than the general elephant of $X$ near $P$. It is because that if $H\in|-K_X|$ near $P$, then $H_Y\in|-K_Y|$ and
		$H_Y\rightarrow H$ is a partial resolution by \cite[Lemma 2.7]{ch}.
	\item One has that $(\Pi)_{\T_{A_1}}$ always holds since if $GE(P\in X)=A_1$ for some non-Gorenstein point $P$, then $P$ is a cyclic
		quotient point of index two (cf. \cite[(6.4)]{re}). In this case there is only one $w$-morphism $Y\rightarrow X$ over $P$ and
		$Y$ is smooth over $X$. Hence \[gdep(P\in X)=dep(P\in X)=1\] and $gdep(Y)=dep(Y)=0$ over $X$. Thus $Y\rightarrow X$ is 
		a strict $w$-morphism and also $dep(Y)=dep(X)-1$.
	\end{enumerate}
\end{rk}

\begin{lem}\label{depy}
	Assume that $Y\rightarrow X$ is a strict $w$-morphism over a non-Gorenstein point $P$. If $\aast_{\Ss_{gdep(Y)}}$ holds
	and $(\Pi)_{\T_{\square}}$ holds for all $\square<GE(P\in X)$, then $dep(Y)=dep(X)-1$.
\end{lem}
\begin{proof}
	By the definition we know that $dep(Y)\geq dep(X)-1$. Assume that $dep(Y)>dep(X)-1$. Then there exists $Y_1\rightarrow X$ such that
	$dep(Y_1)=dep(X)-1<dep(Y)$. Corollary \ref{yua} says that there exist divisorial contractions $Y_2\rightarrow X$, ..., $Y_k\rightarrow X$
	such that $Y_1\ua{X}Y_2\ua{X}...\ua{X}Y_k\ua{X}Y$. From Remark \ref{rge} (1) we know that $Y_i\in\T_{\square}$ for some $\square<GE(P\in X)$
	for all $i$, hence if
	\[ \vc{\xymatrix{ Z_{i,1}\ar@{-->}[r]\ar[d] & ... \ar@{-->}[r] & Z_{i,k_i} \ar[d] \\ Y_i \ar[rd] & & Y_{i+1}\ar[ld] \\ & X & }}\]
	is the induced diagram of $Y_i\ua{X}Y_{i+1}$, then $dep(Z_{i,1})=dep(Y_i)-1$. By Lemma \ref{chfp} we know that
	\[ dep(Y_{i+1})\leq dep(Z_{i,k_i})+1\leq dep(Z_{i,1})+1=dep(Y_i)\] for all $i$. Hence $dep(Y_1)\leq dep(Y)$.
	This leads to a contradiction for we assume that $dep(Y)>dep(Y_1)$.
\end{proof}
\begin{lem}\label{pitn}
	Fix an integer $n$. Assume that $\aast_{\Ss_{n-1}}$ holds, then $(\Pi)_{\T_n}$ holds.
\end{lem}
\begin{proof}
	We need to show that for all $X\in\T_n$ and for all strict $w$-morphism $Y\rightarrow X$ over a non-Gorenstein point, one has that
	$dep(Y)=dep(X)-1$. By Remark \ref{rge} (2) we know that $(\Pi)_{\T_{A_1}}$ holds, hence $(\Pi)_{\T_{A_m,n}}$ as well as $(\Pi)_{T_{D_4,n}}$
	holds for all $m\in\N$ by Lemma \ref{depy} and by induction on $m$. Then one can prove that $(\Pi)_{T_{D_m,n}}$ as well as
	$(\Pi)_{T_{E_6,n}}$ holds by again applying Lemma \ref{depy} and by induction on $m$. The statement $(\Pi)_{T_{E_7},n}$ and
	$(\Pi)_{T_{E_8},n}$ can be proved by the same way.
\end{proof}

\begin{lem}\label{laast}
	Fix an integer $n$ and assume that $\aast_{\Ss_{n-1}}$ holds. Assume that we have a diagram
	\[\vc{\xymatrix{Y_1\ar[d] \ar@{-->}[r] & \cdots \ar@{-->}[r] & Y_k\ar[d]\\
		X\ar[rd] &  & X'\ar[ld] \\ & W &	 }}\]
	such that $gdep(X)=n$, $Y_1\rightarrow X$ is a strict $w$-morphism, $Y_i\dashrightarrow Y_{i+1}$ is a flip or a flop for
	$i=1$, ..., $k-1$, and $Y_k\rightarrow X'$ is a divisorial contraction. Then 
	\begin{enumerate}[(1)]
	\item $\gd(X)=\gd(Y_1)\leq\gd(X')$ and $gdep(X')\leq gdep(X)$.
	\item $\aast_{Y_i\dashrightarrow Y_{i+1}}$ holds for $i=1$, ..., $k-1$.
	\item $\aast_{Y_k\rightarrow X'}$ holds.
	\end{enumerate}
	Moreover, if the non-isomorphic locus of $X\dashrightarrow X'$ on $X'$ contains a Gorenstein singular point, then $\gd(X)<\gd(X')$.
\end{lem}
\begin{proof}
	Since $\aast_{\Ss_{n-1}}$ holds, we know that $(\Pi)_{\T_n}$ holds by Lemma \ref{pitn}. Hence $\gd(Y_1)=\gd(X)$.
	We know that $gdep(Y)=n-1$. Since $\aast_{\Ss_{n-1}}$ holds, one can prove that $gdep(Y_i)\leq gdep(Y_1)\leq n-1$ for all $i$ and
	hence $\aast_{Y_i\dashrightarrow Y_{i+1}}$ as well as $\aast_{Y_k\rightarrow X'}$ holds. Thus
	\[gdep(X')\leq gdep(Y_k)-1\leq gdep(Y_1)-1\leq gdep(X)\] and \[ \gd(X)=\gd(Y_1)\leq...\leq\gd(Y_k)\leq \gd(X').\]
	Now if the non-isomorphic locus of $X\dashrightarrow X'$ on $X'$ contains a Gorenstein singular point, then either the non-isomorphic locus
	of $Y_i\dashrightarrow Y_{i+1}$ on $Y_{i+1}$ contains a Gorenstein singular point, or the non-isomorphic locus of $Y_k\rightarrow X'$
	on $X'$ contains a Gorenstein singular point. Hence at least one of the above inequalities is strict. Thus one has $\gd(X)<\gd(X')$.
\end{proof}

\begin{lem}\label{astn1}
	Assume that $\aast_{\Ss_{n-1}}$ holds, then $\aast_{\Ss_n}^{(1)}$ holds.
\end{lem}
\begin{proof}
	Let $Y\rightarrow X$ be a divisorial contraction to a point which belongs to $\Ss_n$. We know that $gdep(Y)=n$. Assume first that
	$Y\rightarrow X$ is a strict $w$-morphism over a point $P\in X$. Notice that $dep(Y)\geq dep(X)-1$ by Lemma \ref{chfp}.
	If $P$ is a non-Gorenstein point, then $\gd(Y)=\gd(X)$ by Lemma \ref{pitn}, hence $\aast_{Y\rightarrow X}$ holds. If $P$ is a
	Gorenstein point, then \[ \gd(X)=gdep(X)-dep(X)=gdep(Y)+1-dep(X)=\gd(Y)+dep(Y)-dep(X)+1.\]
	Moreover, since $dep(P\in X)=0$, $dep(Y)-dep(X)\geq 0$, hence \[\gd(X)>\gd(Y)=\gd(X)-(dep(Y)-dep(X)+1).\]
	Thus $\aast_{Y\rightarrow X}$ holds.\par
	In general by Corollary \ref{yua} there exists $Y_1\rightarrow X$, ..., $Y_k\rightarrow X$ such that $Y_k\rightarrow X$ is a strict
	$w$-morphism and one has $Y\ua{X}Y_1\ua{X}...\ua{X}Y_k$. By induction on $k$ we may assume that $\aast_{Y_i\rightarrow X}$ holds
	for all $i$ (notice that $gdep(Y_i)\leq gdep(Y)=n$ for all $i$ by Lemma \ref{pgd}). Now we have a diagram 
	\[ \vc{\xymatrix{ Z_1\ar@{-->}[r]\ar[d] & ... \ar@{-->}[r] & Z_k \ar[d] \\ Y \ar[rd] & & Y_1\ar[ld] \\ & X & }}.\]
	By Lemma \ref{laast} we know that $\gd(Y)\leq\gd(Y_1)$ and $\aast_{Z_i\dashrightarrow Z_{i+1}}$ as well as $\aast_{Z_k\rightarrow Y_1}$
	holds. Since $\aast_{Y_1\rightarrow X}$ holds, we know that $\gd(X)\geq\gd(Y_1)\geq\gd(Y)$ and
	\begin{align*}
	\gd(Y)+(dep(Y)-dep(X)+1)&=\gd(Z_1)+(dep(Z_1)-dep(X)+2)\\
	&\geq gdep(Z_1)-dep(X)+2\\
	&\geq gdep(Z_k)-dep(X)+2\\
	&\geq gdep(Y_1)-dep(X)+1\\
	&\geq \gd(Y_1)+(dep(Y_1)-dep(X)+1)\\
	&\geq \gd(X).
	\end{align*}
	Moreover if $Y\rightarrow X$ is a divisorial contraction to a Gorenstein point, then $Y_1\rightarrow X$ is also a divisorial contraction
	to a Gorenstein point. Hence $\gd(Y_1)<\gd(X)$ and we also have $\gd(Y)<\gd(X)$.
\end{proof}

\begin{proof}[Proof of Proposition \ref{ied}]
	We need to say that $\aast_{\Ss_n}$ holds for all $n$ and we will prove by induction on $n$. If $n=0$ then $\Ss_0$ consists only
	smooth blow-downs and smooth flops. One can see that $\aast_{\Ss_0}$ holds. In general assume that $\aast_{\Ss_{n-1}}$ holds.
	By Lemma \ref{astn1} we know that $\aast_{\Ss_n}^{(1)}$ holds. Hence it is enough to show that given a flip $X\dashrightarrow X'$
	or a divisorial contraction to a curve $X\rightarrow V$ such that $gdep(X)=n$, one has that $\aast_{X\dashrightarrow X'}$ or
	$\aast_{X\rightarrow V}$ holds.\par
	If $X\rightarrow V$ is a smooth blow-down, then $\gd(X)=\gd(V)$ so there is nothing to prove. In general we have a diagram as in
	Theorem \ref{chf}:
	\[\vc{\xymatrix{Y_1\ar[d] \ar@{-->}[r] & \cdots \ar@{-->}[r] & Y_k\ar[d]\\
		X\ar[rd] &  & X'\ar[ld] \\ & V & }}.\]
	By Lemma \ref{laast} we know that $\gd(X)\leq \gd(X')$ and $\aast_{Y_i\dashrightarrow Y_{i+1}}$ as well as $\aast_{Y_k\rightarrow X'}$
	holds. One has that
	\begin{align*}
	\gd(X)+(dep(X)-dep(X'))&=gdep(X)-dep(X')\\
	&=gdep(Y_1)-dep(X')+1\\
	&\geq gdep(Y_k)-dep(X')+1\\
	&\geq gdep(X')-dep(X')=\gd(X').
	\end{align*}
	Moreover, if $X\dashrightarrow X'$ is a flip, then either one of $Y_i\dashrightarrow Y_{i+1}$ is a flip, or $Y_k\rightarrow X'$ is a
	divisorial contraction to a curve by \cite[Remark 3.4]{ch}. This implies that either $gdep(Y_1)>gdep(Y_k)$ or $gdep(Y_k)\geq gdep(X')$.
	If $X\rightarrow V$ is a divisorial contraction to a curve then $Y_k\rightarrow X'$ is a divisorial contraction to a curve, hence
	one always has that $gdep(Y_k)\geq gdep(X')$. In conclusion, we have \[ \gd(X)\geq \gd(X')-(dep(X)-dep(X')-1).\]\par
	If $X\dashrightarrow X'$ is a flip then one can see that $\aast_{X\dashrightarrow X'}$ holds. Now assume that $X\rightarrow V$ is a
	divisorial contraction to a point. Then $X'\rightarrow V$ is a divisorial contraction to a point. We know that $\aast_{X'\rightarrow V}$
	holds since $\aast_{\Ss_n}^{(1)}$ holds and $gdep(X')\leq n$ by Lemma \ref{laast}. One can see that $\gd(X)\leq \gd(X')\leq \gd(V)$
	and \begin{align*}
	\gd(X)&\geq \gd(X')-(dep(X)-dep(X')-1)\\
	&\geq\gd(V)-(dep(X')-dep(V)+1)-(dep(X)-dep(X')-1)\\
	&=\gd(V)-(dep(X)-dep(V)).
	\end{align*}
	Thus $\aast_{X\rightarrow V}$ holds.
\end{proof}

\section{Comparing different feasible resolutions}\label{sflp}

In this section we describe the difference between two different feasible resolutions of a terminal $\Q$-factorial threefold.
The final result is the diagrams in Theorem \ref{thm1}. 
\subsection{General discussion}
\begin{lem}\label{lfp}
	Assume that $X$ is a terminal threefold and $Y\rightarrow X$, $Y_1\rightarrow X$ are two different strict $w$-morphisms over $P\in X$,
	such that $Y\ua{X}Y_1$. Let
	\[ \vc{\xymatrix{ Z_1\ar@{-->}[r]\ar[d] & ... \ar@{-->}[r] & Z_k \ar[d] \\ Y \ar[rd] & & Y_1\ar[ld] \\ & X & }}\]
	be the corresponding diagram, then one of the following holds:
	\begin{enumerate}[(1)]
	\item $Y\uan{X}Y_1\uan{X}Y$ and $k=1$.
	\item $k=2$ and $Z_1\dashrightarrow Z_2$ is a smooth flop.
	\item $k=2$, $Z_1\dashrightarrow Z_2$ is a singular flop, $P\in X$ is of type $cA/r$ and both $Y\rightarrow X$ and $Y_1\rightarrow X$
		are of type A1 in Table \ref{taA}.
	\end{enumerate}
\end{lem}
\begin{proof}
	Since $Y\rightarrow X$ and $Y_1\rightarrow X$ are both strict $w$-morphisms, we know that $gdep(Y)=gdep(Y_1)$. Hence
	$Z_i\dashrightarrow Z_{i+1}$ can not be a flip. Thus if $Y\uan{X}Y_1$, then $k=1$.\par
	Now assume that $Y\not\uan{X}Y_1$. According to the result in Section \ref{slk} one has that
	\begin{enumerate}[(i)]
	\item If $P$ is of type $cA/r$, then by Proposition \ref{ca} we know that $Y\rightarrow X$ and $Y_1\rightarrow X$ are both of type A1.
		Since $Y\not\uan{X}Y_1$, we know that $f^{\circ}_4=\eta_4=y$ is irreducible. Hence there are only one $K_{Z_1}$-trivial curve on
		$Z_1$ and so $k=2$.
	\item $P$ is not of type $cAx/r$ since one always has that $Y\uan{X}Y_1$ if $Y\rightarrow X$ and $Y_1\rightarrow X$
		are two different $w$-morphisms over $P$ by Proposition \ref{cax}.
	\item $P$ is not of type $cD$ by Proposition \ref{cd1}.
	\item $P$ is not of type $cD/3$ by Proposition \ref{cd3}.
	\item If $P$ is of type $cD/2$, then by Proposition \ref{cd4} and Proposition \ref{cd5} we know that $Y\rightarrow X$ is of type D17
		in Table \ref{taD3}. We know that $Z_1\rightarrow Y$ is a $w$-morphism over the origin of 
		\[ U_t=({x'}^2+y'+g'(z',u',t')=z'u'+{y'}^3+t'=0)\subset\A^5_{(x',y',z',u',t')}/\frac{1}{5}(3,1,1,2,3)\]
		and we take $(y_1,...,y_5)=(y',u',y'+z',x',t')$. One can see that there are two curves contained in $y_3=y_5=0$, namely
		$\Gamma_1=({x'}^2+g'(0,u',0)=y'=y'+z'=u'=0)$ and $\Gamma_2=({x'}^2+g'(-y',-{y'}^2,0)=u'+{y'}^2=y'+z'=t'=0)$.
		We know that ($\Theta_1$) holds. When computing the intersection number for $\Gamma_2$ one can take $\delta_4=4$ and $\delta_5=2$,
		Hence ($\Xi_-$) holds in this case. This implies that the proper transform $\Gamma_{2,Z_1}$ of $\Gamma_2$ on $Z_1$ is a
		$K_{Z_1}$-trivial curve and do not contained in $exc(Z_1\rightarrow Y)$. So $Z_i\dashrightarrow Z_{i+1}$ is a flip along
		$\Gamma_{2,Z_i}$ for some $i$ and hence $gdep(Y_1)<gdep(Y)$. This is a contradiction. Thus this case won't happen. 
	\item If $P$ is of type $cE$, then by Proposition \ref{ce1} and Proposition \ref{ce2} we know that $Y\rightarrow X$ is of type E14, E16
		or E18 in Table \ref{taE2}. In those cases $Y\rightarrow X$ are five-dimensional weighted blow-ups. Let $\Gamma\subset Y$
		be a curve contained in $exc(Y\rightarrow X)$ such that the proper transform $\Gamma_{Z_1}$ of $\Gamma$ on $Z_1$ is
		a possibly $K_{Z_1}$-trivial curve. From the proof of Proposition \ref{ce2}, one can see that: 
		\begin{enumerate}[({vi}-i)]
		\item $Y\rightarrow X$ is of type E14. In this case $Z_1\rightarrow Y$ is a $w$-morphism over the origin of 
			\[U_t=({x'}^2+{y'}^3+z'+g'(y',z',u',t')=p(x',y',z',u')+t'=0)\]\[\subset\A^5_{(x',y',z',u',t')}/\frac{1}{5}(3,2,1,1,4).\]
			We choose $(y_1,...,y_5)=(x',y',u',z',t')$ with $\delta_4=1$ and $\delta_5=2$ or $4$, or $\delta_4=2$ and $\delta_5=1$ or $4$.
			We also know that ($\Theta_4$) holds. If both $\delta_4$ and $\delta_5\neq 4$, then ($\Xi_-$) holds which implies that
			$Y\uan{X}Y_1$. This contradicts to our assumption. Hence $\delta_5=4$. Now $\Gamma=({x'}^2+{y'}^3=z'=u'=t'=0)$. 
		\item $Y\rightarrow X$ is of type E16. In this case $Z_1\rightarrow Y$ is a $w$-morphism over the origin of 
			\[U_t=({x'}^2+{y'}^3+p(z',u')+g'(y',z',u',t')=q(y',z',u')+t'=0)\]\[\subset\A^5_{(x',y',z',u',t')}/\frac{1}{4}(3,2,1,1,3).\]
			As in the proof of Proposition \ref{ce2} we choose $(y_1,...,y_5)=(y',z',u',x',t')$ with $\delta_4=4$ and $\delta_5=1$ or $2$.
			We know that ($\Theta_2$) holds. Hence if $\delta_5=1$ then ($\Xi_-$) holds and so $Y\uan{X}Y_1$. This leads to a contradiction.
			Thus $\delta_5=2$ and one has $\Gamma=({x'}^2+{y'}^3=z'=u'=t'=0)$.
		\item $Y\rightarrow X$ is of type E18. In this case $Z_1\rightarrow Y$ is a $w$-morphism over the origin of 
			\[U_t=({x'}^2+y'+g'(y',z',u',t')={y'}^2+p(y',z',u')+t=0)\]\[\subset\A^5_{(x',y',z',u',t')}/\frac{1}{7}(5,3,2,1,6).\]
			We choose $(y_1,...,y_5)=(y',z',u',x',t')$ with $\delta_4=4$ and $\delta_5=1$. We know that ($\Theta_1$) holds.
			If ${z'}^3\in p$ then we can choose $\delta_5=2$. Then since ($\Theta_1$) holds we know that ($\Xi_-$) holds and so $Y\uan{X}Y_1$.
			This contradicts to our assumption. Hence ${z'}^3\not\in p$. In this case $u$ divides $p$ and also $X$ has $cE_8$ singularities,
			so ${z'}^5\in g'$. One can see that $\Gamma=({x'}^2+{z'}^5=y'=u'=t'=0)$. 
		\end{enumerate}
		Now the origin of $U_t$ is a cyclic quotient point and the $w$-morphism over this point is a weighted blow-up with the weight
		$v_F(x',y',z',u',t')=\frac{1}{5}(3,2,6,1,4)$, $\frac{1}{4}(3,2,5,1,3)$ and $\frac{1}{7}(5,10,2,1,6)$ respectively.
		An easy computation shows that $\Gamma_{Z_1}$ do not pass through any singular point of $Z_1$, hence $Z_1\dashrightarrow Z_2$
		is a smooth flop. Since there is only one $K_{Z_1}$-trivial curve, we know that $k=2$. 
	\item If $P$ is of type $cE/2$, then by Proposition \ref{ce4} we know that $Y\rightarrow X$ is of type E22 or E23 in Table \ref{taE4}.
		Moreover, if $Y\rightarrow X$ is of type E23 then $Y_1\rightarrow X$ is of type E22. In this case we one interchange $Y$ and $Y_1$
		and hence always assume that $Y\rightarrow X$ is of type E22. As in the proof of Proposition \ref{ce4} we know that
		$Z_1\rightarrow Y$ is a $w$-morphism over the origin of $U_z\subset Y$, which is a $cD/3$ point defined by
		\[ ({x'}^2+{y'}^3+g'(y',z',u')=0)\subset\A^4_{(x',y',z',u')}/\frac{1}{3}(0,2,1,1).\]
		The only possible $K_{Z_1}$-trivial curve $\Gamma_{Z_1}$ is the lifting of the curve $\Gamma=({x'}^2+{y'}^3=z'=u'=0)$
		on $Z_1$. Moreover, $Z_1\rightarrow Y$ is defined by weighted blowing-up the weight
		\[v_F(x',y',z',u'+\lambda z')=\frac{1}{3}(b,c,1,4)\]
		for some $\lambda\in\Cc$, where $(b,c)=(3,2)$ or $(6,5)$. If $(b,c)=(6,5)$, then one can see that $(\Xi_-)$ holds by considering
		the function $y$, hence $Y\uan{X}Y_1$ which contradict to our assumption. Hence $(b,c)=(3,2)$. In this case
		an easy computation shows that $\Gamma_{Z_1}$ do not pass through any singular point of $Z_1$, so $Z_1\dashrightarrow Z_2$ is
		a smooth flop. Since there is only one $K_{Z_1}$-trivial curve, we know that $k=2$.
	\end{enumerate}
\end{proof}

We need to construct a factorization of the flop in Lemma \ref{lfp} (3). Assume that 
\[ X=(xy+f(z,u)=0)\subset\A^4_{(x,y,z,u)}/\frac{1}{r}(\beta,-\beta,1,0)\] is a $cA/r$ singularity with $f(z,u)=z^{rk}+g(z,u)$
and $Y\rightarrow X$, $Y_1\rightarrow X$ are two different strict $w$-morphisms with the following factorization
\[ \vc{\xymatrix{ Z_1\ar@{-->}[rr]\ar[d] &  & Z_2 \ar[d] \\ Y \ar[rd] & & Y_1\ar[ld] \\ & X & }}\]
such that $Z_1\dashrightarrow Z_2$ is a flop.
\begin{rk}\label{kg1}
	One always has that $k>1$ since if $k=1$ then there is only one $w$-morphism over $X$ by the classification Table \ref{taA}.
\end{rk}

\begin{lem}\label{flp}
	$Z_1\dashrightarrow Z_2$ is a flop over 
	\[V=(\bar{x}\bar{y}+\bar{u}f_1(\bar{z},\bar{u})=0)\subset\A^4_{(\bar{x},\bar{y},\bar{z},\bar{u})}/\frac{1}{r}(\beta,-\beta,1,0),\]
	where $f_1=f(zu^{\frac{1}{r}},u)/u^k$. Moreover, $Z_1=Bl_{(\bar{x},\bar{u})}V$ and $Z_2=Bl_{(\bar{y},\bar{u})}V$.
\end{lem}
\begin{proof}
	We may assume that $Y\rightarrow X$ is a weighted blow-up with the weight $w(x,y,z,u)=\frac{1}{r}(b,c,1,r)$ with $b>r$.
	The chart $U_u\subset Y$ is defined by \[(x_1y_1+f_1(z_1,u_1)=0)\subset\A^4_{(x_1,y_1,z_1,u_1)}/\frac{1}{r}(b,c,1,r)\]
	and the chart $U_x\subset Y$ is defined by \[(y_2+f_2(x_2,z_2,u_2)=0)\subset\A^4_{(x_2,y_2,z_2,u_2)}/\frac{1}{b}(b-r,c,1,r)\]
	for some $f_2$. As in the proof of Proposition \ref{ca} we know that $Z_1\rightarrow Y$ is a weighted blow-up over the origin
	of $U_x$ with the weight $w'(x_2,y_2,z_2,u_2)=\frac{1}{b}(b-r,rk,1,r)$. The flopping curve $\Gamma$ of $Z_1\rightarrow Z_2$ is the
	strict transform of the curve $\Gamma_Y\subset Y$ such that $\Gamma_Y|_{U_u}=(y_1=z_1=u_1=0)$ and $\Gamma_Y|_{U_x}=(x_2=y_2=z_2=0)$. 
	One can see that $\Gamma$ intersects $exc(Z_1\rightarrow Y)$ at the origin of $U'_u\subset Z_1$ which is defined by
	\[ (y'+f'(x',z',u')=0)\subset\A^4_{(x',y',z',u')}/\frac{1}{r}(b,0,1,-b).\] It is easy to see that $\Gamma$ is contained in
	$U'_u\cup U_u$, and on $U'_u$ we know that $\Gamma$ is defined by $(x'=z'=0)$.\par
	We have the following change of coordinates formula:
	\[ x=x_1u_1^{\frac{b}{r}},\quad y=y_1u_1^{\frac{c}{r}},\quad z=z_1u_1^{\frac{1}{r}},\quad u=u_1, \mbox{ and }\]
	\[ x=x_2^{\frac{b}{r}},\quad y=y_2x_2^{\frac{c}{r}},\quad z=z_2x_2^{\frac{1}{r}},\quad u=u_2x_2.\]
	And also \[ x_2=x'{u'}^{\frac{b-r}{b}},y_2=y'{u'}^{\frac{rk}{b}}, \quad z_2=z'{u'}^{\frac{1}{b}}\quad, u_2={u'}^{\frac{r}{b}}.\]
	One can see that
	\[ x={x'}^{\frac{b}{r}}{u'}^{\frac{b-r}{r}}, y=y'{x'}^{\frac{c}{r}}{u'}^{\frac{rk}{b}+\frac{c}{r}-\frac{c}{b}}
		=y'{x'}^{\frac{c}{r}}{u'}^{\frac{c}{r}+1}, \quad z=z'{x'}^{\frac{1}{r}}{u'}^{\frac{1}{r}}\mbox{ and }u=x'u'.\]
	It follows that \[ u_1=x'u',\quad z_1=z',\quad y_1=y'u'\mbox{ and }x_1={u'}^{-1}.\] If we choose an isomorphism
	\[ U'_u\cong (y_1+u'f'(x',z',u')=0)\subset\A^4_{(x',y_1,z',u')}/\frac{1}{r}(b,-b,1,-b),\]
	then $U_u\cup U'_u=Bl_{(x',u_1)}V$ where \[ V=(x'y_1+u_1f_1(z_1,u_1)=0)\subset\A^4_{(x',y_1,z_1,u_1)}/\frac{1}{r}(b,-b,1,0)\]
	by notice that
	\begin{align*}
	f'(x',z',u')&=f_2(x'{u'}^{\frac{b-r}{b}},z'{u'}^{\frac{1}{b}},{u'}^{\frac{r}{b}})/{u'}^{\frac{rk}{b}}\\
	&=f(z_2x_2^{\frac{1}{r}},u_2x_2)/x_2^k{u'}^{\frac{rk}{b}}\\
	&=f(z'{x'}^{\frac{1}{r}}{u'}^{\frac{1}{r}},x'u')/{x'}^k{u'}^k\\
	&=f(z_1u_1^{\frac{1}{r}},u_1)/{u_1}^k\\
	&=f_1(z_1,u_1).
	\end{align*}
	Now we can choose $(\bar{x},\bar{y},\bar{z},\bar{u})=(x',y_1,z_1,u_1)$.\par
	Finally we know that $Y_1\rightarrow X$ is a weighted blow-up with the weight $\frac{1}{r}(b-r,c+r,1,r)$ and $Z_2\rightarrow Y_1$
	is a $w$-morphism over the origin of $U_{1,y}\subset Y_1$. Hence the local picture of $Z_2\rightarrow V$ can be obtained by
	a similar computation, but interchange the role of $x$ and $y$. One has that $Z_2=Bl_{(\bar{y},\bar{u})}V$.
\end{proof}

\subsection{Explicit factorization of flops}

In this subsection we assume that \[V=(xy+uf(z,u)=0)\subset\A^4_{(x,y,z,u)}/\frac{1}{r}(\beta,-\beta,1,0)\]
with $f(z,u)=z^{rk}+g(z,u)$ for some $k>1$. Let $Z_1=Bl_{(x,u)}V$ and $Z_2=Bl_{(y,u)}V$, such that $Z_1\dashrightarrow Z_2$ is a 
$\Q$-factorial terminal flop. Let $w$ be the weight $w(z,u)=\frac{1}{r}(1,r)$ and let $m=w(f(z,u))$. Then $m\leq k$.
Let $U_{1,x}\subset Z_1$ be the chart \[ (y+u_1f_1(z,u_1)=0)\subset\A^4_{(x,y,z,u_1)}/\frac{1}{r}(\beta,-\beta,1,-\beta)\]
and $U_{1,u}\subset Z_1$ be the chart \[ (x_1y+f(z,u)=0)\subset\A^4_{(x_1,y,z,u)}/\frac{1}{r}(\beta,-\beta,1,0)\]
with the relations $u=u_1x$, $x=x_1u$ and $f_1=f(z,xu_1)$. Similarly, let $U_{2,y}\subset Z_2$ be the chart
\[ (x+u_2f_2(z,u_2)=0)\subset\A^4_{(x,y,z,u_2)}/\frac{1}{r}(\beta,-\beta,1,\beta)\]
and $U_{2,u}\subset Z_2$ be the chart \[ (xy_2+f(z,u)=0)\subset\A^4_{(x,y_2,z,u)}/\frac{1}{r}(\beta,-\beta,1,0)\]
with the relations $u=u_2y$, $y=y_2u$ and $f_2=f(z,yu_2)$.\par
\begin{lem}\label{v1}
	Let $\phi:V'\rightarrow V$ be a strict $w$-morphism. Then 
	\begin{enumerate}[(1)]
	\item The chart $U'_u\subset V'$ is $\Q$-factorial.
	\item If $m=k$ then $U'_z\subset V'$ is smooth. Otherwise $U'_z$ contains exactly one non-$\Q$-factorial $cA$ point
		which is defined by $xy+uf''(z,u)=0$ where $f''=f(z^{\frac{1}{r}},zu)/z^m$. One has that $w(f'')<m$.
	\item All other singular points on $V'$ are cyclic quotient points.
	\end{enumerate}
\end{lem}
\begin{proof}
	From Table \ref{ca} we know that $V'\rightarrow V$ is a weighted blow-up with the weight $\frac{1}{r}(b,c,1,r)$ with $b+c=r(m+1)$.
	The statement (3) follows from direct computations. One can compute that 
	\[ U'_u\cong (xy+f'(z,u)=0)\subset\A^4_{(x,y,z,u)}/\frac{1}{r}(\beta,-\beta,1,0)\] with $f'=f(zu^{\frac{1}{r}},u)/u^m$.
	By \cite[2.2.7]{ko} we know that $U'_u$ is $\Q$-factorial if and only if $f'$ is irreducible as an $\Z/r\Z$-invariant function.
	If $U'_u$ is not $\Q$-factorial, then $f'=f'_1f'_2$ for some non-unit $\Z/r\Z$-invariant functions $f'_1$ and $f'_2$. Write
	\[f'_1=\sum_{(i,j)\in\Z^2_{\geq0}}\xi_{i,j}z^{ri}u^j\mbox{ and }f'_2=\sum_{(i,j)\in\Z^2_{\geq0}}\zeta_{i,j}z^{ri}u^j.\] Let
	\[m_1=-\min_{\xi_{i,j}\neq 0}\{j-i\},\quad m_2=-\min_{\zeta_{i,j}\neq 0}\{j-i\}.\] Notice that if $f'=\sum_{(i,j)}\sigma_{i,j}z^{ri}u^j$,
	then $j+m-i\geq 0$ if $\sigma_{i,j}\neq 0$ because $f'=f(zu^{\frac{1}{r}},u)/u^m$. Hence we have the relation $m_1+m_2\leq m$. Now let
	\[f_1=\sum_{(i,j)}\xi_{ij}z^{ir}u^{j-i+m_1}\mbox{ and }f_2=\sum_{(i,j)}\zeta_{ij}z^{ir}u^{j-i+m_2}.\] Then $f_1$ and $f_2$ can be viewed as
	$\Z/r\Z$-invariant functions on $V$ such that $\phi\st f_i=u^{m_i}f'_i$ for $i=1$, $2$. This means that $f=u^{m-m_1-m_2}f_1f_2$ is not
	irreducible. Nevertheless, the chart $U_{1,u}\subset Z_1$ is defined by $x_1y+f(z,u)=0$ and has $\Q$-factorial singularities.
	This leads to a contradiction. Thus $U'_u$ is $\Q$-factorial.\par
	The chart $U'_z$ is defined by $xy+uf''(z,u)=0$ and one has that $f''$. When $m=k$, we know that $f''$ is a unit along $u=0$,
	hence $U'_z$ is smooth. When $m<k$ one can see that the origin of $U'_z$ is of the form
	$xy+uf'(z,u)$. This is a non-$\Q$-factorial $cA$ singularity.
\end{proof}

From the construction in \cite{me2}, for any strict $w$-morphism $V'\rightarrow V$ we have the following diagram
\[\vc{\xymatrix{ Z'_1\ar[d] & V''_1 \ar@{-->}[l]\ar[rd] & & V''_2\ar@{-->}[r] \ar[ld] & Z'_2\ar[d] \\
	Z_1\ar[rrd] & & V'\ar[d] & & Z_2\ar[lld] \\ & & V & & }}.\]
where $V''_i\rightarrow V'$ is a $\Q$-factorization and $V''_i\dashrightarrow Z'_i$ is a composition of flips for $i=1$, $2$. By Lemma \ref{v1},
we know that if $m=k$, then $V''_1=V'=V''_2$. Otherwise $V''_1\dashrightarrow V''_2$ is a flop over the singularity $xy+uf''(z,u)=0$.\par
First we discuss the factorization of $V''_1\dashrightarrow Z'_1$. If $m=k$, then $V''_1=V'$ is covered by four affine charts
$U'_x$, $U'_y$, $U'_z$ and $U'_u$. The origin of $U'_x$ and $U'_y$ are cyclic quotient points and all other singular points are contained 
in $U'_u$. When $m<k$ the chart $U'_z$ has a non-$\Q$-factorial point and there are two chart $U''_{1,x}$ and $U''_{1,u}$ over this point.
We fix the following notation for the latter discussion.\\
\begin{tabular}{rl}
$\bullet$ & $\bar{U}''_1=U'_x=(y''_1+u''_1f''_1(x''_1,z''_1,u''_1)=0)\subset\A^4_{(x''_1,y''_1,z''_1,u''_1)}/\frac{1}{b}(b-r,c,1,r)$ \\
	& with $x={x''_1}^{\frac{b}{r}}$, $y=y''_1{x''_1}^{\frac{c}{r}}$, $z=z''_1{x''_1}^{\frac{1}{r}}$, $u=u''_1x''_1$
	and $f''_1=f(z''_1{x''_1}^{\frac{1}{r}},u''_1x''_1)/{x''_1}^m$.\\
$\bullet$ & $\bar{U}''_2=U'_y=(x''_2+u''_2f''_2(y''_2,z''_2,u''_2)=0)\subset\A^4_{(x''_2,y''_2,z''_2,u''_2)}/\frac{1}{c}(b,c-r,1,r)$ \\
	& with $x=x''_2{y''_2}^{\frac{b}{r}}$, $y={y''_2}^{\frac{c}{r}}$, $z=z''_2{y''_2}^{\frac{1}{r}}$, $u=u''_2y''_2$
	and $f''_2=f(z''_2{y''_2}^{\frac{1}{r}},u''_2y''_2)/{y''_2}^m$.\\
$\bullet$ & $\bar{U}''_3=U'_u=(x''_3y''_3+f''_3(z''_3,u''_3)=0)\subset\A^4_{(x''_3,y''_3,z''_3,u''_3)}/\frac{1}{r}(b,c,1,r)$ \\
	& with $x={x''_3}{u''_3}^{\frac{b}{r}}$, $y=y''_3{u''_3}^{\frac{c}{r}}$, $z=z''_3{u''_3}^{\frac{1}{r}}$, $u=u''_3$
	and $f''_3=f(z''_3{u''_3}^{\frac{1}{r}},u''_3)/{u''_3}^m$.\\
\end{tabular}\\
If $m=k$, define $\bar{U}''_4=U'_z$. In this case this chart is smooth. When $m<k$ we define\\
\begin{tabular}{rl}
$\bullet$ & $\bar{U}''_4=U''_{1,x}=(y''_4+u''_4f''_4(x''_4,z''_4,u''_4)=0)\subset\A^4_{(x''_4,y''_4,z''_4,u''_4)}$ \\
	& with $x={x''_4}{z''_4}^{\frac{b}{r}}$, $y=y''_4{z''_4}^{\frac{c}{r}}$, $z={z''_4}^{\frac{1}{r}}$, $u=x''_4u''_4z''_4$
	and $f''_4=f({z''_4}^{\frac{1}{r}},x''_4u''_4z''_4)/{z''_4}^m$.\\
$\bullet$ & $\bar{U}''_5=U''_{1,u}=(x''_5y''_5+f''_5(z''_5,u''_5)=0)\subset\A^4_{(x''_5,y''_5,z''_5,u''_5)}$ \\
	& with $x={x''_5}u''_5{z''_5}^{\frac{b}{r}}$, $y=y''_5{z''_5}^{\frac{c}{r}}$, $z={z''_5}^{\frac{1}{r}}$, $u=u''_5z''_5$
	and $f''_5=f({z''_5}^{\frac{1}{r}},u''_5z''_5)/{z''_5}^m$.\\
\end{tabular}
\par
On the other hand, we will see later that it is enough to assume that $Z'_1\rightarrow Z_1$ is a divisorial contraction over the origin
of $U_{1,u}\subset Z_1$. This means that $Z'_1$ is covered by five affine charts $U_{1,x}$, $U'_{1,x}$, $U'_{1,y}$, $U'_{1,z}$ and $U'_{1,u}$
where the latter four charts correspond to the weighted blow-up $Z'_1\rightarrow Z_1$. Notice that $exc(Z'_1\rightarrow Z_1)$ and
$exc(V'\rightarrow V)$ are the same divisor since $V''_1$ and $Z'_1$ are isomorphic in codimension one. One can see that $Z'_1\rightarrow Z_1$
is a weighted blow-up with the weight $w'(x_1,y,z,u)=\frac{1}{r}(b-r,c,1,r)$. Again we use the following notation:\\
\begin{tabular}{rl}
$\bullet$ & $\bar{U}'_1=U'_{1,x}=(y'_1+f'_1(x'_1,z'_1,u'_1)=0)\subset\A^4_{(x'_1,y'_1,z'_1,u'_1)}/\frac{1}{b-r}(b-2r,c,1,r)$ \\
	& with $x={x'_1}^{\frac{b}{r}}u'_1$, $y=y'_1{x'_1}^{\frac{c}{r}}$, $z=z'_1{x'_1}^{\frac{1}{r}}$, $u=u'_1x'_1$
	and $f'_1=f(z'_1{x'_1}^{\frac{1}{r}},u'_1x'_1)/{x'_1}^m$.\\
$\bullet$ & $\bar{U}'_2=U'_{1,y}=(x'_2+f'_2(y'_2,z'_2,u'_2)=0)\subset\A^4_{(x'_2,y'_2,z'_2,u'_2)}/\frac{1}{c}(b-r,c-r,1,r)$ \\
	& with $x=x'_2{y'_2}^{\frac{b}{r}}u'_2$, $y={y'_2}^{\frac{c}{r}}$, $z=z'_2{y'_2}^{\frac{1}{r}}$, $u=u'_2y'_2$
	and $f'_2=f(z'_2{y'_2}^{\frac{1}{r}},u'_2y'_2)/{y'_2}^m$.\\
$\bullet$ & $\bar{U}'_3=U'_{1,u}=(x'_3y'_3+f'_3(z'_3,u'_3)=0)\subset\A^4_{(x'_3,y'_3,z'_3,u'_3)}/\frac{1}{r}(b-r,c,1,r)$ \\
	& with $x={x'_3}{u'_3}^{\frac{b}{r}}$, $y=y'_3{u'_3}^{\frac{c}{r}}$, $z=z'_3{u'_3}^{\frac{1}{r}}$, $u=u'_3$
	and $f'_3=f(z'_3{u'_3}^{\frac{1}{r}},u'_3)/{u'_3}^m$.\\
$\bullet$ & $\bar{U}'_4=U_{1,x}=(y'_4+u'_4f'_4(z'_4,u'_4)=0)\subset\A^4_{(x'_4,y'_4,z'_4,u'_4)}/\frac{1}{r}(\beta,-\beta,1,-\beta)$ \\
	& with $x=x'_4$, $y=y'_4$, $z=z'_4$, $u=u'_4x'_4$
	and $f'_4=f(z'_4,u'_4x'_4)$.\\
$\bullet$ & $\bar{U}'_5=U'_{1,z}=(x'_5y'_5+f'_5(z'_5,u'_5)=0)\subset\A^4_{(x'_5,y'_5,z'_5,u'_5)}$ \\
	& with $x={x'_5}{z'_5}^{\frac{b}{r}}u'_5$, $y=y'_5{z'_5}^{\frac{c}{r}}$, $z={z'_5}^{\frac{1}{r}}$, $u=u'_5z'_5$
	and $f'_5=f({z'_5}^{\frac{1}{r}},u'_5z'_5)/{z'_5}^m$.\\
\end{tabular}\\

\begin{lem}\label{flip}
	Assume that $Z'_1\rightarrow Z_1$ is a divisorial contraction over the origin of $U_{1,u}$, then
	\begin{enumerate}[(1)]
	\item $gdep(Z'_1)=gdep(V''_1)-1$.
	\item All the singular points on the non-isomorphic loci of $V''_1\dashrightarrow Z'_1$
		on both $V''_1$ and $Z'_1$ are cyclic quotient points.
	\item The flip $V''_1\dashrightarrow Z'_1$ is of type IA in the convention of \cite[Theorem 2.2]{km}. 
	\end{enumerate}  
\end{lem}
\begin{proof}
	It is easy to see that $\bar{U}'_i\cong\bar{U}''_i$ for $i=2$, $3$ and $i=5$ if $m<k$. Since $\bar{U}'_4$ is smooth,
	the only singular point which contained in the non-isomorphic locus of $V''_1\dashrightarrow Z'_1$ on $V''_1$ is the
	origin of $\bar{U}''_1$. This point is a cyclic quotient point of index $b$, so it has generalized depth $b-1$.\par
	On the other hand, singular points on the non-isomorphic locus of $V''_1\rightarrow Z'_1$ on $Z'_1$ are origins of $\bar{U}'_1$
	and $\bar{U}'_4$. They are cyclic quotient points of indices $b-r$ and $r$ respectively, One can see that
	\[ gdep(V''_1)-gdep(Z'_1)=b-1-(b-r-1+r-1)=1.\]\par
	Now we know that that the flipping curve contains only one singular point which is a cyclic quotient point. Also the general elephant
	of the flip is of $A$-type since it coming from the factorization of a flop over a $cA/r$ point. Thus the flip is of type IA by
	the classification \cite[Theorem 2.2]{km}.
\end{proof}

\begin{lem}\label{iafp}
	Assume that $\vc{\xymatrix@R=0.1cm{T\ar[rd] & & \ar[ld] T'\\ & V & }}$ is a flip of type IA. Assume that
	\[\vc{\xymatrix{ S_1\ar@{-->}[r]\ar[d] & ...\ar@{-->}[r] & S_k\ar[d] \\ T\ar@{-->}[rr] & & T'}}\]
	is the factorization in Theorem \ref{chf}. Then 
	\begin{enumerate}[(1)]
	\item If $S_1\dashrightarrow S_2$ is a flop, then it is a Gorenstein flop.
	\item If $S_1\dashrightarrow S_2$ is a flip, or $S_1\dashrightarrow S_2$ is a flop and $S_2\dashrightarrow S_3$ is a flip, then
		the flip is of type IA.
	\end{enumerate}
\end{lem}
\begin{proof}
	Let $C\subset T$ be the flipping curve. Since $T\dashrightarrow T'$ is a type IA flip, there is exactly one non-Gorenstein point which
	is contained in $C$, and this point is a $cA/r$ point. From the construction we know that $S_1\rightarrow T$ is a $w$-morphism
	over this $cA/r$ point. Also there exists a Du Val section $H\in|-K_T|$
	such that $C\not\subset H$. We know that $H_{S_1}\in|-K_{S_1}|$ by \cite[Lemma 2.7 (2)]{ch}. Hence all non-Gorenstein point of
	$S_1$ is contained in $H_{S_1}$. Now assume that $C_{S_1}$ contains a non-Gorenstein point, then $H_{S_1}$ intersects $C_{S_1}$
	non-trivially. Since $C_{S_1}\not\subset H_{S_1}$, we know that $H_{S_1}.C_{S_1}>0$, hence $C_{S_1}$ is a $K_{S_1}$-negative curve and
	so $S_1\dashrightarrow S_2$ is a flip. Thus if $S_1\dashrightarrow S_2$ is a flop, then it is a Gorenstein flop.\par
	If $S_1\dashrightarrow S_2$ is a flip, then $C_{S_1}$ passes through a non-Gorenstein point since $0>K_{S_1}.C_{S_1}>-1$ by
	\cite[Theorem 0]{b}. Since $C\subset T$ passes through exactly one non-Gorenstein point, $C_{S_1}$ passes through exactly one
	non-Gorenstein point and this point is contained in $E=exc(S_1\rightarrow T)$. Since $S_1\rightarrow T$ is a $w$-morphism over a $cA/r$
	point, an easy computation shows that $E$ contains only $cA/r$ singularities. Also we know that $H_{S_1}$ is a Du Val
	section which do not contain $C_{S_1}$. Thus $S_1\dashrightarrow S_2$ is also a flip of type IA by the classification
	\cite[Theorem 2.2]{km}.\par
	Now assume that $S_1\dashrightarrow S_2$ is a flop and $S_2\dashrightarrow S_3$ is a flip. Then the flipping curve $\Gamma$ of
	$S_2\dashrightarrow S_3$ is contained in $E_{S_2}$ since all $K_{S_2}$-negative curves over $V$ is contained in $E_{S_2}$. Since
	$S_1\dashrightarrow S_2$ is a flop, we know that $E_{S_2}$ contains only $cA/r$ singularities \cite[Theorem 2.4]{ko2} and
	$H_{S_2}\in|-K_{S_2}|$ is also a Du Val section. If $\Gamma\not\subset H_{S_2}$, then $S_2\dashrightarrow S_3$ is a flip of type IA by
	the classification \cite[Theorem 2.2]{km}.\par
	Thus we only need to prove that	$\Gamma\not\subset H_{S_2}$. Assume that $\Gamma\subset H_{S_2}\cap E_{S_2}$. Then since $H_{S_1}$ do not
	intersects $C_{S_1}$ (otherwise $C_{S_1}$ is a $K_{S_1}$-negative curve and then $S_1\dashrightarrow S_2$ is not a flop),
	we know that $\Gamma$ do not intersects the flopping curve $C'_{S_2}\subset S_2$. Let $B\subset E_{S_2}$ be a curve which intersects 
	$C'_{S_2}$ non-trivially. Then $\Gamma_{S_1}\equiv \lambda B_{S_1}$ for some $\lambda\in\Q$ since the both curves are contracted by
	$S_1\rightarrow T$. Hence for all divisor $D\subset S_2$ such that $D.C'_{S_2}=0$, we know that $D.\Gamma=\lambda D.B$.
	On the other hand, $S_1\dashrightarrow S_2$ is a $K_{S_1}+E$-anti-flip since $C_{S_1}$ is not contained in $E$.
	By Corollary \ref{ntl2} we know that $(K_{S_2}+E_{S_2}).B_{S_2}>(K_{S_1}+E).B_{S_1}$, hence $E_{S_2}.B_{S_2}>E.B_{S_1}$.
	Now we know that $\rho(S_2/V)=2$ and $B$ is not numerically equivalent to a multiple of $C'_{S_2}$, hence we may write
	$\Gamma\equiv \lambda B+\mu C'_{S_2}$ for some $\mu\in\Q$. Since
	\[ E_{S_2}.\Gamma=E_{S_1}.\Gamma_{S_1}=\lambda E.B_{S_1}<\lambda E_{S_2}.B_{S_2}\]
	and $E_{S_2}.C'_{S_2}<0$, we know that $\mu>0$. Hence $\Gamma$ is not contained in the boundary of the relative effective cone $NE(S_2/V)$.
	Thus $\Gamma$ can not be the flipping curve of $S_2\dashrightarrow S_3$. This leads to a contradiction.
\end{proof}

\begin{lem}\label{fld}
	Assume that $T\dashrightarrow T'$ is a three-dimensional terminal $\Q$-factorial flip satisfies conditions (1)-(3) of Lemma \ref{flip}.
	Then the factorization in Theorem \ref{chf} for $T\dashrightarrow T'$ is one of the following diagram: 
	\begin{enumerate}[(1)]
	\item \[\vc{\xymatrix{ S_1\ar@{-->}[r]\ar[d] & S_2\ar[d] \\ T\ar@{-->}[r] & T'}}\]
		where $S_1\dashrightarrow S_2$ is a flip which also satisfies conditions (1)-(3) of Lemma \ref{flip} and
		$S_2\rightarrow T'$ is a strict $w$-morphism.
	\item \[\vc{\xymatrix{ S_1\ar@{-->}[r]\ar[d] & S_2\ar@{-->}[r] & S_3\ar[d] \\ T\ar@{-->}[rr] & & T'}}\]
		where $S_1\dashrightarrow S_2$ is a smooth flop and $S_2\dashrightarrow S_3$ is a flip which also satisfies conditions (1)-(3)
		of Lemma \ref{flip} and $S_3\rightarrow T'$ is a strict $w$-morphism.
	\item \[\vc{\xymatrix{ S_1\ar@{-->}[r]\ar[d] & S_2\ar[d] \\ T\ar@{-->}[r] & T'}}\] where $S_1\dashrightarrow S_2$ is a smooth flop
		and $S_2=Bl_{C'}T'$ where $C'$ is a smooth curve contained in the smooth locus of $T'$.
	\end{enumerate}
\end{lem}
\begin{proof}
	We have the factorization \[\vc{\xymatrix{ S_1\ar@{-->}[r]\ar[d] & ...\ar@{-->}[r] & S_k\ar[d] \\ T\ar@{-->}[rr] & & T'}}\]
	such that $S_1\dashrightarrow S_2$ is a flip or a flop and $S_i\dashrightarrow S_{i+1}$ is a flip for all $2\leq i\leq k-1$.
	One has that \[ gdep(S_k)\leq gdep(S_1)=gdep(T)-1=gdep(T')\leq gdep(S_k)+1.\]
	If $gdep(S_k)=gdep(S_1)$ then $k=2$ and $S_1\dashrightarrow S_2$ is a flop. By \cite[Remark 3.4]{ch} we know that $S_2\rightarrow T'$
	is a divisorial contraction to a curve. Since singular points on the non-isomorphic locus of $T\dashrightarrow T'$ are all
	cyclic quotient points and there is no divisorial contraction to a curve which passes through a cyclic quotient point \cite[Theorem 5]{ka},
	we know that $S_2\rightarrow T'$ is a divisorial contraction to a curve $C'$ which contained in the smooth locus. Now we also has that
	$gdep(S_2)=gdep(T')$, hence $S_2$ is smooth over $T'$ and so $C'$ is also a smooth curve.\par
	Now assume that $gdep(S_k)<gdep(S_1)$. Then $gdep(S_k)=gdep(S_1)-1=gdep(T')-1$, hence either $k=2$ or $k=3$ and $S_1\dashrightarrow S_2$
	is a flop. Also $S_k\rightarrow T'$ is a $w$-morphism. Since singular points on the non-isomorphic locus of $T\dashrightarrow T'$ are all
	cyclic quotient points, singular points on the exceptional divisor of $S_k\rightarrow T'$ are all cyclic quotient points. 	
	Since flops do not change singularities \cite[Theorem 2.4]{ko2}, we know that the singular points on the non-isomorphic locus of
	$S_i\dashrightarrow S_{i+1}$ are all cyclic quotient points for $i=1$, ..., $k-1$. If $S_1\dashrightarrow S_2$ is a flop, then
	by Lemma \ref{iafp} we know that it is a Gorenstein flop. Since cyclic quotient points are not Gorenstein, we know that
	the flop is in fact a smooth flop. Now assume that $S_i\dashrightarrow S_{i+1}$ is a flip for $i=1$ or $2$. Then again by Lemma \ref{iafp}
	we know that it is a flip of type IA. Hence the conditions (1)-(3) of Lemma \ref{flip} are satisfied for this flip.
\end{proof}
\begin{cor}\label{fpdi}
	Assume that $T\dashrightarrow T'$ is a three-dimensional terminal $\Q$-factorial flip satisfies conditions (1)-(3) of Lemma \ref{flip}.
	Then we have a factorization
	\[ \vc{\xymatrix{\tl{T}\ar[dd]\ar@{-->}[r] & \bar{T'} \ar[d] \\ & \tl{T'} \ar[d] \\ T\ar@{-->}[r] & T' }}\]
	where $\tl{T}\rightarrow T$ and $\tl{T'}\rightarrow T$ are feasible resolutions, $\bar{T'}=Bl_{C'}\tl{T'}$
	where $C'\subset \tl{T'}$ is a smooth curve, and $\tl{T}\dashrightarrow \bar{T'}$ is a sequence of smooth flops.
\end{cor}
\begin{proof}
	For convenience we denote the diagram
	$\vc{\xymatrix@R=0.3cm{\tl{T}\ar[dd]\ar@{-->}[r] & \bar{T'} \ar[d] \\ & \tl{T'} \ar[d] \\ T\ar@{-->}[r] & T' }}$
	by $(A)_{T\dashrightarrow T'}$.\par
	We know that the factorization of $T\dashrightarrow T'$ is of the form (1)-(3) in Lemma \ref{fld}. Notice that since the only singular point
	on the non-isomorphic locus of $T\dashrightarrow T'$ are cyclic quotient points, the feasible resolutions $\tl{T}$ and $\tl{T'}$
	are uniquely determined. Hence $\tl{T}$ is also a feasible resolution of $S_1$. If the factorization of
	$T\dashrightarrow T'$ is of type (3) in Lemma \ref{fld}, then the non-isomorphic locus of $S_1\dashrightarrow T'$ contains
	no singular points. This means that $S_2\rightarrow T'$ induces a smooth blow-up $\bar{T'}\rightarrow \tl{T'}$ on $\tl{T'}$,
	and $S_1\dashrightarrow S_2$ induces a smooth flop $\tl{T}\dashrightarrow\bar{T'}$. Thus $(A)_{T\dashrightarrow T'}$ exists.\par
	Now if the factorization of $T\dashrightarrow T'$ is of type (1), then $\tl{T'}$ is a feasible resolution of $S_2$.
	Since $gdep(S_1)=gdep(T)-1$, we may induction on $gdep(T)$ and assume that $(A)_{S_1\dashrightarrow S_2}$ exists, and then
	$(A)_{T\dashrightarrow T'}$ can be induced by $(A)_{S_1\dashrightarrow S_2}$. If the factorization of $T\dashrightarrow T'$ is of type (2).
	Then again by induction we may assume that
	\[(A)_{S_2\dashrightarrow S_3}=\vc{\xymatrix{\tl{S_2}\ar[dd]\ar@{-->}[r] & \bar{S_3} \ar[d] \\ & \tl{S_3} \ar[d] \\ S_2\ar@{-->}[r] & S_3}}\]
	exists. Since $S_3\rightarrow T'$ is a strict $w$-morphism, we know that $\tl{S_3}=\tl{T'}$. Also since $S_1\dashrightarrow S_2$
	is a smooth flop, it induces a smooth flop $\tl{T}=\tl{S_1}\dashrightarrow \tl{S_2}$. If we let $\bar{T'}=\bar{S_3}$
	and let $\tl{T}\dashrightarrow \bar{T'}$ be the composition $\tl{T}\dashrightarrow \tl{S_2}\dashrightarrow \bar{T'}$,
	then we get the diagram $(A)_{T\dashrightarrow T'}$.
\end{proof}

\begin{defn}
	Let $W_1\dashrightarrow W_2$ be a birational map between smooth threefolds. We say that $W_1\dashrightarrow W_2$ is of type
	$\Omega_0$ if it is a composition of smooth flops. We say that $W_1\dashrightarrow W_2$ of type $\Omega_n$ if there exists the
	following diagram
	\[ \vc{\xymatrix{ \bar{W}_1 \ar[d]& \bar{W'}_1 \ar@{-->}[l]\ar@{-->}[r] & \bar{W'}_2\ar@{-->}[r] & \bar{W}_2\ar[d]\\  W_1 & & & W_2}}.\]
	such that \begin{enumerate}[(1)]
	\item $\bar{W}_i=Bl_{C_i}W_i$ for some smooth curve $C_i\subset W_i$ for $i=1$, $2$.
	\item $\bar{W'}_i\dashrightarrow \bar{W}_i$ is a composition of smooth flops for $i=1$, $2$.
	\item $\bar{W'_1}\dashrightarrow \bar{W'_2}$ has the following factorization
		\[ \bar{W'_1}=\bar{W}'_{1,1}\dashrightarrow \bar{W}'_{1,2}\dashrightarrow...\dashrightarrow\bar{W}'_{1,m}=\bar{W'_2}\]
		such that $\bar{W}'_{1,j}\dashrightarrow\bar{W}'_{1,j+1}$ is a birational map of type $\Omega_{m_j}$ for some $m_j<n$.
	\end{enumerate}
\end{defn}
\begin{eg}
	In the following diagrams, dashmaps stand for smooth flops and all other maps are blowing-down smooth curves.
	\begin{enumerate}[(1)]
	\item \[ \vc{\xymatrix@C=2cm@R=0.5cm{ \bar{W}_1 \ar[d] \ar@{-->}[r] & \bar{W}_2\ar[d]\\  W_1 & W_2}}.\]
	\item \[\vc{\xymatrix@C=2cm@R=0.5cm{ & \bar{W'_1}\ar[d] \ar@{-->}[r] & \bar{W'_2}\ar[d] 
			\\ \bar{W}_1\ar[d] \ar@{-->}[r] & W'_1 & W'_2 &  \bar{W}_2\ar@{-->}[l]\ar[d] \\ W_1 & & & W_2}}.\]
	\item \[\vc{\xymatrix@C=1cm@R=0.5cm{  & \bar{W'_1}\ar[d] & \bar{W''_1}\ar[rd] \ar@{-->}[l] & & \bar{W''_2}\ar@{-->}[r]\ar[ld] &
		\bar{W'_2}\ar[d] \\ \bar{W}_1\ar[d] \ar@{-->}[r] & W'_1 & & W'' & &
		 W'_2 & \bar{W}_2\ar@{-->}[l]\ar[d] \\ W_1 & & & & & & W_2}}.\]
	\end{enumerate}
	In diagram (1) $W_1\dashrightarrow W_2$ is of type $\Omega_1$. In both diagram (2) and (3) $W_1\dashrightarrow W_2$ is of type
	$\Omega_2$.
\end{eg}
\begin{defn}
	Let $W\dashrightarrow W'$ be a birational map between smooth threefolds. We say that $W\dashrightarrow W'$ has
	an $\Omega$-type factorization if there exists birational maps between smooth threefolds
	\[ W=W_1\dashrightarrow W_2\dashrightarrow...\dashrightarrow W_k=W'\]
	such that $W_i\dashrightarrow{W_{i+1}}$ is of type $\Omega_{n_i}$ for some $n_i\in\Z_{\geq0}$.
\end{defn}

\begin{pro}\label{twof}
	Assume that $X$ is a $\Q$-factorial terminal threefold and $W\rightarrow X$, $W'\rightarrow X$ are two different feasible resolutions,
	then the birational map $W\dashrightarrow W'$ has an $\Omega$-type factorization.
\end{pro}
\begin{proof}
	First notice that if $gdep(X)=1$, then $X$ has either a cyclic quotient point of index 2, or a $cA_1$ point defined by $xy+z^2+u^n$ for
	$n=2$ or $3$ by \cite[Corollary 3.4]{mo}. In those cases there is exactly one feasible resolution (which is obtained by blowing-up the
	singular point). Hence one may assume that $gdep(X)>1$. Let $Y\rightarrow X$ (resp. $Y'\rightarrow X$) be the strict $w$-morphism
	which is the first factor of $W\rightarrow X$ (resp. $W'\rightarrow X$). If $Y=Y'$, then $W$ and $W'$ are two different feasible
	resolutions of $Y$. In this case the statement can be proved by induction on $gdep(X)$. Thus we may assume that $Y\neq Y'$.\par
	Since both $Y$ and $Y'$ are strict $w$-morphisms, by Corollary \ref{yua} there exists a sequence of
	strict $w$-morphisms $Y_1=Y\rightarrow X$, $Y_2\rightarrow X$, ..., $Y_k=Y'\rightarrow X$ such that $Y_i\ua{X}Y_{i+1}$ for $i=1$, ..., $k-1$.
	For each $2\leq i\leq k-1$ let $W_i\rightarrow Y_i$ be a feasible resolution, then $W_i$ is also a feasible resolution of $X$ and it is
	enough to prove that our statement holds for $W_i$ and $W_{i+1}$, for all $i=1$, ..., $k-1$. Thus we may assume that $Y\ua{X}Y'$.\par
	We have a diagram \[ \vc{\xymatrix{ Z_1\ar@{-->}[r]\ar[d] & ... \ar@{-->}[r] & Z_k \ar[d] \\ Y \ar[rd] & & Y'\ar[ld] \\ & X & }}\]
	Lemma \ref{lfp} says that there are three possibilities. If $k=1$, then $W$ and $W'$ are two different feasible resolutions of $Z_1$.
	Since $gdep(Z_1)=gdep(X)-2$, again by induction on $gdep(X)$ we know that $W\dashrightarrow W'$ has an $\Omega$-type factorization.
	Assume that $k=2$ and $Z_1\dashrightarrow Z_2$ is a smooth flop, then it induces a smooth flop $\tl{Z}_1\dashrightarrow \tl{Z}_2$
	where $\tl{Z}_i\rightarrow Z_i$ is a feasible resolution of $Z_i$ for $i=1$, $2$. Also we know that $W$ and $\tl{Z}_1$ are two
	feasible resolutions of $Y$, and $W'$ and $\tl{Z}_2$ are two feasible resolutions of $Y'$. Again by induction on $gdep(X)$ we know that
	both $W\dashrightarrow \tl{Z}_1$ and $\tl{Z}_2\dashrightarrow W'$ have $\Omega$-type factorizations, hence so does
	$W\dashrightarrow W'$.\par
	Finally assume that we are in the case Lemma \ref{lfp} (3), namely $X$ has a $cA/r$ singularity and $Z_1\dashrightarrow Z_2$ is a
	singular flop. By Lemma \ref{flp} we know that
	$Z_1\dashrightarrow Z_2$ is a flop over \[V=(xy+uf(z,u)=0)\subset\A^4_{(x,y,z,u)}/\frac{1}{r}(\beta,-\beta,1,0).\]
	We have the factorization of the flop $Z_1\dashrightarrow Z_2$
	\[\vc{\xymatrix{ Z'_1\ar[d] & V''_1 \ar@{-->}[l]\ar[rd] & & V''_2\ar@{-->}[r] \ar[ld] & Z'_2\ar[d] \\
		Z_1\ar[rrd] & & V'\ar[d] & & Z_2\ar[lld] \\ & & V & & }}.\]
	We use the notation at the beginning of this subsection. Assume that $m>1$, then we can choose $V'\rightarrow V$ to be the
	weighted blow-up with the weight $w'(x,y,z,u)=\frac{1}{r}(\beta+r,mr-\beta,1,r)$. In this case $Z'_i\rightarrow Z_i$ is a divisorial
	contraction over the origin of $U_{i,u}$ for $i=1$, $2$ since both $w'(x_1)$ and $w'(y_2)>0$. When $m=1$, let $V'\rightarrow V$ be the
	weighted blow-up with the weight $w'(x,y,z,u)=\frac{1}{r}(r+\beta,r-\beta,1,r)$. Then $Z'_1\rightarrow Z_1$ is a divisorial contraction
	over the origin of $U_{1,u}$ but $Z'_2\rightarrow Z_2$ is a divisorial contraction over the origin of $U_{2,y}$. Since $m=1$ and $k>1$
	by Remark \ref{kg1}, we know that $f(z,u)=\lambda u+z^{rk}$ for some unit $\lambda$. If we let $\bar{u}=\lambda u+z^{rk}$, then we can write
	the defining equation of $V$ as $xy+\bar{u}\bar{f}(z,\bar{u})$ where $\bar{f}=\frac{1}{\lambda}(\bar{u}-z^{rk})$. One has that
	$Z_2\cong Bl_{(x,\bar{u})}V$ and under this notation one also has that $Z'_2\rightarrow Z_2$ is a divisorial contraction over the 
	origin of $U_{2,\bar{u}}$. In conclusion, Corollary \ref{fpdi} holds for both $V''_1\dashrightarrow Z'_1$ and $V''_2\dashrightarrow Z'_2$.\par
	Let $\tl{Z'_i}\rightarrow Z'_i$ be a feasible resolution. If $V''_1=V''_2$ then one can see that $\tl{Z'_1}\dashrightarrow \tl{Z'_2}$
	is of type $\Omega_1$. Assume that $V''_1\dashrightarrow V''_2$ is a flop. Notice that $Z'_i\rightarrow Z_i$ is a
	$w$-morphism since $a(E,Z_i)=a(E,V)$ where $E=exc(Z'_i\rightarrow Z_i)=exc(V'\rightarrow V)$. One has that
	\[gdep(V''_i)=gdep(Z'_i)+1=gdep(Z_i)\] where the second equality follows from Corollary \ref{cawmo}. Now we know that
	$V''_1\dashrightarrow V''_2$ is a flop over $V'$ with $gdep(V')<gdep(V)$. By induction on $gdep(V)$ we may assume that
	$\tl{V''_1}\dashrightarrow \tl{V''_2}$ has an $\Omega$-type factorization,
	where $\tl{V''_i}\rightarrow V''_i$ is a feasible resolution corresponding to the diagram in Corollary \ref{fpdi}.
	One can see that $\tl{Z'_2}\dashrightarrow \tl{Z'_2}$ can be connected by a diagram of the form $\Omega_n$ for some $n\in\N$.
	Finally we know that $W$ and $\tl{Z'_1}$ (resp. $W'$ and $\tl{Z'_2}$) are feasible resolutions of $Y$ (resp. $Y'$).
	Again by induction on $gdep(X)$ we may assume that $W\dashrightarrow \tl{Z'_1}$ and $W'\dashrightarrow \tl{Z'_2}$ have $\Omega$-type
	factorizations. Hence $W\dashrightarrow W'$ had an $\Omega$-type factorization.
\end{proof}
\begin{rk}
	Assume that $X=(xy+z^m+u^k=0)\subset\A^4$ is a $cA$ singularity. Then there exists feasible resolutions $W$ and $W'$ such that
	the birational map $W\dashrightarrow W'$ is connected by $\Omega_n$ for $n=\ru{\frac{m}{k-m}}$.
\end{rk}
\section{Minimal resolutions of threefolds}\label{spmr}

In this section we prove our main theorems. First we recall some definitions which are defined in the introduction section.
\begin{defn}
	Let $X$ be a projective variety. We say that a resolution of singularities $W\rightarrow X$ is a P-minimal resolution
	if for any smooth model $W'\rightarrow X$ one has that $\rho(W)\leq \rho(W')$.
\end{defn}
\begin{defn}
	Let $W\dashrightarrow W'$ be a birational map between smooth varieties. We say that this birational map is a P-desingularization
	of a flop if there exists a flop $X\dashrightarrow X'$ such that $Y\rightarrow X$ and $Y'\rightarrow X'$ are P-minimal resolutions.
\end{defn}

\begin{pro}\label{mcf}
	Assume that $X$ is a threefold. Then $W\rightarrow X$ is a P-minimal resolution if and only if $W$ is a feasible resolution
	of a terminalization of $X$. In particular, if $X$ is a terminal and $\Q$-factorial threefold, then P-minimal resolutions
	of $X$ coincide with feasible resolutions. 
\end{pro}
\begin{proof}
	Let $W\rightarrow X$ be a resolution of singularities. We can run $K_W$-MMP over $X$. The minimal model $X_W\rightarrow X$ is a
	terminalization of $X$. Since terminalizations are connected by flops \cite[Theorem 1]{ka3}, we know that $\rho(X_W/X)$ is independent
	of the choice of $W$. Thus $W$ is a P-minimal resolution if and only if $\rho(W/X_W)$ is minimal. It is equivalent to that
	$W$ is a feasible resolution of $X_W$ by Corollary \ref{feam}.
\end{proof}
\begin{proof}[Proof of Theorem \ref{thm1}]
	Let $X$ be a threefold and $W\rightarrow X$, $W'\rightarrow X$ be two P-minimal resolutions. By Proposition \ref{mcf}
	we know that $W$ (resp. $W'$) is a feasible resolution of a terminalization $X_W\rightarrow X$ (resp. $X_{W'}\rightarrow X$).
	If $X_W\not\cong X_{W'}$, then $X_W$ and $X_{W'}$ are connected by flops \cite[Theorem 1]{ka3}, hence $W\dashrightarrow W'$
	is connected by P-desingularizations of terminal $\Q$-factorial flops.\par
	Now assume that $X_W=X_{W'}$, then $W$ and $W'$ are two different feasible resolutions of $X_W$. The first two paragraphs in the proof
	of Proposition \ref{twof} and Lemma \ref{fld} implies that $W\dashrightarrow W'$ can be also connected by
	P-desingularizations of terminal $\Q$-factorial flops. Moreover, Proposition \ref{twof} says that those P-desingularizations of flops
	can be factorize into compositions of diagrams of the form $\Omega_i$. This finishes the proof. 
\end{proof}

\begin{rk}\label{fgdp}
	Assume that $X$ is a terminal $\Q$-factorial threefold and $W\rightarrow X$, $W'\rightarrow X$ are two different P-minimal resolutions.
	We know that $W$ and $W'$ can be connected by P-desingularizations of flops. Let $W_i\dashrightarrow W_{i+1}$ be a P-desingularization
	of a flop $X_i\dashrightarrow X_{i+1}$ which appears in the factorization of $W\dashrightarrow W'$. Then from the construction
	we know that $gdep(X_i)<gdep(X)$.
\end{rk}
Now we compare an arbitrary resolution of singularities to a P-minimal resolution.

\begin{defn}
	Let $W\dashrightarrow X$ be a birational map where $W$ is a smooth threefold and $X$ is a terminal threefold. We say that the
	birational map has a \emph{bfw-factorization} if $W\dashrightarrow X$ can be factorized into a composition of smooth blow-downs,
	P-desingularizations of flops and strict $w$-morphisms.
\end{defn}
\begin{rk}\label{fea}
	If $X_2\rightarrow X_1$ is a strict $w$-morphism and $X_1\dashrightarrow X$ is a smooth blow-down or a P-desingularization of a flop,
	then on $X_1$ the indeterminacy locus of $X_1\dashrightarrow X$ is disjoint to the indeterminacy locus of $X_1\dashrightarrow X_2$
	since the former one lying on the smooth locus of $X_1$ and the latter one is a singular point. Hence there exists
	$X_2\dashrightarrow X'_1\rightarrow X$ where $X_2\dashrightarrow X'_1$ is a smooth blow-down or a P-desingularization of a flop, and
	$X'_1\rightarrow X$ is a strict $w$-morphism. In other word, $W\dashrightarrow X$ has a bfw-factorization if and only if
	there exists a birational map $W\dashrightarrow \bar{X}$ which is a composition of smooth blow-downs and P-desingularization of flops,
	where $\bar{X}$ is a feasible resolution of $X$. 
\end{rk}

\begin{pro}
	Assume that a birational map $W\dashrightarrow X$ has a bfw-factorization where $W$ is a smooth threefold and $X$ is a terminal threefold.
	\begin{enumerate}[(1)]
	\item If $X\dashrightarrow X'$ is a flop, then there is a birational map $W\dashrightarrow X'$ which has a bfw-factorization.
	\item If $Y\rightarrow X$ is a strict $w$-morphism, then there exists a birational map $W\dashrightarrow Y$ which also has
		a bfw-factorization.
	\item If $X\dashrightarrow X'$ is a flip or a divisorial contraction, then the induces birational map $W\dashrightarrow X'$
		has a bfw-factorization.
	\end{enumerate}
\end{pro}
\begin{proof}
	Assume first that $X\dashrightarrow X'$ is a flop. By Remark \ref{fea} we know that there exists a bfw-map $W\dashrightarrow \bar{X}$,
	where $\bar{X}$ is a feasible resolution of $X$. Let $\bar{X}'\rightarrow X'$ be a feasible resolution of $X'$, then
	$\bar{X}'\rightarrow X'$ is a composition of strict $w$-morphisms and the induced birational map $\bar{X}\dashrightarrow \bar{X}'$
	is a P-desingularization of the flop $X\dashrightarrow X'$. It follows that the composition
	\[W\dashrightarrow \bar{X}\dashrightarrow \bar{X}'\rightarrow X'\] is a bfw-map. This proves (1).\par
	We will prove (2) and (3) by induction $gdep(X)$. If $gdep(X)=0$, then $X$ is smooth. In this case there is no strict $w$-morphism
	$Y\rightarrow X$ or flip $X\dashrightarrow X'$. Assume that $X\rightarrow X'$ is a divisorial contraction. If $X'$ is smooth then
	it is a smooth blow-down by \cite[Theorem 3.3, Corollary 3.4]{mo} and if $X'$ is singular then $X\rightarrow X'$ should be a strict
	$w$-morphism since in this case $X'$ is terminal $\Q$-factorial and $X$ is a P-minimal resolution of $X'$.
	Now we may assume that $gdep(X)>0$ and the statements (2) and (3) hold for threefolds with generalized depth less then $gdep(X)$.\par
	Let \[ W=X_k\dashrightarrow X_{k-1}\dashrightarrow ...\dashrightarrow X_1\dashrightarrow X_0=X\] be a sequence of birational maps
	so that $X_{i+1}\dashrightarrow X_i$ is a smooth blow-down, a P-desingularization of a flop, or a strict $w$-morphism for all
	$1\leq i\leq k-1$. By Remark \ref{fea} we can assume that $X_1\rightarrow X$ is a strict $w$-morphism. Now given a strict $w$-morphism
	$Y\rightarrow X$. If $Y\cong X_1$ then there is nothing to prove. Otherwise by Corollary \ref{yua} there exists a sequence of strict
	$w$-morphisms $Y_2\rightarrow X$, ..., $Y_{m-1}\rightarrow X$ such that \[X_1=Y_1\ua{X}Y_2\ua{X}...\ua{X}Y_{m-1}\ua{X}Y_m=Y.\]
	For each $1\leq i\leq m-1$ one has the factorization
	\[ \vc{\xymatrix{ Z_{i,1}\ar@{-->}[r] \ar[d] & ... \ar@{-->}[r] & Z_{i,k_i}\ar[d] \\ Y_i \ar[rd] & & Y_{i+1}\ar[ld] \\ & X & }}\]
	such that $Z_{i,1}\dashrightarrow Z_{i,k_i}$ is a composition of flops, $Z_{i,1}\rightarrow Y_i$ is a strict $w$-morphism and
	\[ gdep(Z_{i,k_i})=gdep(Z_{i,1})<gdep(Y_i)<gdep(X).\]
	By the induction hypothesis we know that there if there exists a bfw-map $W\dashrightarrow Y_i$, then there exists a sfw-map
	$W\dashrightarrow Y_{i+1}$. Now one can prove the statement (2) by induction on $m$.\par
	Assume that $X\dashrightarrow X'$ is a flip. Then we have a factorization
	\[ \vc{\xymatrix{ Y_1\ar@{-->}[r] \ar[d] & ... \ar@{-->}[r] & Y_k\ar[d] \\ X & & X'}}\] as in Theorem \ref{chf}.
	Since $Y_1\rightarrow X$ is a strict $w$-morphism by Corollary \ref{cawmo}, there exists a bfw-map $W\dashrightarrow Y_1$. Since
	\[ gdep(Y_k)\leq ...\leq gdep(Y_1)<gdep(X),\] the induction hypothesis implies that
	there exists a bfw-map $W\dashrightarrow X'$.\par
	Finally assume that $X\rightarrow X'$ is a divisorial contraction. If it is a smooth blow-down or a strict $w$-morphism then there is
	nothing to prove. Otherwise there exists a diagram
	\[ \vc{\xymatrix{ Y_1\ar@{-->}[r] \ar[d] & ... \ar@{-->}[r] & Y_k\ar[d] \\ X \ar[rd] & & Z\ar[ld] \\ & X' &}}\] 
	such that $Y_1\rightarrow X$ is a strict $w$-morphism and $Y_i\dashrightarrow Y_{i+1}$ is a flip or a flop for all $1\leq i\leq k-1$.
	One has that \[ gdep(Y_k)\leq ...\leq gdep(Y_1)<gdep(X),\] hence there exists a bfw-map $W\dashrightarrow Z$. If $X\rightarrow X'$ is
	a divisorial contraction to a curve then $Y_k\rightarrow Z$ is a divisorial contraction to a curve as in Theorem \ref{chf}. In this case
	we also have $gdep(Z)\leq gdep(Y_k)<gdep(X)$, so there exists a bfw-map $W\dashrightarrow X'$. If $X\rightarrow X'$ is a
	divisorial contraction to a point, then the discrepancy of $Z\rightarrow X'$ is less than the discrepancy of $X\rightarrow X'$ unless
	$X\rightarrow X'$ is a $w$-morphism. Also when $X\rightarrow X'$ is a $w$-morphism we know that $gdep(Z)<gdep(X)$ by Lemma \ref{pgd}.
	Thus we can prove statement (3) by induction on the generalized depth and the discrepancy of $X$ over $X'$.
\end{proof}

One can easily see the following corollary:

\begin{cor}\label{bfwm}
	Assume that $W$ is a smooth threefold and $W\dashrightarrow X$ is a birational map which is a composition of steps of MMP. Then
	this birational map be factorized into a composition of smooth blow-downs, P-desingularizations of flops and strict $w$-morphisms.
\end{cor}

\begin{proof}[Proof of Theorem \ref{thm2}]
	By Corollary \ref{bfwm} and Remark \ref{fea} we know that there is exists a feasible resolution $\tl{X}_W\rightarrow X_W$ such that
	$W\dashrightarrow \tl{X}_W$ is a composition of smooth blow-downs and P-desingularizations of flops, where $X_W$ is a minimal model
	of $W$ over $X$. By Proposition \ref{mcf} we know that $\tl{X}_W$ is also a P-minimal resolution of $X$, hence the birational map
	$\tl{X}_W\dashrightarrow \tl{X}$ is connected by P-desingularizations of flops. Thus the composition
	$W\dashrightarrow \tl{X}_W\dashrightarrow \tl{X}$ is connected by smooth blow-downs and P-desingularizations of flops.
\end{proof}

\begin{proof}[Proof of Corollary \ref{cor}]
	Let \[W=\tl{X}_k\dashrightarrow ... \dashrightarrow \tl{X}_1\dashrightarrow \tl{X}_0=\tl{X}\] be a sequence of smooth blow-downs
	and P-desingularization of flops as in Theorem \ref{thm2}. We only need to show that if $\tl{X}_{i+1}\dashrightarrow \tl{X}_i$
	is a P-desingularization of a flop $X_{i+1}\dashrightarrow X_i$, then $b_j(\tl{X}_{i+1})=b_j(\tl{X}_i)$ for all $j=0$, ..., $6$.\par
	By \cite[Lemma 2.12]{me3} we know that $b_j(X_{i+1})=b_j(X_j)$ for all $j$. Since $X_i$ and $X_{i+1}$ have the same analytic
	singularities \cite[Theorem 2.4]{ko2}, there exists a feasible resolution $\tl{X}'_{i+1}\rightarrow X_{i+1}$ such that
	$b_j(\tl{X}_i)=b_j(\tl{X}'_{i+1})$ for all $j$. Now $\tl{X}'_{i+1}$ and $\tl{X}_{i+1}$ are two different P-minimal resolutions of
	$X_{i+1}$, so they can be connected by P-desingularizations of flops with smaller generalized depth by Remark \ref{fgdp}. 
	By induction on the generalized depth one can see that $b_j(\tl{X}'_{i+1})=b_j(\tl{X}_{i+1})$. Hence
	$b_j(\tl{X}_{i+1})=b_j(\tl{X}_i)$ for all $j=0$, ..., $6$.
\end{proof}

\section{Further discussion}\label{sfd}

This section is dedicated to exploring minimal resolutions for singularities in higher dimensions and the potential applications of our main theorems.

\subsection{Higher dimensional minimal resolutions}
In three dimensions, $P$-minimal resolutions appear to be a viable generalization of minimal resolutions for surfaces. However, in higher dimensions, $P$-minimal resolutions are not good enough. For example, let $X\dashrightarrow X'$ be a smooth flip
(eg. a standard flip \cite[Section 11.3]{hu}). Then $X$ and $X'$ are both $P$-minimal resolutions of the underlying space, but $X'$ is
better than $X$. It is reasonable to assume that $X'$ is a minimal resolution, while $X$ is not.
Inspired by Corollary \ref{cor}, we define a new kind of minimal resolution:
\begin{defn}
	Let $X$ be a projective variety over complex numbers. We say that a resolution of singularities $W\rightarrow X$ is a $B$-minimal resolution
	if for any smooth model $W'\rightarrow X$ one has that $b_i(W)\leq b_i(W')$ for all $0\leq i\leq 2\dim X$.
\end{defn}
As stated in Corollary \ref{cor}, B-minimal resolutions coincide with P-minimal resolutions in dimension three. Our main theorems says that
B-minimal resolutions of threefolds satisfy certain nice properties.  It is logical to anticipate that B-minimal resolutions of higher dimensional varieties share similar properties. 
\begin{conj}
	For any projective variety $X$ over the complex number, one has that
	\begin{enumerate}[(1)]
	\item B-minimal resolutions of $X$ exists.
	\item Two different B-minimal resolutions are connected by desingularizations of $\Q$-factorial terminal flops
	\item If $\tl{X}\rightarrow X$ is a B-minimal resolution and $W\rightarrow X$ is an arbitrary resolution of singularities, then
		$W\dashrightarrow \tl{X}$ can be connected by smooth blow-downs, smooth flips and desingularizations of $\Q$-factorial terminal flops. 
	\end{enumerate}
\end{conj}
\subsection{The strong factorization theorem}
Let \[\X_3=\{\mbox{ smooth threefolds }\}/\sim,\] where $W_1\sim W_2$ if $W_1\dashrightarrow W_2$ is connected by P-desingularization
of $Q$-factorial terminal flops. For $\eta_1$, $\eta_2\in\X_3$ we say that $\eta_1>\eta_2$ if there exists $W_1$ and $W_2$ so that
$\eta_i=[W_i]$ and $W_1\rightarrow W_2$ is a smooth blow-down. Then Theorem \ref{thm1} and Theorem \ref{thm2} implies that
\begin{cor}
	Given a threefold $X$. Let \[\X_{3,X}=\se{\mbox{ }[W]\in\X_3}{\mbox{ There exists a birational morphism }W\rightarrow X}.\]
	Then $\X_{3,X}$ has a unique minimal element.
\end{cor}
In other words, if we consider the resolution of singularities inside $\X_3$, then there is a unique minimal resolution, which behaves
like the minimal resolution of a surface.\par
As a consequence, inside the space $\X_3$ the strong factorization theorem holds:
\begin{thm}[Strong factorization theorem for $\X_3$]
	Assume that $W_1$ and $W_2$ are smooth threefolds which are birational to each other. Then there exists a smooth threefold $\bar{W}$
	such that inside $\X_3$ one has $[\bar{W}]\geq[W_i]$ for $i=1$, $2$.
\end{thm}
\begin{proof}
	Let $W_1\leftarrow \bar{W}\rightarrow W_2$ be a common resolution. Then $[\bar{W}]\in\X_{3,W_i}$ for $i=1$, $2$.
	Since the minimal element of $\X_{3,W_i}$ is $[W_i]$ itself, one has that $[\bar{W}]\geq[W_i]$ for $i=1$, $2$.
\end{proof}

\subsection{Essential valuations}
One can characterize a surface singularity by the information of exceptional curves on the minimal resolution. One may ask, does similar
phenomenon happen for higher dimensional singularities? Since for higher dimensional singularities there is no unique minimal resolution,
what we really want to study is the following object.
\begin{defn}
	Let $X$ be a projective threefold over the complex number. We say that a divisorial valuation $v_E$ over $X$ is an
	almost essential valuation if for any P-minimal resolution $\tl{X}\rightarrow X$ one has that $\cen_{\tl{X}}E$ is an irreducible
	component of the exceptional locus of $\tl{X}\rightarrow X$.
\end{defn}
This name coming from ``essential valuation'' in the theories of arc spaces.
\begin{defn}
	Let $X$ be a variety. We say that a divisorial valuation $v_E$ over $X$ is an essential valuation if for any resolution of singularities
	$W\rightarrow X$ one has that $\cen_{W}E$ is an irreducible component of the exceptional locus of $W\rightarrow X$.
\end{defn}
From the definition one can see that essential valuations are almost essential, but almost essential valuation may not be essential.
\begin{eg}
	Let $X=(xy+z^2+u^{2n+1})\subset\A^4$ for some $n>2$. There is a exactly one $w$-morphism $X_1\rightarrow X$ over the singular point,
	which is obtained by blowing-up the origin. There is only one singular point on $X_1$ which is defined by $xy+z^2+u^{2(n-1)+1}$.
	Keep blowing-up the singular point $n-1$ more times we get a resolution of singularities $\tl{X}\rightarrow X$. From the construction
	we know that $\tl{X}$ is a unique feasible resolution of $X$. Since $X$ is terminal and $\Q$-factorial, $\tl{X}$ is the unique P-minimal
	resolution of $X$. Hence almost essential valuations of $X$ are those divisorial valuations appear on $\tl{X}$.
	One can compute that $exc(\tl{X}\rightarrow X)=E_1\cup...\cup E_n$, such that $v_{E_i}(x,y,z,u)=(i,i,i,1)$.
	On the other hand, by \cite[Lemma 15]{jk} we know that essential valuations of $X$ are $v_{E_1}$ and $v_{E_2}$. Hence $v_{E_3}$, ...,
	$v_{E_n}$ are almost essential valuations which are not essential.\par
	Notice that the set of essential valuations does not really characterize the singularity since it is independent of $n$. The set of
	almost essential valuations carries more information of the singularity.
\end{eg}

\subsection{Derived categories}
Let $X$ be a smooth variety. The bounded derived category of coherent sheaves of $X$, denoted by $D^b(X)$, is an interesting subject of investigation. One possible method to study $D^b(X)$ is to construct a semi-orthogonal decomposition of $D^b(X)$ (refer to \cite{ku} for more information). Orlov \cite{o} proved that a smooth blow-down yields a semi-orthogonal decomposition. In particular, if $X$ is a smooth surface, then the $K_X$-MMP is a series of smooth blow-downs, thereby resulting in a semi-orthogonal decomposition of $D^b(X)$.\par
Now assume that $X$ is a smooth threefold and let
\[ X=X_0\dashrightarrow X_1\dashrightarrow ...\dashrightarrow X_k\] be the process of $K_X$-minimal model program.
According to Corollary \ref{bfwm} and Remark \ref{fea}, $\tl{X}_i\dashrightarrow \tl{X}_{i+1}$ can be factored into a composition of smooth blow-downs and P-desingularization of flops, where $\tl{X}_i\rightarrow X_i$ is a P-minimal resolution of $X_i$. If every P-desingularization of flops that appears in the factorization is a smooth flop, then the sequence induces a semi-orthogonal decomposition of $D^b(X)$ since smooth flops are derived equivalent \cite{br}.
\begin{eg}
	Let $X_1\dashrightarrow X_2$ be the flip which is a quotient of an Atiyah flop by an $\Z/2\Z$-action \cite[Example 2.7]{km2}.
	Then $X_2$ is smooth and $X_1$ has a $\frac{1}{2}(1,1,1)$ singular point. Let $X\rightarrow X_1$ be the smooth resolution obtained by
	blowing-up the singular point. Then $X\rightarrow X_1\dashrightarrow X_2$ is a sequence of MMP.\par
	The factorization of the flip is exactly the diagram (3) in Lemma \ref{fld}, namely the following diagram
	\[\vc{\xymatrix{ X\ar[d]\ar@{-->}[r] & X'\ar[d] \\ X_1\ar@{-->}[r] & X_2 }}\] where $X\dashrightarrow X'$ is a smooth flop
	and $X'\rightarrow X_2$ is a blow-down of a smooth curve. We know that there exists an equivalence of category
	$\Phi: D^b(X')\rightarrow D^b(X)$ and a semi-orthogonal decomposition $D^b(X')=\ang{D_{-1},D^b(X_2)}$,
	Hence $D^b(X)=\ang{\Phi(D_{-1}),\Phi(D^b(X_2))}$ is a semi-orthogonal decomposition.
\end{eg}

In general, a P-desingularization of a flop $\tl{X}_i\dashrightarrow \tl{X}_{i+1}$ may not be derived equivalent since $\tl{X}_i$ and
$\tl{X}_{i+1}$ may not be isomorphic in codimension one. Nevertheless, due to the symmetry between $\tl{X}_i$ and $\tl{X}_{i+1}$,
one might expect that a semi-orthogonal decomposition on $D^b(\tl{X}_i)$ will result in a semi-orthogonal decomposition on $D^b(\tl{X}_{i+1})$. We still have faith that our approach will prove effective for all smooth threefolds.


\begin{thebibliography}{aaa}
\bibitem{b} X. Benveniste, \emph{Sur le cone des 1-cycles effectifs en dimension 3}, Math. Ann. \textbf{272} (1985), 257-265.
\bibitem{br} T. Bridgeland, \emph{Flops and derived categories}, Invent. Math. \textbf{147} (2002), 613-632.
\bibitem{me3} H-K Chen, \emph{Betti numbers in the three-dimensional minimal model program}, Bull. London Math. Soc. \textbf{51} (2019), 563-576.
\bibitem{me} H-K Chen, \emph{On the Nash problem for terminal threefolds of type $cA/r$}, preprint, arXiv:1907.06326.
\bibitem{me2} H-K Chen, \emph{On the factorization of three-dimensional terminal flops}, preprint, arXiv:2004.12711.
\bibitem{c} J. A. Chen, \emph{Explicit resolution of three dimensional terminal singularities}, Minimal Models and Extremal Rays (Kyoto, 2011),
	Adv. Stud. Pure Math. \textbf{70} (2016), 323-360.
\bibitem{c2} J. A. Chen, \emph{Birational maps of 3-folds}, Taiwanese J. Math. \textbf{19} (2015), 1619-1642.
\bibitem{ch} J. A. Chen and C. D. Hacon, \emph{Factoring 3-Fold Flips and Divisorial Contractions to Curves}, J. Reine Angew. Math., \text{657}
	(2011), 173-197.
\bibitem{h1} T. Hayakawa, \emph{Blowing ups of 3-dimensional terminal singularities}, Publ. Res. Inst. Math. Sci. \textbf{35} (1999), 515-570.
\bibitem{h2} T. Hayakawa, \emph{Blowing ups of 3-dimensional terminal singularities, II}, Publ. Res. Inst. Math. Sci. \textbf{36} (2000), 423-456.
\bibitem{h3} T. Hayakawa, \emph{Divisorial contractions to 3-dimensional terminal singularities with discrepancy one},
	J. Math. Soc. Japan \textbf{57} (2005), 651-668.
\bibitem{h4} T. Hayakawa, \emph{Divisorial contractions to cD points}, preprint.
\bibitem{h5} T. Hayakawa, \emph{Divisorial contractions to cE points}, preprint.
\bibitem{hu} D. Huybrechts, \emph{Fourier-Mukai transforms in algebraic geometry}, Oxford. University Press, USA (2006). 
\bibitem{jk} J. M. Johnson, J. Koll\'{a}r, \emph{Arc spaces of $cA$-type singularities}, J. Sing. \textbf{7} (2013), 238-252.
\bibitem{k1} M. Kawakita, \emph{Divisorial contractions in dimension three which contract divisors to smooth points},
	Invent. Math. \textbf{145} (2001), 105-119.
\bibitem{k2} M. Kawakita,  \emph{Three-fold divisorial contractions to singularities of higher indices}, Duke Math. J. \textbf{130} (2005),
	57-126.
\bibitem{k3} M. Kawakita, \emph{Supplement to classification of three-fold divisorial contractions}, Nagoya Math. J. \textbf{208} (2012), 67-73.
\bibitem{ka2} Y. Kawamata, \emph{Crepant blowing-up of 3-dimensional canonical singularities and its application to degenerations of surfaces},
	Ann. Math. \textbf{127} (1988), 93-163.
\bibitem{ka} Y. Kawamata, \emph{Divisorial contractions to 3-dimensional terminal quotient singularities}, Higher-dimensional complex varieties
	(Trento, 1994) de Gruyter, Berlin, 1996, 241-246. 
\bibitem{ka3} Y. Kawamata, \emph{Flops connect minimal models}, Publ. Res. Inst. Math. Sci. \textbf{44} (2008), 419-423.
\bibitem{ko2} J. Koll\'{a}r, \emph{Flops}, Nagoya Math. J. \textbf{113} (1989), 15-36.
\bibitem{ko} J. Koll\'{a}r, \emph{Flips, flops, minimal models, etc}, Surveys in Diff. Geo. \textbf{1} (1991), 113-199.
\bibitem{k} \emph{Flips and abundance for algebraic threefolds - A summer seminar at the University of Utah (Salt Lake City, 1991)},
	J. Koll\'{a}r (ed.), Ast\'{e}risque 211 (1992).
\bibitem{km} J. Koll\'{a}r, S. Mori, \emph{Classification of three-dimensional terminal flips}, J. Amer. Math. Soc. \textbf{5} (1992), 533-703.
\bibitem{km2} J. Koll\'{a}r, S. Mori, \emph{Birational geometry of algebraic varieties},
	Cambridge Tracts in Mathematics 134, Cambridge Univ. Press, 1998.
\bibitem{ku} A. Kuznetsov, \emph{Semiorthogonal decompositions in algebraic geometry}, in Proceedings of the International Congress of
	 Mathematicians-Seoul 2014, 635-660.
\bibitem{mo} S. Mori, \emph{Threefolds whose canonical bundles are not numerically effective}, Ann. Math. \textbf{116} (1982), 113-176.
\bibitem{mo3} S. Mori, \emph{On 3-dimensional terminal singularities}, Nagoya Math. J. \textbf{98} (1985), 43-66.
\bibitem{mo2} S. Mori, \emph{Flip theorem and the existence of minimal models for 3-folds}, J. Amer. Math. Soc. \textbf{1} (1988), 117-253.
\bibitem{o} D. Orlov, \emph{Projective bundles, monoidal transformations, and derived categories of coherent sheaves}, Russian Acad. Sci.
	Izv. Math.\textbf{41} (1993), 133-141.
852–862; English transl., Russian Acad. Sci. Izv. Math., 41 (1993): 133–141.
\bibitem{re2} M. Reid, \emph{Minimal model of canonical 3-folds}, Adv. Stud. pure Math. \textbf{1} (1983), 131-180.
\bibitem{re} M. Reid, \emph{Young person’s guide to canonical singularities}, Proc. Symp. pure Math. \textbf{46} (1987), 345-414.
\bibitem{y} Y. Yamamoto, \emph{Divisorial contractions to cDV points with discrepancy $>1$}, Kyoto J. Math. \textbf{58} (2018), 529-567.
\end{thebibliography}
\end{document}